\documentclass[11pt]{article}

\usepackage[dvips]{epsfig}
\usepackage{graphicx}
\usepackage{amssymb,graphics,amsmath,amsthm,amsopn,amstext,amsfonts, algorithm,algorithmic}
\usepackage{color}

\newcommand{\gap}{\vspace{0.1in}}

\newcommand{\wt}{\widetilde}

\newcommand{\wh}{\widehat}

\newcommand{\supp}{ \mbox{supp} }

\newcommand{\argmax}{\operatornamewithlimits{\arg\max}}
\newcommand{\argmin}{\operatornamewithlimits{\arg\min}}
\newcommand{\Argmin}{\mbox{Argmin}}

\newcommand{\Ical}{\mathcal I}
\newcommand{\Jcal}{\mathcal J}
\newcommand{\Lcal}{\mathcal L}
\newcommand{\Pcal}{\mathcal P}

\newcommand{\Scal}{\mathcal S}

\newcommand{\sgn}{{\rm sgn}}

\newcommand{\mycut}[1]{{}}

\newcommand{\tblue}[1]{\textcolor{black}{#1}}  %%% update from blue to black, March 2020.

\newtheorem{theorem}{Theorem}[section] %[section]
\newtheorem{lemma}{Lemma}[section] %[section]
\newtheorem{corollary}{Corollary}[section] %[section]
\newtheorem{proposition}{Proposition}[section] %[section]
\newtheorem{definition}{Definition}[section] %[section]
\newtheorem{remark}{Remark}[section]%[section]
\newtheorem{example}{Example}[section] %[section]

\begin{document}

\title{Exact Support and Vector Recovery of Constrained Sparse Vectors via Constrained Matching Pursuit}
%

%\author{}

\author{Jinglai Shen\footnote{Department of Mathematics and Statistics, University of Maryland Baltimore County, Baltimore, MD 21250, U.S.A. Emails: shenj@umbc.edu and smousav1@umbc.edu.} \ \ \ and \ \ \ Seyedahmad Mousavi  }

\maketitle

\begin{abstract}
  Matching pursuit, especially its orthogonal version (OMP) and variations, is a greedy algorithm widely used in signal processing, compressed sensing, and sparse modeling. Inspired by constrained sparse signal recovery, this paper proposes a constrained matching pursuit algorithm and develops conditions for exact support and vector recovery on constraint sets via this algorithm. We show that exact recovery via constrained matching pursuit not only depends on a measurement matrix but also critically relies on a constraint set. We thus identify an important class of constraint sets, called coordinate projection admissible set, or simply CP admissible sets;
%
%  which include  many others arising from various applications
%  This class of sets includes the Euclidean space, the nonnegative orthant, and many others arising from %various applications;
%
  analytic and geometric properties of these sets are established.
  We study exact vector recovery on convex, CP admissible cones for a fixed support. We provide sufficient exact recovery conditions for a general support as well as necessary and sufficient recovery conditions when a support has small size.
  As a byproduct, we construct a nontrivial counterexample to a renowned necessary condition of exact recovery via the OMP
  %
  %given by Foucart, Rauhut, and Tropp,
  %
   for a support of size three. Moreover, using the properties of convex CP admissible sets and \tblue{convex} optimization techniques, we establish sufficient conditions for uniform exact recovery on convex CP admissible sets  in terms of the restricted isometry-like constant and the restricted orthogonality-like constant.
  %
  % Moreover, by making use of cone properties and conic hull structure of \textcolor{blue}{convex} CP %admissible sets and \tblue{convex} optimization techniques, we establish sufficient conditions for %uniform exact recovery on \textcolor{blue}{convex} CP admissible sets  in terms of the restricted %isometry-like constant and the restricted orthogonality-like constant. %%% for $K$-sparse vectors in CP %admissible sets.
\end{abstract}
%

%-----------------------------------------------------------------------------------
%
\section{Introduction}

Sparse models and representations find  broad applications in numerous fields of contemporary interest \cite{ElderK_book2012}, e.g., signal and image processing, high dimensional statistics, compressed sensing, and machine learning. Effective recovery of sparse signals from a few measurements poses  challenging theoretical and numerical questions. A variety of sparse recovery schemes have been proposed and studied, including the basis pursuit and its extensions, greedy algorithms, and thresholding based algorithms \cite{FoucartRauhut_book2013, ShenMousavi_SIOPT18}.

Originally introduced in signal processing and statistics, matching pursuit \cite{MallatZ_TSP93}, especially the orthogonal matching pursuit (OMP) \cite{PRKrishnaprasad_Asilomar93}, is a greedy algorithm widely used in  sparse signal recovery. At each step, the OMP uses the current target vector to select an additional ``best'' index via coordinate-wise optimization and adds it to the target support, and then updates the target vector over the new support via  optimal fitting of a measurement vector. The deterministic and statistical performance of the OMP  has been extensively studied in the literature \cite{CaiWang_TIT2011, Tropp_ITI04, TroppGilbert_TIT2007, Zhang_JMLR09, Zhang_TIT2011}. In particular, the exact support and vector recovery via the OMP has been characterized in term of the restricted isometry constant with extensions to noisy measurements \cite{MoS_TIT12, WenZWTM_SIT16}. Besides, many variations of the OMP have been developed to improve the recovery accuracy, effectiveness, and robustness under noise and errors; representative examples of these variations include compressive sampling matching pursuit \cite{NeedellTropp_ACHA09, NeedellV_JSTSP2010}, simultaneous OMP \cite{TroppGS_SP06}, stagewise OMP \cite{DonohoTDS_TIT12}, subspace pursuit \cite{DaiMilenk_TIT2009}, generalized OMP \cite{WangKS_TSP2012}, grouped OMP \cite{SwirszczAL_NIPS09}, and multipath matching pursuit \cite{KwonWS_TIT14}, just to name a few; see \cite{FoucartRauhut_book2013} and the references therein for more details.
%
%Other matching pursuit algorithms, e.g. compressive sampling matching pursuit ... (its orthogonal version and %other variations)
%

Sparse signals arising from diverse applications are subject to constraints, for example, the nonnegative constraint in nonnegative factorization in signal and image processing \cite{RrucksteinEZ_TIT2008}, the polyhedral constraint in index tracking problems in finance \cite{XuLX_OMS16},  and the monotone or shape constraint in order statistics and shape constrained estimation \cite{ShenLebair_Auto15, ShenWang_SICON11}. Hence, constrained sparse recovery has attracted increasing interest from different areas, such as machine learning and sparse optimization \cite{BahmaniRaj_JMLR2013,  BeckE_SIOPT2013, BeckH_MOR15, FoucartKoslicki_ISPL14, IDP_ISP17, Locatello_NIPS2017, MouShen_COCV18, WangXTang_TSP11}. While matching pursuit, particularly the OMP and its variations or extensions, has been extensively studied on $\mathbb R^N$, its constrained version has received much less attention, especially the exact recovery on a general constraint set; exceptions include \cite{RrucksteinEZ_TIT2008} where the uniqueness of the OMP recovery on the nonnegative orthant is considered. Inspired by the constrained sparse recovery, this paper proposes a constrained matching pursuit algorithm for a general constraint set, and develops conditions for exact support and vector recovery on constraint sets via this algorithm.
Similar to the OMP,  the constrained matching pursuit algorithm  selects a new optimal index by solving a constrained coordinate-wise optimization problem at each step, and then updates its target vector over the updated support by solving another constrained optimization problem for the best fitting of a measurement vector.
We show that exact recovery via the constrained matching pursuit not only depends on a measurement matrix but also critically relies on a constraint set. This motivates us to introduce an important class of constraint sets, called coordinate projection admissible sets, or simply CP admissible sets. This class of sets includes
%
%the Euclidean space $\mathbb R^N$, the nonnegative orthant $\mathbb R^N_+$,
%
the Cartesian product of arbitrary copies of $\mathbb R$, $\mathbb R_+$, and $\mathbb R_-$, and many others arising from applications. We establish analytic and geometric properties of these sets to be used for exact recovery analysis.  We then study exact vector recovery on convex, CP admissible cones for a fixed support. When  a fixed support has the size of two and three, we develop necessary and sufficient recovery conditions; when the support size is large, we  provide sufficient exact recovery conditions. As a byproduct of our results, we construct a nontrivial counterexample to the necessary conditions of exact vector recovery via the OMP given by Foucart, Rauhut, and Tropp, when the size of a given support is three
(cf. Section~\ref{subsect:counterexample_S=3}).
Moreover, we  establish sufficient conditions for uniform exact recovery on general convex CP admissible sets  in terms of the restricted isometry-like constant and the restricted orthogonality-like constant, by leveraging the properties of convex CP admissible sets and \tblue{convex} optimization techniques.

%
%by leveraging cone properties and conic hull structure of \tblue{convex} CP admissible sets, the positive %homogeneity of the aforementioned constants, and \tblue{convex} optimization techniques. Its extensions are %also discussed.
%

The rest of the paper is organized as follows. Section~\ref{sect:Constrained_MP} presents the constrained matching pursuit algorithm and discusses underlying optimization problems in this algorithm. Section~\ref{sect:exact_supp_recover} studies basic properties of  exact support recovery via constrained matching pursuit. In Section~\ref{sect:CP_admissible_set}, the CP admissible sets are introduced, and their properties are established. Section~\ref{sect:exact_vector_recovery} is concerned with the exact vector recovery of convex, CP admissible cones for a fixed support. \tblue{In Section~\ref{sect:suff_cond_exact_recovery}, sufficient  conditions  for uniform exact recovery on general convex, CP admissible sets are derived with conclusions made in Section~\ref{sect:conclusion}.}

{\it Notation}.
Let $A$ be an $m\times N$ real matrix.
For any index set $\Scal \subseteq \{1, \ldots, N\}$, let $|\Scal|$ denote the cardinality of $\Scal$, $\mathcal S^c$ denote the complement of $\Scal$, and $A_{\bullet\Scal}$ be the matrix formed by the columns of $A$ indexed by elements of $\Scal$. We write the $i$th column of $A$ as $A_{\bullet i}$ instead of $A_{\bullet \{i\} } $.
Further, $\mathbb R^N_+$ and $\mathbb R^N_{++}$ denote the nonnegative and positive orthants of $\mathbb R^N$ respectively, and $\mathbf e_j$ denotes the $j$th column of the $N\times N$ identity matrix.
For  $a \in \mathbb R$, let $a_+:=\max(a, 0) \ge 0$ and $a_-:=\max(-a, 0) \ge 0$. %% Hence, $a=a_+ - a_-$.
 For a given $x \in \mathbb R^N$, $\supp(x)$ denotes the support of $x$, i.e., $\supp(x)=\{ i \, | \, x_i \ne 0 \}$. The standard inner product on $\mathbb R^n$ is denoted by $\langle \cdot, \cdot \rangle$.
When a minimization problem has multiple solutions, $x \in \Argmin$ denotes an arbitrary optimal solution; if there is a unique optimal solution, then we use $x =\argmin$.
Let $\mbox{cone}(S)$ denote the conic hull of a set $S$ in $\mathbb R^N$, i.e., the collection of nonnegative combinations of finitely many vectors in $S$. We always assume that a cone in $\mathbb R^n$ contains the zero vector.
For two sets $A$ and $B$, $A \subseteq B$ means that $A$ is a subset of $B$ and $A$ possibly equals to $B$, while $A \subset B$ means that $A$ is a proper subset of $B$.
For $K \in \mathbb N$, let $\Sigma_K$ be the set of all vectors $x \in \mathbb R^N$  satisfying $|\supp(x)|\le K$. For  $u, v \in \mathbb R^n$, $u \perp v$ stands for the orthogonality of $u$ and $v$, i.e., $u^T v =0$.

%-------------------------------------------------------------------------------
%
\section{Constrained Matching Pursuit: Algorithm and Preliminary Results} \label{sect:Constrained_MP}

%
%Matching pursuit is a greedy algorithm to recover a sparse vector (from a linear model), and it finds a wide %range of applications in compressed sensing, signal processing, and machine learning. The classical matching %pursuit aims at solving the $\ell_0$-norm minimization problem: $\min \|x \|_0$ subject to $Ax = y$ for a %given matrix $A$ and a vector $y \in \mathbb R^m$ which is assumed to satisfy $y \in R(A)$ so that the %minimization problem is feasible. This yields the so-called orthogonal matching pursuit (OMP); see references %[literature review]. The OMP has been extended to the linear model subject to measurement noise or errors so %that $y \notin R(A)$.
%%
%Motivated by constrained sparse signal recovery arising from various applied fields (e.g., finance and %operations research),
%

Consider the following constrained sparse recovery problem:
\begin{equation} \label{eqn:constrained_L0}
   \min_{x \in \mathbb R^N} \, \| x \|_0 \qquad \mbox{subject to } \quad A x =y, \quad x \in \Pcal,
\end{equation}
where $\| x\|_0:=|\supp(x)|$,  $A \in \mathbb R^{m\times N}$ with $N> m$, $y \in \mathbb R^m$, and $\Pcal$ is a nonempty constraint set  in $\mathbb R^N$. We make the following assumptions on the matrix $A$, the vector $y$, and the set $\Pcal$:
\tblue{
\begin{itemize}
  \item [$\bf A.1$] Each column of $A$ is nonzero, and $y\in A \Pcal:=\{ A x \, | \, x \in \Pcal\}$;
  \item [$\bf A.2$] $\Pcal$ is a (possibly nonconvex) closed set containing the zero vector, i.e., $0 \in \Pcal$.
\end{itemize}
}
\mycut{
Throughout this paper, we also assume that
%
% $\Pcal$ contains the zero vector,
%
% there is no measurement error so that
%
 $y$ is in the range of $A$,
%
%%$0\in \Pcal$,
%
and each column of $A$ is nonzero, i.e., $\| A_{\bullet i}\|_2>0$ for each $i=1, \ldots, N$.
}
To solve the problem (\ref{eqn:constrained_L0}), we introduce the constrained matching pursuit scheme given below.

\begin{algorithm}%[t]
%=====================================================================================
\caption{Constrained Matching Pursuit}
\begin{algorithmic}[1]
\label{algo:constrained_MP}
%=====================================================================================
%
%\STATE Let $d_k$ be a descent direction, $s=1$, $\beta=0.5$, and $c_1=0.1$;
%
\STATE Input: $A \in \mathbb R^{m\times N}$, $y \in \mathbb R^m$, $\Pcal\subseteq \mathbb R^N$, and a stopping criteria

\STATE Initialize: $k=0$, $x^0=0$, and $\Jcal_0=\emptyset$

\WHILE{the stopping criteria is not met}

  \STATE $g^*_j =  \min_{t \in \mathbb R} \| y - A (x^k + t \, \mathbf e_j  ) \|^2_2 \ \ \mbox{ subject to }  \ x^k + t \, \mathbf e_j  \in \mathcal P$, $\forall \, j \in \{ 1, \ldots, N\}$

  \STATE $j^*_{k+1} \in \Argmin_{j \in \{1, \ldots, N \} } \, g^*_j$ %%%with $j^*_{k+1} \notin\Jcal_k$

  \STATE $\Jcal_{k+1} = \Jcal_k \cup \{ j^*_{k+1}\}$

  \STATE $x^{k+1} \in \Argmin_{w \in \Pcal, \ \supp(w) \subseteq \Jcal_{k+1} } \, \| A w - y\|^2_2$

   \STATE $k\leftarrow k+1$

\ENDWHILE

%%%\STATE Iterate until the stopping criteria is met

\STATE Output: $x^* = x^k$

%=====================================================================================
\end{algorithmic}
%=====================================================================================
\end{algorithm}

At each step in the constrained matching pursuit algorithm, two constrained optimization problems are solved. The first problem, given in Line 4 of Algorithm~\ref{algo:constrained_MP}, is a constrained coordinate-wise minimization problem; the second problem, given in Line 7 of  Algorithm~\ref{algo:constrained_MP}, is a minimization problem on the constraint set $\Pcal$ subject to an additional support constraint $\supp(w) \subseteq \Jcal_{k+1}$.
\tblue{
Note that the first problem given in Line 4 is always solvable due to Assumption A.1; see the details below. The following assumption is made for the second problem given in Line 7:
\begin{itemize}
  \item [$\bf A.3$] The optimization problem in Line 7 of  Algorithm~\ref{algo:constrained_MP} attains a solution for any index set $J_{k+1}$.
\end{itemize}
}

In what follows, we discuss these two underlying problems and their solution properties. For a given $x \in \mathcal P$ and an index $j=1, \ldots, N$,  the first minimization problem can be written as
%
%({\bf should we use a general vector norm instead 2-norm?}):
%
\[
   (\mbox{P}_{x, j}): \quad \min_{t \in \mathbb R} \| y - A (x + t \, \mathbf e_j  ) \|^2_2 \qquad \mbox{ subject to }  \quad x + t \, \mathbf e_j  \in \mathcal P.
\]
Since $\mathcal P$ is closed, the following constraint set of $(\mbox{P}_{x, j})$ is a closed  set in $\mathbb R$
\begin{equation} \label{eqn:interval_j}
 \mathbb I_j(x) \, := \, \big \{ \, t \in \mathbb R \, | \, x + \mathbf e_j t \in \mathcal P \, \big\}.
\end{equation}
Besides, for any $x \in \mathcal P$ and $j=1, \ldots, N$, we have $0 \in \mathbb I_{j}(x)$, and $(\mbox{P}_{x, j})$ attains an optimal solution because $\| A_{\bullet j} \|_2 >0$.
Since $y=A u$ for some $u\in \Pcal$, we define, for any $u, v \in \mathcal P$ and $j=1, \ldots, N$,
%
%For the problem (\ref{eqn:constrained_L0}), the vector $y$ is given by $y=A u$ for some $u\in \Pcal$.
%
%Hence, for any $u, v \in \mathcal P$ and $j=1, \ldots, N$, define
%
\[
     f^*_j(u, v) \, := \, \min_{t \in \mathbb I_j(v)} \| A u - A (v + t \, \mathbf e_j  ) \|^2_2  = \min_{t \in \mathbb I_j(v)} \| A( u -v) - t A_{\bullet j}  \|^2_2.
\]

A particularly interesting and important case is when $\mathcal P$ is also convex. In this case, for any $v\in \Pcal$ and any index $j$, $\mathbb I_j(v)$ is closed and convex and thus  is a closed interval in $\mathbb R$. Letting
$
  a_j(v):=\inf \, \mathbb I_j(v)$ and  $b_j(v) := \sup \, \mathbb I_j(v)$,
 where $a_j(v)\in \mathbb R_-\cup\{-\infty\}$ and $b_j(v) \in \mathbb R_+\cup \{+\infty\}$, $\mathbb I_j(v)$ can be written as $\mathbb I_j(v)=[a_j(v), b_j(v)]$.
%
%{\bf Q: Are $a_j(v)$ and $b_j(v)$ continuous in $v$?}
%
For any given $u, v \in \Pcal$,  since $A_{\bullet j} \ne 0$, the minimization problem $\min_{t \in [a_j(v), b_j(v)]} \|A(u-v) - t \, A_{\bullet j}\|^2_2$ attains a unique optimal solution
\[
   t^*_j(u, v) \, = \, \left\{\begin{array}{llc} a_j(v), & \mbox{ if  } \  \wt t_j(u, v) \le a_j(v) \\ \wt t_j(u, v), & \mbox{ if  } \ \wt t_j(u, v) \in [a_j(v), b_j(v)] \\  b_j(v), & \mbox{ if  } \ \wt t_j(v)  \ge b_j(v)  \end{array} \right.,
\]
where
\begin{equation} \label{eqn:def_tilda_t}
  \wt t_j(u, v) \, := \, \langle A(u-v), A_{\bullet j} \rangle/\|A_{\bullet j}\|^2_2.
\end{equation}
Consequently,
\begin{equation} \label{eqn:f*_j_convex}
  f^*_j(u, v) \, = \, \left\{\begin{array}{llc} \|A(u-v)\|^2_2 - \| A_{\bullet j}\|^2_2 \cdot [ 2 a_j(v) \wt t_j(u, v) - a^2_j(v)], & \mbox{ if }\  \wt t_j(u, v) \le a_j(v) \\ \|A(u-v)\|^2_2 -\| A_{\bullet j}\|^2_2 \cdot \wt t^2_j(u, v), & \mbox{ if } \ \wt t_j(u, v) \in [a_j(v), b_j(v)] \\   \|A(u-v)\|^2_2 - \| A_{\bullet j}\|^2_2 \cdot [ 2 b_j(v) \wt t_j(u, v) - b^2_j(v)], & \mbox{ if } \ \wt t_j(v)  \ge b_j(v)  \end{array} \right.
\end{equation}
\tblue{
Since $2 a_j(v) \wt t_j(u, v) - a^2_j(v) \le \wt t^2_j(u, v)$ and $2 b_j(v) \wt t_j(u, v) - b^2_j(v) \le \wt t^2_j(u, v)$ for any $(u, v)$ and $j$, it is easy to see that $f^*_j(u, v) \ge \|A(u-v)\|^2_2 -\| A_{\bullet j}\|^2_2 \cdot \wt t^2_j(u, v)$ for any $(u, v)$ and $j$.
}

For illustration, we show the expressions of $f^*_j(u, v)$ for two special cases below.
%
%several different $a_i(v)$'s and $b_j(v)$'s below.
%

%\noindent

(i) $\mathbb I_j(v)=\mathbb R$, i.e., $a_j(v)=-\infty$ and $b_j(v)=+\infty$. In this case,
\begin{equation} \label{eqn:f*_RN}
f^*_j(u, v) \, = \, \|A(u-v)\|^2_2 -\| A_{\bullet j}\|^2_2 \cdot \wt t^2_j(u, v).
\end{equation}

%\noindent

 (ii) $\mathbb I_j(v)=\mathbb R_+$, i.e., $a_j(v)=0$ and $b_j(v)=+\infty$. In this case,
\begin{equation} \label{eqn:f*_RN+}
f^*_j(u, v) \, = \, \|A(u-v)\|^2_2 -\| A_{\bullet j}\|^2_2 \cdot \big([\wt t_j(u, v)]_+\big)^2.
\end{equation}

\mycut{
\noindent  (iii) $\mathbb I_j(v)=[a_j(v), \infty)$ with $a_j(v)<0$ and $b_j(v)=+\infty$. In this case,
\[
   f^*_j(u, v) = \|A(u-v)\|^2_2 - \| A_{\bullet j}\|^2_2 \, \Big\{ \, [2 a_j(v) \wt t_j(u, v) - a^2_j(v)] \cdot \delta_{ \wt t_j(u, v) \le a_j(v)} +  \wt t^2_j(u, v) \cdot \delta_{ \wt t_j(u, v) \ge a_j(v)} \, \Big\},
\]
where $\delta_S$ denotes the indicator function of a set $S$, i.e., $\delta_S(t)=1$ if $t\in S$ and $\delta_S(t)=0$ otherwise.

\noindent  (iv) $\mathbb I_j(v)=(-\infty, b_j(v)]$ with $a_j(v)=-\infty$ and $b_j(v)>0$. In this case,
\[
   f^*_j(u, v) = \|A(u-v)\|^2_2 - \| A_{\bullet j}\|^2_2 \, \Big\{ \,  \wt t^2_j(u, v) \cdot \delta_{ \wt t_j(u, v) \le b_j(v)} + [2 b_j(v) \wt t_j(u, v) - b^2_j(v)] \cdot \delta_{ \wt t_j(u, v) \le b_j(v)} \, \Big\}.
\]
}
%%\gap

%%%\gap

We next study the constrained minimization problem in Line 7 of  Algorithm~\ref{algo:constrained_MP} for a given $y \in \mathbb R^m$ and a given index set $\Jcal \subseteq\{1, \ldots, N\}$:
\begin{equation} \label{eqn:P_y_J}
  (\mbox{P}_{y, \Jcal}): \quad \min_{w \in \mathbb R^N} \| A w - y \|^2_2 \qquad \mbox{ subject to }  \quad w \in \mathcal P \quad \mbox{ and } \quad \supp(w) \subseteq \Jcal.
\end{equation}
Since $0\in \Pcal$, $(\mbox{P}_{y, \Jcal})$ is always feasible for any index set $\Jcal$, even if $\Jcal$ is empty.
%
%Note that we always assume that the minimization problem in Line 7 of  Algorithm~\ref{algo:constrained_MP} %has a solution in each step.
%
Certain solution existence and uniqueness results for $(\mbox{P}_{y, \Jcal})$ can be established under mild assumptions on $A$ and $\Pcal$ as shown below.

%%%We present solution existence and uniqueness results for $(\mbox{P}_{y, \Jcal})$ as follows.

\begin{lemma} \label{lemma:sol_existence}
  Let the set $\mathcal P \subseteq \mathbb R^N$ and the matrix $A \in \mathbb R^{m\times N}$. The following hold:
  \begin{itemize}
   \item [(i)] If $A \mathcal P$ is closed, then for any index set $\mathcal J$ and any $y \in \mathbb R^m$,  $(\mbox{P}_{y, \Jcal})$ attains an optimal solution.
     \item [(ii)] If $\Pcal$ is closed and an index set $\Ical$ is such that $A_{\bullet \Ical}$ has linearly independent columns, then $(\mbox{P}_{y, \Ical})$ has an optimal solution. If, in addition, $\mathcal P$ is convex, then such an optimal solution is unique.
    \end{itemize}
\end{lemma}

\begin{proof}
 (i) Given any $y \in \mathbb R^m$ and any index set $\mathcal J$,
   $(\mbox{P}_{y, \Jcal})$ is equivalent to $\min_{w \in \Pcal \cap \mathcal V} \|A w - y\|^2_2$, where $\mathcal V:=\{ z=(z_\Jcal, z_{\Jcal^c}) \, | \, z_{\Jcal^c} = 0\}$ is a subspace of $\mathbb R^N$. Note that $A \mathcal V$ is a subspace and thus closed.
 Since $A (\Pcal \cap \mathcal V)=(A\Pcal) \cap (A \mathcal V)$ and  $A\Pcal$ is closed,  $A (\Pcal \cap \mathcal V)$ is also closed. Moreover, the function $\| \cdot \|^2_2$ is continuous, coercive, and bounded below on $\mathbb R^m$. By \cite[Lemma 4.1]{MouShen_COCV18}, $(\mbox{P}_{y, \Jcal})$ has an optimal solution.

 (ii) Suppose $\Pcal$ is closed. Then the set $\Pcal_\Jcal:=\Pcal \cap \mathcal V$ is closed for any index set $\Jcal$, where  $\mathcal V$ is the subspace associated with $\Jcal$ defined in the proof for (i). Since $A_{\bullet \Ical}$ has linearly independent columns, it is easy to see that $\{A_{\bullet \Ical} \, w_\Ical \, | \, (w_\Ical, 0) \in \Pcal_\Ical \}$ is closed. By the similar argument for (i), $(\mbox{P}_{y, \Ical})$ attains an optimal solution.
 If, in addition, $\Pcal$ is convex, then $(\mbox{P}_{y, \Ical})$ is a convex optimization problem with a strongly convex objective function in $w_\Ical$. This yields a unique optimal solution for any $y\in \mathbb R^m$.
\end{proof}

%
%{\bf On uniqueness: there exists a dense set in $\mathbb R^{m\times N}$ such that any matrix $A$ from this %set is such that for any index set $\Jcal$ with $|\Jcal|=m$, $A_{\bullet \Jcal}$ has full column rank. This %condition guarantees the solution uniqueness of $(\mbox{P}_{y, \Jcal})$ when $\Pcal$ is closed and convex.}
%
%\gap
%

Typical constraint sets $\mathcal P$ satisfying the closedness assumption in statement (i) of Lemma~\ref{lemma:sol_existence} for an arbitrary matrix $A \in \mathbb R^{m\times N}$ include compact sets and polyhedral sets.  Also see Corollary~\ref{coro:sol_existence_CP_adm} in Section~\ref{sect:CP_admissible_set} for a general class of sets on which $(\mbox{P}_{y, \Jcal})$ attains a solution.

When $(\mbox{P}_{y, \Jcal})$ is a convex optimization problem (whose  $\Pcal$ is closed and  convex),  well developed numerical solvers can be exploited to solve  $(\mbox{P}_{y, \Jcal})$, e.g., the gradient projection method and primal-dual schemes.
%
%When $\Pcal$ is a closed and  convex set and $(\mbox{P}_{y, \Jcal})$ attains an optimal solution %$w^*=(w^*_\Jcal, 0) \in \Pcal$,
%
In particular, the necessary and sufficient optimality condition for an optimal solution $w^*=(w^*_\Jcal, 0)\in \Pcal$ of $(\mbox{P}_{y, \Jcal})$ is given by the variational inequality (VI): $\langle A^T_{\bullet \Jcal}( A_{\bullet \Jcal} w^*_\Jcal-y),  w_\Jcal- w^*_\Jcal \rangle \ge 0$ for all $(w_\Jcal, 0) \in \Pcal$.
When $\Pcal$ is a closed convex cone, the above VI is equivalent to the cone complementarity problem: $  \mathcal C \ni w^*_\Jcal \perp A^T_{\bullet \Jcal}( A_{\bullet \Jcal} w^*_\Jcal-y) \in \mathcal C^*$, where the closed convex cone $\mathcal C:=\{ w_\Jcal \, | \,  (w_{\Jcal}, 0) \in \Pcal \}$ and $\mathcal C^*$ denotes the dual cone of $\mathcal C$.
Especially, when $\Pcal=\mathbb R^N_+$, it is further equivalent to the linear complementarity problem (LCP): $0 \le w^*_\Jcal \perp  A^T_{\bullet \Jcal}( A_{\bullet \Jcal} w^*_\Jcal-y) \ge 0$. These optimality conditions will be invoked later.
%
%in the subsequent sections.
%

\tblue{
At the end of this section, we present an example to illustrate Algorithm~\ref{algo:constrained_MP}. This example shows that a desired solution can be recovered from a nonconvex constraint set via Algorithm~\ref{algo:constrained_MP}.}

\begin{example} \rm \label{example:nonconvex_constraint}
\tblue{
    Consider the closed nonconvex set $\Pcal =  \{ x =(x_1, x_2) \in \mathbb R^2 \, | \, x \ge 0,   x_2 \le 1, x^2_2 \ge x_1 \}\cup \{ (x_1, 0) \in \mathbb R^2 \, | \, x_1 \in [0, 1] \}$.
    %
    %; see Figure XXX.
    %
    %Hence, $\Pcal$ is {\em CP admissible}.
    %
    Let $A=[ \frac{3}{4} \ 1 ]\in \mathbb R^{1\times 2}$, $y=\frac{3}{2}$, and the set $\Scal:=\{ x \, | \, A x = y \} = \{ x =(x_1, x_2) \,  | \, x_2= 2 -\frac{4}{3}x_1 \}$. Thus $\Pcal \cap \Scal$ is the line segment joining the points $p=(\frac{3}{4}, 1)$ and $q =  (\frac{(\sqrt{105}-3)^2}{64}, \frac{\sqrt{105}-3}{8})$. Hence, any solution of the  recovery problem given by (\ref{eqn:constrained_L0}) has support size two.}

    \indent $\bullet$  \tblue{Step 1: Since $x^0=0$, the problem in Line 4 yields: (i) $\min_{t} ( \frac{3}{2} - \frac{3}{4} t)^2_2$ subject to $t \mathbf e_1 \in \Pcal$ or equivalently $0 \le t \le 1$. Hence, $g^*_1=\frac{9}{16}$; (ii)  $\min_{t} ( \frac{3}{2} - t)^2_2$ subject to $t \mathbf e_2 \in \Pcal$ or equivalently $0 \le t \le 1$. Hence, $g^*_2=\frac{1}{4}$. Thus $j^*_1=2$, and $\Jcal_1=\{ 2 \}$. Further, the problem in Line 7 becomes: $\min \|A w - \frac{3}{2} \|^2_2$ subject to $w \in \Pcal$ with $\supp(w)\subseteq \Jcal_1=\{2 \}$. Therefore, its unique optimal solution is $w^*=(0, 1)=x^1$.
    }
    %
    %\textcolor{blue}{Note that $x^1$ is a boundary point of $\Pcal$.}
    %

    %%\tblue{
    \indent $\bullet$  \tblue{Step 2: Since $x^1=(0, 1)$, the problem in Line 4 yields: (i) $\min_{t} (\frac{1}{2}  - \frac{3}{4} t)^2$ subject to $x^1+ t \mathbf e_1 \in \Pcal$ or equivalently $0\le t \le 1$. Thus $g^*_1=0$; (ii) $\min_{t} (\frac{1}{2}  - t)^2$ subject to $x^1+ t \mathbf e_2 \in \Pcal$ or equivalently $-1 \le t\le 0$. Thus $g^*_2=\frac{1}{4}$.  Hence, $j^*_2= 1$, and $\Jcal_2=\{1, 2 \}$. Therefore, the problem in Line 7 becomes: $\min_w \|A w - \frac{1}{2} \|^2_2$ subject to $w \in \Pcal$ with $\supp(w)\subseteq \Jcal_2=\{1, 2 \}$. Thus any point in $\Pcal \cap \Scal=[p, q]$ is an optimal solution.
%
%    its optimal solution is $w^* \in [p ,q]$, where $[p, q]=\Pcal \cap \Scal$ defined before.
%
    (Note that there are infinitely many solutions.)  Consequently, a desired solution is recovered in Step 2.}

   %
   % This shows that the exact support and vector recovery is achieved in Step 2. \quad  $\square$
  \end{example}

%--------------------------------------------------------------------------------------------------------
%
\section{Exact Support Recovery via Constrained Matching Pursuit} \label{sect:exact_supp_recover}
%%: Necessary and Sufficient Conditions}

Fix $K \in \mathbb N$ with $K < N$ throughout the rest of the paper.
\tblue{
Recall that $\Sigma_K$ is the set of all vectors $x \in \mathbb R^N$  satisfying $|\supp(x)|\le K$.
}
For a given $z \in \Sigma_K \cap \mathcal P$, let $\big( (x^k, j^*_k, \Jcal_k) \big)_{k \in \mathbb N}$ be a sequence of triples generated by Algorithm~\ref{algo:constrained_MP} with $y=A z$ starting from $x^0=0$ and $\Jcal_0=\emptyset$, where $\Jcal_{k+1} =\Jcal_k \cup \{ j^*_{k+1} \}$ such that $\Jcal_0 \subseteq \Jcal_1 \subseteq \Jcal_2 \subseteq \cdots$. Note that there are multiple sequences in general for a given $z$, since the optimization problems in Lines 5 and 7 of Algorithm~\ref{algo:constrained_MP} may attain non-unique solutions at each step. For example, if the underlying problem (\ref{eqn:P_y_J}) is a convex minimization problem with non-unique solutions for some $\Jcal=\Jcal_k$ and $y=Az$, then it attains infinitely many $x^k$'s. In this case, there are infinitely many sequences  $\big( (x^k, j^*_k, \Jcal_k) \big)_{k \in \mathbb N}$. \tblue{Another example is given by Step 2 of Example~\ref{example:nonconvex_constraint}.}

\begin{definition} \rm \label{def:exact_suppt_recovery}
\tblue{
 Given a matrix $A \in \mathbb R^{m\times N}$ and a constraint set $\Pcal$, we say that  {\it the exact support recovery} of a given $z \in \Sigma_K \cap \mathcal P$ is achieved from $y=A z$ via constrained matching pursuit (c.f. Algorithm~\ref{algo:constrained_MP}),  if along any sequence $\big( (x^k, j^*_k, \Jcal_k) \big)_{k \in \mathbb N}$, there exists an index $s$
such that $\mathcal J_s = \supp(z)$.
If the exact support recovery of any $z \in \Sigma_K \cap \mathcal P$ is achieved, then we call {\it the exact support recovery on $\Sigma_K \cap \mathcal P$} (or simply the exact support recovery) is achieved.
}
\end{definition}

Necessary and sufficient conditions for the exact support recovery are given as follows.

%------------------------------------------------------------------------------------------
%

\begin{lemma} \label{lem:exact_suppt_recovery_index}
 Given $0 \ne u \in \sum_K \cap \,  \mathcal P$ and an index set $\mathcal J \subseteq \supp(u)$, let $v$ be an optimal solution to $\min_{w \in \mathcal P, \ \supp(w)\subseteq \mathcal J} \| A (u - w) \|^2_2$, where we assume that such a solution exists. Then $f^*_j(u, v) = \| A (u -v) \|^2_2$ for each $j\in \mathcal J$, and $f^*_j(u, v) \le \| A (u -v) \|^2_2$ for each $j \notin \mathcal J$.
\end{lemma}

\begin{proof}
   Consider an arbitrary $j \notin \mathcal J$. Noting that $0\in \mathbb I_j(v)$, we have $f^*_j(u, v) \le \| A(u -v) \|^2_2$. We then
    consider an arbitrary $j \in \mathcal J$. For any $t \in \mathbb I_j(v)$, we have $v + \mathbf e_j t \in \mathcal P$ and $\supp(v + \mathbf e_j t) \subseteq \mathcal J$.  Since $v$ is an optimal solution to $\min_{w \in \mathcal P, \ \supp(w)\subseteq\mathcal J} \| A (u - w) \|^2_2$, we have $\|A (u - v)\|^2_2 \le \|A u - A(v + \mathbf e_j t) \|^2_2$ for all $t\in \mathbb I_{j}(v)$. This shows that  $\| A (u -v) \|^2_2 \le f^*_j(u, v)$. Furthermore, $f^*_j(u, v) \le \| A(u -v) \|^2_2 $ since $0\in \mathbb I_j(v)$.
    Therefore, $f^*_j(u, v) = \| A (u -v) \|^2_2$ for each $j \in \Jcal$.
\end{proof}

\begin{theorem} \label{thm:nec_suf_condition_for_exact_supp_recovery}
Given a matrix $A \in \mathbb R^{m\times N}$ and a constraint set $\Pcal$, let $0 \ne z \in \Sigma_K \cap \Pcal$ with $|\supp(z)|=r$.
Then the exact support recovery of $z$ is achieved via constrained matching pursuit if and only if for any sequence $\big( (x^k, j^*_k, \Jcal_k) \big)_{k \in \mathbb N}$ generated by Algorithm~\ref{algo:constrained_MP} with $y=A z$, the following holds
\begin{equation} \label{eqn:exact_sppt_recovery_inequality}
   \min_{j \in \supp(z)\setminus \Jcal_k} f^*_j(z, x^k) \, < \, \min_{j\in [\supp(z)]^c} f^*_j(z, x^k), \qquad \forall \ k=0, 1, \ldots, r-1.
\end{equation}
Moreover, when the exact support recovery of $z$ is achieved, the support of $z$ is firstly attained at the $r$th step along any sequence $\big( (x^k, j^*_k, \Jcal_k) \big)_{k \in \mathbb N}$, i.e., $\Jcal_r=\supp(z)$ and $\Jcal_k \subset \supp(z)$ for each $k < r$.
\end{theorem}

\begin{proof}
``If''. For the given $z$,  suppose an arbitrary sequence $\big( (x^k, j^*_k, \Jcal_k) \big)_{k \in \mathbb N}$ generated by Algorithm~\ref{algo:constrained_MP} satisfies (\ref{eqn:exact_sppt_recovery_inequality}).
We prove below by induction on iterative steps of  Algorithm~\ref{algo:constrained_MP} that $\Jcal_k \subseteq \supp(z)$ with $|\Jcal_k|=k$ and $j^*_{k+1} \in \supp(z)\setminus\Jcal_k$ for each $k=1, \ldots, r-1$.
At Step 1, since $x^0=0$ and  $\Jcal_0$ is the empty set, we deduce from (\ref{eqn:exact_sppt_recovery_inequality}) that $\min_{j \in \supp(z)} f^*_{j}(z, 0) \, < \,  \min_{j \in [\supp(z) ]^c} f^*_{j}(z, 0)$. It follows from Algorithm~\ref{algo:constrained_MP} that the optimal index $j^*_1 \in \Argmin_{j=1, \ldots, N} f^*_j(z, 0)$ satisfies $j^*_1 \in \supp(z)$ such that $\Jcal_1=\{ j^*_1 \} \subseteq  \supp(z)$ and $|\Jcal_1|=1$.
Now  suppose $\Jcal_k \subseteq \supp(z)$ with $|\Jcal_k|=k$  and $j^*_k \in \supp(z) \setminus \Jcal_{k-1}$ for  $1\le k \le r-2$. Consider Step $(k+1)$. In view of Lemma~\ref{lem:exact_suppt_recovery_index}, the optimal index $j^*_{k+1} \in \Argmin_{j=1, \ldots, N} f^*_j(z, x^k)$ satisfies $j^*_{k+1} \notin \Jcal_k$. Since $\Jcal_k \subseteq \supp(z)$, $j^*_{k+1} \in [\supp(z) \setminus \Jcal_k]\cup [\supp(z)]^c$.
Further, it follows from (\ref{eqn:exact_sppt_recovery_inequality}) that $j^*_{k+1} \in \supp(z)\setminus \Jcal_k$.
Therefore, $\Jcal_{k+1}:=\Jcal_k \cup \{ j^*_{k+1} \}$ satisfies  $\Jcal_{k+1} \subseteq  \supp(z)$ and $|\Jcal_{k+1}|=k+1$. By the induction principle, we see that $\Jcal_r\subseteq \supp(z)$ and $|\Jcal_r|=r=|\supp(z)|$. This implies that $\Jcal_r=\supp(z)$ and $\Jcal_k \subset \supp(z)$ for each $k < r$.

``Only if''. Suppose the exact support recovery of $z$ is achieved via Algorithm~\ref{algo:constrained_MP}.
By Definition~\ref{def:exact_suppt_recovery}, we claim that for any given sequence $\big( (x^k, j^*_k, \Jcal_k) \big)_{k \in \mathbb N}$ generated by Algorithm~\ref{algo:constrained_MP} with $y=A z$ starting from $x^0=0$ and $\Jcal_0=\emptyset$, the following must hold:
\[
   \min_{j \in \supp(z)} f^*_j(z, x^k) \, < \, \min_{j\in [\supp(z)]^c} f^*_j(z, x^k), \qquad \forall \ k=0, 1, \ldots, r-1,
\]
This is because otherwise,  $\min_{j \in \supp(z)} f^*_j(z, x^\ell) \ge \min_{j\in [\supp(z)]^c} f^*_j(z, x^\ell)$ for some $\ell=0,1, \ldots, r-1$. Hence, there exists an optimal index $j^*_{\ell+1}\notin \supp(z)$ such that $\Jcal_{\ell+1} \neq \supp(z)$ (along a possibly  different sequence), leading to $\Jcal_{s}  \neq \supp(z)$ for all $s \ge \ell$. Note that  $\Jcal_k\ne \supp(z)$ for each $k=1, \ldots, \ell$ since each $|\Jcal_k|<r$. Therefore, there exists a sequence so that $\Jcal_k \ne \supp(z)$ for all $k \in \mathbb N$, yielding a contradiction.
Finally, since each $x^k$ is a minimizer of $\min_{w\in \Pcal, \supp(w) \subseteq \Jcal_k} \|A (z - w) \|^2_2$, we deduce via Lemma~\ref{lem:exact_suppt_recovery_index} that $\min_{j \in \supp(z)} f^*_j(z, x^k)=\min_{j \in \supp(z)\setminus \Jcal_k} f^*_j(z, x^k)$. This leads to (\ref{eqn:exact_sppt_recovery_inequality}).
\end{proof}

In what follows, we show the implications of the exact support recovery.

\begin{proposition} \label{prop:index_set}
Given a matrix $A$ and a constraint set $\Pcal$, let $0\ne z \in \Sigma_K \cap \mathcal P$ with $|\supp(z)|=r$ be such that the exact support recovery of $z$ is achieved. Then for any sequence $\big( (x^k, j^*_k, \Jcal_k) \big)_{k \in \mathbb N}$ generated by Algorithm~\ref{algo:constrained_MP} with $y=A z$, the following hold:
\begin{itemize}
   \item [(i)] $\| A (z - x^{k+1}) \|^2_2 \le f^*_{j^*_{k+1}}(z, x^k) < \| A (z - x^k) \|^2_2$ for each $k=0, 1, \ldots, r-1$;
   \item [(ii)] For each $k=1, \ldots, r$, $(x^k)_{j^*_k} \ne 0$, and $x^k_{\Jcal_{k-1}} \ne 0$ when $k>1$. Hence, $\supp(x^k)=\Jcal_k$ for $k=1, 2$.
\end{itemize}
\end{proposition}

\begin{proof}
  (i) Fix $k \in \{0, 1, \ldots, r-1\}$. Since $x^k$ is an optimal solution to $\min_{w \in \Pcal, \, \supp(w)\subseteq \Jcal_k} \| A(z - w)\|^2_2$, it follows from Lemma~\ref{lem:exact_suppt_recovery_index} that $f^*_j(z, x^k) \le \| A (z - x^k)\|^2_2$ for all $j=1, \ldots, N$. In light of the inequality given by (\ref{eqn:exact_sppt_recovery_inequality}), we have $\min_{j \in \supp(z)\setminus \mathcal J_k} f^*_{j}(z, x^k) \, < \,  \min_{j \in [\supp(z) ]^c} f^*_{j}(z, x^k) \le \| A (z - x^k)\|^2_2$. Since $j^*_{k+1} \in \Argmin_{j \in \supp(z)\setminus \mathcal J_k} f^*_{j}(z, x^k)$, we have $ f^*_{j^*_{k+1}}(z, x^k) < \| A (z - x^k) \|^2_2$. Besides, by virtue of the definition of $f^*_{j}(\cdot, \cdot)$, we deduce that there exists $0 \ne t_* \in \mathbb I_{j^*_{k+1}}$ such that
  \[
   f^*_{j^*_{k+1}}(z, x^k) \, =  \, \big \|A z- A \big(x^k+ t_* \mathbf e_{j^*_{k+1}} \big) \big\|^2_2.
  \]
    Note that $x^k + t_* \mathbf e_{j^*_{k+1}} \in \mathcal P$ and $\supp(  x^k + t_* \mathbf e_{j^*_{k+1}}  ) = \Jcal_k \cup \{ j^*_{k+1}\}= \Jcal_{k+1}$. Since $x^{k+1}$ is an optimal solution to $\min_{w \in \Pcal, \, \supp(w)\subseteq \Jcal_{k+1}} \| A(z - w)\|^2_2$, we have $\|A(z - x^{k+1} ) \|^2_2 \le \|A z- A (x^k+ t_* {\mathbf e}_{j^*_{k+1}} ) \|^2_2 =   f^*_{j^*_{k+1}}(z, x^k)$.

 (ii) Fix $k \in \{1, \ldots, r\}$. We first show the following claim: $x^k_{\Jcal_k \setminus \Jcal_s} \ne 0$ for each $s \in \{0, 1, \ldots, k-1 \}$.
 Suppose, in contrast, that  $(x^k)_{\Jcal_k \setminus \Jcal_s} = 0$ for some $s\in \{0, 1, \ldots, k-1\}$. In light of $\Jcal_s \subset \Jcal_k$, we have $\supp(x^k)\subseteq \Jcal_{s}$.
 Since $x^k \in \Pcal$ and $x^{s}$ is an optimal solution to $\min_{w \in \Pcal, \, \supp(w)\subseteq \Jcal_{s}} \| A(z - w)\|^2_2$, we deduce that $\| A(z - x^{s})\|^2_2 \le \| A(z - x^k)\|^2_2$. Since $s<k$, this yields a contradiction to statement (i). Hence, the claim holds. In view of $\Jcal_k\setminus \Jcal_{k-1} =  \{ j^*_k \}$, we obtain $(x^k)_{j^*_k} \ne 0$.

  We  then show that $x^k_{\Jcal_{k-1}} \ne 0$ when $k>1$. Suppose, in contrast, that $x^k_{\Jcal_{k-1}} = 0$. Then $\supp(x^k)=\{ j^*_{k} \}$ since $(x^k)_{j^*_k} \ne 0$. By the definition of $f^*_{j}(\cdot, \cdot)$, we have that $f^*_{j^*_k}(z, 0) \le  \| A( z - x^k) \|^2_2$. Furthermore, we deduce via $x^0=0$ that $f^*_{j^*_1}(z, x^0)\le f^*_{j^*_k}(z, 0)$. Therefore, $f^*_{j^*_1}(z, x^0) \le  \| A( z - x^k) \|^2_2$.
  On the other hand, it follows from statement (i) that $\| A(z - x^1)\|^2_2 \le f^*_{j^*_1}(z, x^0)$. This leads to $\| A(z - x^1)\|^2_2 \le \| A( z - x^k) \|^2_2$. Since $k>1$, we attain a contradiction to statement (i). Consequently, $x^k_{\Jcal_{k-1}} \ne 0$ when $k>1$.
\end{proof}

We specify particular conditions for the exact support recovery on $\mathbb R^N$ and $\mathbb R^N_+$, respectively.

\begin{corollary} \label{coro:Nec_Suf_suppt_recovery_RN_RN+}
Given a matrix $A \in \mathbb R^{m\times N}$ with unit columns (i.e., $\|A_{\bullet i}\|_2=1$ for all $i$) and a constraint set $\Pcal$,  let $0 \ne z \in \Sigma_K \cap \Pcal$ with $|\supp(z)|=r$. The following hold:
\begin{itemize}
  \item[(i)] When $\Pcal=\mathbb R^N$, the exact support recovery of $z$ is achieved if and only if for any sequence $\big( (x^k, j^*_k, \Jcal_k) \big)_{k \in \mathbb N}$ generated by Algorithm~\ref{algo:constrained_MP} with $y=A z$, $$\max_{j \in \supp(z)\setminus \Jcal_k} |A^T_{\bullet j} A(z - x^k)| \, > \, \max_{j\in [\supp(z)]^c} |A^T_{\bullet j} A(z - x^k)|, \qquad \forall \ k=0, 1, \ldots, r-1;$$
 \item [(ii)] When $\Pcal=\mathbb R^N_+$, the exact support recovery of $z$ is achieved if and only if for any sequence $\big( (x^k, j^*_k, \Jcal_k) \big)_{k \in \mathbb N}$ generated by Algorithm~\ref{algo:constrained_MP} with $y=A z$,
     \[
     \max_{j \in \supp(z)\setminus \Jcal_k} [A^T_{\bullet j} A(z - x^k) ]_+ \, > \, \max_{j\in [\supp(z)]^c} [ A^T_{\bullet j} A(z - x^k) ]_+, \qquad \forall \ k=0, 1, \ldots, r-1.
     \]
 \end{itemize}
\end{corollary}

\begin{proof}
 (i) Let $\Pcal=\mathbb R^N$. Then for any $v \in \mathbb R^N$ and any index $j$, $\mathbb I_j(v)=\mathbb R$. It follows from the definition of $\wt t_j(u, v)$ given by (\ref{eqn:def_tilda_t}) and $f^*_j(u, v)$ given by (\ref{eqn:f*_RN}) that (\ref{eqn:exact_sppt_recovery_inequality}) holds if and only if for each $k=0,1, \ldots, r-1$,
 \[
    \max_{j \in \supp(z)\setminus \Jcal_k} \langle A(z- x^k), A_{\bullet j} \rangle^2 > \max_{j \in [\supp(z)]^c} \langle A(z- x^k), A_{\bullet j} \rangle^2.
 \]
 The latter is equivalent to $\max_{j \in \supp(z)\setminus \Jcal_k} |A^T_{\bullet j} A(z - x^k)| \, > \, \max_{j\in [\supp(z)]^c} |A^T_{\bullet j} A(z - x^k)|$.

(ii) Let $\Pcal=\mathbb R^N_+$. Consider the pair $(x^k, \Jcal_k)$ for any fixed $k\in \{0, 1, \ldots, r-1\}$. For each $j \in \supp(z)\setminus \Jcal_k$, we have $(x^k)_j=0$ such that $\mathbb I_j(x^k)=\mathbb R_+$. Further, since $\supp(x^k) \subset \supp(z)$ as shown in Theorem~\ref{thm:nec_suf_condition_for_exact_supp_recovery}, we see that for any $j\in [\supp(z)]^c$, $j \notin \supp(x^k)$ such that $(x^k)_j=0$ and $\mathbb I_j(x^k)=\mathbb R_+$. Hence, in view of $f^*_j(\cdot, \cdot)$ given by  (\ref{eqn:f*_RN+}), we see that $\min_{j \in \supp(z)\setminus \Jcal_k} f^*_j(z, x^k) < \min_{j\in [\supp(z)]^c} f^*_j(z, x^k)$  if and only if $\max_{j \in \supp(z)\setminus \Jcal_k} ([A^T_{\bullet j} A(z - x^k) ]_+)^2  >  \max_{j\in [\supp(z)]^c} ([ A^T_{\bullet j} A(z - x^k) ]_+)^2$, which is equivalent to $\max_{j \in \supp(z)\setminus \Jcal_k} [A^T_{\bullet j} A(z - x^k) ]_+ > \max_{j\in [\supp(z)]^c} [ A^T_{\bullet j} A(z - x^k) ]_+$.
%
%We use the induction argument on $k=0, 1, \ldots, r-1$. When $k=0$, $x^0=0$ and $\Jcal_0$ is the empty set. %We thus have $\mathbb I_j(x^0)=0$ for any $j$. Hence,  in view of $f^*_j(\cdot, \cdot)$ given by  %(\ref{eqn:f*_RN+}), we see that $\min_{j \in \supp(z)\setminus \Jcal_0} f^*_j(z, x^0) < \min_{j\in %[\supp(z)]^c} f^*_j(z, x^0)$ holds if and only if $\max_{j \in \supp(z)\setminus \Jcal_0} ([A^T_{\bullet j} %A(z - x^0) ]_+)^2  >  \max_{j\in [\supp(z)]^c} ([ A^T_{\bullet j} A(z - x^0) ]_+)^2$, which is equivalent to %$\max_{j \in \supp(z)\setminus \Jcal_0} [A^T_{\bullet j} A(z - x^0) ]_+ > \max_{j\in [\supp(z)]^c} [ %A^T_{\bullet j} A(z - x^0) ]_+$.
%
This yields the desired result.
%%% follows from the similar argument and  $f^*_j(\cdot, \cdot)$ given by  (\ref{eqn:f*_RN+}).
\end{proof}

 Inspired by Theorem~\ref{thm:nec_suf_condition_for_exact_supp_recovery}, we introduce the following condition for a matrix $A$ and a constraint set $\Pcal$:
%
%For a given matrix $A$ and a constraint set $\Pcal$, we introduce condition $(\mathbf H)$:  \\
%
\begin{align}
(\mathbf H): \quad & \mbox{For any $0 \ne u \in \Sigma_K \cap \,  \mathcal P$, any index set $\Jcal \subset \supp(u)$ (where $\Jcal$ is possibly the empty set),} \notag \\
& \mbox{and an arbitrary  optimal solution $v$ of $\min_{w \in \mathcal P, \ \supp(w)\subseteq \mathcal J} \| A (u - w) \|^2_2$, the following holds:} \notag \\
 %
      %%%%(\mathbf H'): \quad
    &  \qquad \qquad  \qquad \qquad \quad \min_{j \in \supp(u)\setminus \mathcal J} f^*_{j}(u, v) \, < \,  \min_{j \in [\supp(u) ]^c} f^*_{j}(u, v). \label{eqn:condition_H'}
\end{align}
The next proposition states that $(\mathbf H)$ is a sufficient condition for the exact support recovery. We omit its proof since it follows directly from the fact that the inequality in (\ref{eqn:condition_H'}) implies  (\ref{eqn:exact_sppt_recovery_inequality}) given in Theorem~\ref{thm:nec_suf_condition_for_exact_supp_recovery}.

\begin{proposition} \label{prop:condition_H_suppt_recovery}
Given a matrix $A\in \mathbb R^{m\times N}$ and a constraint set $\Pcal$, suppose condition $(\mathbf H)$ holds. Then  the exact support recovery is achieved on $\Sigma_K \cap \Pcal$.
\end{proposition}

\begin{remark} \rm
 In general, condition $(\mathbf H)$ is {\em not} necessary for the exact support recovery. This is because the exact support recovery of a vector $z \in \Sigma_K \cap \Pcal$ requires that the inequality (\ref{eqn:exact_sppt_recovery_inequality}) hold for $\Jcal_k$'s only along a sequence $\big( (x^k, j^*_k, \Jcal_k) \big)_{k \in \mathbb N}$ for $z$, while condition $(\mathbf H)$ says that the inequality (\ref{eqn:condition_H'}) hold for {\em all} proper subsets $\Jcal \subset \supp(z)$. Nevertheless, condition $(\mathbf H)$ is  necessary for the exact support recovery when $K$ is small; see Corollary~\ref{coro:condition_H'_necessary_RN_S2} for $\Sigma_2\cap \mathbb R^N$ and Corollary~\ref{coro:condition_H'_necessary_RN+_S2}  for $\Sigma_2\cap \mathbb R^N_+$, respectively.
\end{remark}

Before ending this section, we give an example of a closed convex set $\Pcal$, on which no matrix $A$ can achieve the exact support recovery. It demonstrates that the exact support recovery and condition $(\mathbf H)$ not only depend on the measurement matrix $A$ but also critically rely on the constraint set $\mathcal P$.

\begin{example} \label{example:counter01} \rm
Let $d=(d_1, \ldots, d_N)^T \in \mathbb R^N$ be such that $d_i \ne 0$ for each $i$. Consider the hyperplane $\mathcal P:=\{ x \in \mathbb R^N \, | \, d^T x = 0 \}$. Clearly, $\mathcal P$ is closed and convex, and it contains the zero vector and other sparse vectors.
 %
% (When $\gamma=0$, $\mathcal C$ is a subspace and thus a polyhedral cone.)
%
  Since each $d_i\ne 0$, it is easy to verify that for any $v\in \mathcal P$ and any index $j$, the set $\mathbb I_j(v) =\{ 0 \}$. This shows that for any $u, v \in \mathcal P$ and any index $j$, $f^*_j(u,v)=\| A(u-v)\|^2_2$ for any matrix $A$. Hence, for any $z \in \Sigma_K \cap \Pcal$, we deduce that at Step 1 of Algorithm~\ref{algo:constrained_MP}, $\Argmin_{j \in \{1, \ldots, N\}} f^*(z, 0)=\{1, \ldots, N \}$. Thus $j^*_1$ can be chosen as $j^*_1 \notin\supp(z)$. This means that no matrix $A$ achieves the exact support recovery of any $z \in\Sigma_K \cap \Pcal$. It also implies that no matrix $A$ satisfies condition $(\mathbf H)$ on $\Pcal$.
\end{example}

%---------------------------------------------------------------------------
%
\section{Coordinate Projection Admissible Sets} \label{sect:CP_admissible_set}

%%%{\bf This set should contain sufficiently many sparse vectors.}

Since the exact recovery via constrained matching pursuit
%
%not only depends on a measurement matrix $A$ but also
%
critically relies on a constraint set,  it is essential to find a class of constraint sets to which the constrained matching pursuit can be successfully applied for exact recovery.
 An ideal class of constraint sets is expected to satisfy some crucial conditions,  including but not limited to: (i) each set in this class contains sufficiently many sparse vectors; (ii) this class of sets is broad enough to include important sets arising from applications, such as $\mathbb R^N$ and $\mathbb R^N_+$; and (iii) (relatively) easily verifiable sufficient recovery conditions can be established using general properties of this class of sets. Motivated by these requirements, we identify an important class of constraint sets in this section and study their analytic properties to be used for the exact recovery.

We introduce some notation first. Let $\mathcal U$ be a nonempty set in $\mathbb R$, and $\Ical$ be an index subset of $\{1, \ldots, N\}$. We let $\mathcal U^\Ical := \{ x =(x_1, \ldots, x_N)^T \in \mathbb R^N \, | \, x_i \in \mathcal U, \forall \, i \in \Ical, \mbox{ and } \ x_{\Ical^c}=0 \}$, and $\mathcal U_\Ical:=\{ u \in \mathbb R^{|\Ical|} \, | \, u_i \in \mathcal U, \forall \, i \in \Ical \}$.
For each $x\in \mathbb R^N$ and an index set $\Ical$, define the coordinate projection operator $\pi_{\Ical}:\mathbb R^N \rightarrow \mathbb R^N$ as $\pi_\Ical(x):=z$, where $z_i=x_i, \forall \, i\in \Ical$ and $z_{\Ical^c}=0$. If $\Ical$ is the empty set, then $\pi_\Ical(x)=0, \forall \, x$. We often write $\pi_\Ical(x)=(x_\Ical, 0)$ with $x_{\Ical^c}=0$ for notational simplicity. We also write $\pi_{\{i\}}$ as $\pi_i$ for $i=1, \ldots, N$ when the context is clear.
 Given an index set $\Ical$, $\pi_\Ical$ is obviously a linear operator on $\mathbb R^N$ given by $\pi_\Ical(x) =  (x_\Ical, 0)$ for $x =(x_\Ical, x_{\Ical^c})\in \mathbb R^N$.
%
%where the matrix
%$
%   W \, = \, \begin{bmatrix} W_{\Ical \Ical} &  W_{\Ical \Ical^c} \\ W_{\Ical^c \Ical} &  W_{\Ical^c \Ical^c} %\end{bmatrix} \, = \, \begin{bmatrix} I &  0 \\ 0 &  0 \end{bmatrix} \in \mathbb R^{N\times N}.
%$
%
Let $\circ$ denotes the composition of two functions.
For any index sets $\Ical, \Jcal \subseteq\{1, \ldots, N\}$, the following results can be easily established:
\begin{equation} \label{eqn:pi_two_index sets}
   \pi_\Ical \circ \pi_\Jcal \, = \, \pi_{\Ical \, \cap \Jcal} \, = \, \pi_\Jcal \circ \pi_\Ical.
\end{equation}
%%where $\circ$ denotes the composition of two functions.

\begin{definition} \rm
We call a nonempty set $\mathcal P \subseteq \mathbb R^N$ {\em coordinate projection admissible} or simply {\em CP admissible} if for any $x \in \mathcal P$ and any index set $\Jcal \subseteq \supp(x)$, $\pi_\Jcal(x)=(x_\Jcal, 0) \in \mathcal P$, where $\Jcal$ may be the empty set.
\end{definition}

Clearly, $\mathcal P$ must contain the zero vector (by setting $\Jcal = \emptyset$). An equivalent geometric condition for a CP admissible set is shown in the following lemma.

\begin{lemma} \label{lem:CP_adm_geometry}
 $\mathcal P$ is CP admissible if and only if $\pi_\Ical(\mathcal P) \subseteq \mathcal P$ for any index set $\Ical \subseteq \{1, \ldots, N\}$.
\end{lemma}

\begin{proof}
  ``If''. Since $\pi_\Ical(\mathcal P) \subseteq \mathcal P$ for any index set $\Ical$, we have $\pi_\Ical(x) \in \mathcal P$  for any $x \in \mathcal P$ and any $\Ical$. Hence, for any $x \in \mathcal P$ and any index set $\Jcal \subseteq \supp(x)$, we have $\pi_\Jcal(x) \in \mathcal P$. This shows that $\mathcal P$ is CP admissible.

  ``Only If''. Suppose $\mathcal P$ is CP admissible, and let $\Ical$ be an arbitrary index set. It suffices to show that $\pi_\Ical(x) \in \mathcal P$ for any given $x \in \mathcal P$. Toward this end, in view of $
  \Ical = (\Ical \cap \supp(x)) \cup (\Ical\setminus \supp(x))$ and $x_{\Ical\setminus \supp(x)}=0$, we have $\pi_\Ical(x) = (x_{\Ical\cap \supp(x)}, x_{\Ical\setminus \supp(x)}, x_{\Ical^c}) = (x_{\Ical \cap \supp(x)}, 0, 0) = \pi_{ \Ical \cap \supp(x)}(x) \in \mathcal P$, where the last membership is due to the facts that $\Ical \cap \supp(x) \subseteq \supp(x)$ and that $\mathcal P$ is CP admissible.
\end{proof}

\begin{example} \rm (\tblue{Examples of CP admissible sets})
Examples of bounded CP admissible sets include $\{ x \in \mathbb R^N \, | \, a^T x \le 1, \mbox{ and } x \ge 0 \}$ for a vector $a\in \mathbb R^N_{++}$, and any $\ell_p$-ball $\{ x \in \mathbb R^N \, | \, \| \| x \|_p \le \varepsilon \}$ with $p > 0$ and $\varepsilon > 0$, and $\mathcal P=[a_1, b_1]\times [a_2, b_2] \times \cdots \times [a_N, b_N]$ where $a_i \le 0 \le b_i$ for each $i$. Examples of unbounded CP admissible sets include $\mathbb R^N$, $\mathbb R^N_+$, and $\Sigma_K=\{ x \in \mathbb R^N \, | \, \| x \|_0 \le K \}$ for some $K \in \mathbb N$.  Note that $\Sigma_K$ and the $\ell_p$-ball with $0< p <1$ are non-convex.
\tblue{Other examples of non-convex CP admissible sets include $\mathcal P=\mathbb R^N_+\cup \mathbb R^N_-$ and the constraint set $\Pcal$ given in Example~\ref{example:nonconvex_constraint}.}
%
%, which is CP admissible but non-convex.
%
A CP admissible set may be neither open nor closed, e.g., $\mathcal P=[0, 1)\times (-1, 2]$ in $\mathbb R^2$.
\end{example}

%%%\gap

%%{\bf Properties of CP admissible sets}:
The following proposition provides a list of important properties of CP admissible sets.

\begin{proposition} \label{prop:properties_CP_admissible}
 The following hold:
 \begin{itemize}
  \item[(i)] The set $\mathcal P$ is CP admissible if and only if $\lambda \mathcal P$ is CP admissible for any real number $\lambda \ne 0$, and the intersection and union of CP admissible sets are CP admissible;
  \item [(ii)] The algebraic sum of two CP admissible sets is CP admissible;
  \item [(iii)] If $\mathcal P$ is CP admissible, then for any index set $\Ical$, $\pi_\Ical(\mathcal P)$ is also CP admissible;
  \item [(iv)] If $\mathcal P$ is a convex and CP admissible set, then $\dim(\mathcal P)=\max\{ |\supp(x)| \, : \, x \in \mathcal P \}$.
\end{itemize}
\end{proposition}

\begin{proof}
 (i) This is a direct consequence of the definition of a CP admissible set.

 (ii) Let $\mathcal P_1$ and $\mathcal P_2$ be two CP admissible sets, and $z$ be an arbitrary vector in $\Pcal_1+\Pcal_2$. Hence, $z=x+y$, where $x\in \Pcal_1$ and $y\in \Pcal_2$. For any index set $\Ical$, it follows from Lemma~\ref{lem:CP_adm_geometry} that $\pi_\Ical(x) \in \Pcal_1$ and $\pi_\Ical(y) \in \Pcal_2$. Therefore, $\pi_\Ical(z)= \pi_\Ical(x) + \pi_\Ical(y) \in \mathcal P_1+\mathcal P_2$. By Lemma~\ref{lem:CP_adm_geometry} again, we deduce that $\mathcal P_1+\Pcal_2$ is CP admissible.
%
% Hence, $\pi_\Ical(\mathcal P_1 + \mathcal P_2) \subseteq \mathcal P_1 + \mathcal P_2$ for any index set %$\Ical$. Using Lemma~\ref{lem:CP_adm_geometry} again, we deduce that $\mathcal P_1+\Pcal_2$ is CP %admissible.
%

 (iii) Let $\mathcal P$ be CP admissible, and $\Ical$ be an arbitrary but fixed index set. Then for any index set $\Jcal$, we deduce via equation (\ref{eqn:pi_two_index sets}) that $\pi_\Jcal (\pi_\Ical(\Pcal)) = \pi_{\Ical} (\pi_\Jcal(\Pcal))$. Since $\Pcal$ is CP admissible, $\pi_\Jcal(\Pcal) \subseteq \Pcal$. Hence, by Lemma~\ref{lem:CP_adm_geometry}, we have $\pi_{\Ical} (\pi_\Jcal(\Pcal))
  \subseteq \pi_\Ical(\Pcal)$. This shows that $\pi_\Ical(\Pcal)$ is CP admissible.

  (iv) Suppose $\mathcal P$ is a convex and CP admissible set. Let $\wh x\in \mathcal P$ be such that $|\supp(\wh x)| \ge |\supp(x)$ for all $x \in \mathcal P$. We claim that for any $x \in \mathcal P$, $\supp(x) \subseteq \supp(\wh x)$. Suppose not. Then there exist a point $x' \in \mathcal P$ and an index $i\in \supp(x')$ such that $i\notin \supp(\wh x)$. Since $\mathcal P$ is convex, $z(\lambda):=\lambda x' + (1-\lambda) \wh x \in \mathcal P$ for all $\lambda \in [ 0, 1]$. However, for all $\lambda>0$ sufficiently small, $(\supp(\wh x) \cup \{ i \}) \subseteq \supp(z(\lambda))$. This shows that $|\supp(z(\lambda))|> |\supp(\wh x)|$, leading to a contradiction. Therefore,  $\supp(x) \subseteq \supp(\wh x)$ for all $x \in \mathcal P$. %%Define the index set $\wh\Jcal := \supp(\wh x)$.
  Furthermore, it is known that $\dim(\mathcal P) = \dim(\mbox{aff}(\mathcal P))$, where $\mbox{aff}(\cdot)$ denotes the affine hull of a set. Since $\Pcal$ contains the zero vector, $\mbox{aff}(\mathcal P)=\mbox{span}(\Pcal)$.
  In view of the claim that  $\supp(x) \subseteq \supp(\wh x)$ for any $x \in \mathcal P$,
  we deduce that $\dim(\mathcal P) =\dim(\mbox{span}(\Pcal))\le |\supp(\wh x)|$.
   Letting $p:=|\supp(\wh x)|$, we assume without loss of generality that $\supp(\wh x)=\{1, \ldots, p\}$. For each $s \in \{1, \ldots, p\}$, let $\wh\Jcal_s:=\{1, 2, \ldots, s\}$ and $z^s:=(\wh x_{\wh\Jcal_s}, 0)$. Therefore, $z^p = \wh x$. Since $\Pcal$ is CP admissible, each $z^s \in \mathcal P$. Besides, $\{z^1, z^2, \ldots, z^p \}$ is linearly independent. Since $\mathcal P$ is convex and $\{ 0, z^1, z^2, \ldots, z^p\}$ is affinely independent, the convex hull of $\{ 0, z^1, z^2, \ldots, z^p\}$ is a simplex  of dimension $p$ and is contained in $\Pcal$. Therefore, it follows from \cite[Theorem 2.4]{Rockafellar_book70} that $\dim(\mathcal P) \ge p=|\supp(\wh x)|$. Consequently, $\dim(\mathcal P)= |\supp(\wh x)|$.
\end{proof}

Using  (iv) of Proposition~\ref{prop:properties_CP_admissible}, we see that the hyperplane $\mathcal P =\{ x \in \mathbb R^N \, | \, d^T x = 0 \}$ with each $d_i \ne 0$ given in Example~\ref{example:counter01}  is {\em not} CP admissible, since $\dim(\mathcal P)=N-1$ but $\max\{ |\supp(x)| :  x \in \mathcal P \}=N$.

%
%It is worth mentioning that an algebraic sum of two (even convex) CP admissible sets need {\em not} be CP %admissible. Example.
%
%
%By using statement (3) above, we see that $\mathcal P:=\{ x \in \mathbb R^N \, | \, \mathbf 1^T x = \gamma %\}$, where $\gamma\in \mathbb R$, is convex but {\em not} CP admissible, since $\dim(\mathcal P)=N-1$ but %$\max\{ |\supp(x)| \, : \, x \in \mathcal P \}=N$.
%
%\gap
%

\begin{lemma} \label{lem:CP_adm_closedness}
Let $\mathcal P$ be a closed and CP admissible set. Then for any index set $\Jcal$, $\pi_\Jcal(\mathcal P)$ is closed. %%%%(This result is closely related to a generate one using a decomposition of $\mathcal P$.)
\end{lemma}

%
%Moreover, if $\mathcal P$ is CP admissible, then $\pi_\Jcal(\mathcal P)$ is also CP admissible.
%

\begin{proof}
 Fix an index set $\Jcal$.  Let $(z^k)$ be a convergent sequence in $\pi_\Jcal(\mathcal P)$ such that $(z^k) \rightarrow z^*$. Hence, for each $k$, $z^k=(z^k_\Jcal, z^k_{\Jcal^c}) \in \pi_\Jcal(\mathcal P)$ with $z^k_{\Jcal^c}=0$. Since $(z^k)$ converges to $z^*$, we have  $z^*=(z^*_\Jcal, 0)$ and $(z^k_\Jcal) \rightarrow z^*_\Jcal$.
%
% Furthermore,  for each $k$, there exists $x^k\in \mathcal P$ such that $\pi_\Jcal(x^k)=z^k$.
%
 Since $\mathcal P$ is CP admissible, $\pi_\Jcal(\mathcal P) \subseteq \mathcal P$ such that $z^k \in \mathcal P$ for each $k$. Further, since $\mathcal P$ is closed, we have $z^* \in \mathcal P$. Clearly, $\pi_\Jcal(z^*) = z^*\in \mathcal P$. Hence, $z^* \in \pi_\Jcal(\mathcal P)$. This shows that $\pi_\Jcal(\mathcal P)$ is closed.
\end{proof}

Note that the above result may fail when $\mathcal P$ is not CP admissible, even if it is closed and convex. For example, consider $\mathcal P=\{x=(x_1, x_2) \, | \, x_2 \ge \frac{1}{x_1}, \ x_1 > 0 \} \subset \mathbb R^2$. Clearly, $\mathcal P$ is closed and convex but not CP admissible. Letting $\Jcal=\{1\}$, we see that  $\pi_\Jcal(\mathcal P) = \{ (x_1, 0) \, | \, x_1 \in (0, \infty) \}$ and thus is not closed.

%%\gap

The following result gives a complete characterization of a closed, convex and CP admissible cone. Particularly, it shows that a closed, convex and CP admissible  cone is a Cartesian product of Euclidean spaces and nonnegative or nonpositive orthants.

\begin{proposition} \label{prop:CP_admissible_cone}
 Let $\mathcal C$ be a closed convex cone in $\mathbb R^N$. Then $\mathcal C$ is CP admissible if and only if there exist four disjoint index subsets $\mathcal I_1$, $\mathcal I_+$, $\mathcal I_-$, and $\mathcal I_0$ (some of which can be empty) whose union is $\{1, \ldots, N\}$ such that $\mathcal C = \mathbb R^{\Ical_1}+ (\mathbb R_+)^{\Ical_+} + (\mathbb R_-)^{\Ical_-} + \{0\}^{\Ical_0}$ or equivalently $\mathcal C = \mathbb R_{\Ical_1} \times (\mathbb R_+)_{\Ical_+} \times (\mathbb R_-)_{\Ical_-} \times \{ 0 \}_{\Ical_0}$.
\end{proposition}

\begin{proof}
 ``If''. Suppose $\mathcal C = \mathbb R^{\Ical_1}+ \mathbb R^{\Ical_+}_+ + \mathbb R^{\Ical_-}_- + \{0\}^{\Ical_0}$, where  the four index sets $\mathcal I_1$, $\mathcal I_+$, $\mathcal I_-$, and $\mathcal I_0$ form a disjoint union of $\{1, \ldots, N\}$. It is easy to see that $\mathcal C$ is closed and convex and that  $\mathbb R^{\Ical_1}$, $\mathbb R^{\Ical_+}_+$, $\mathbb R^{\Ical_-}_-$ and $\{0\}^{\Ical_0}$ are all CP admissible. By (ii) of Proposition~\ref{prop:properties_CP_admissible}, $\mathcal C$ is also CP admissible.

%
% Using the notation $\mathcal U_\Ical$ defined at the beginning of this section, $\mathcal C$ can be written %a Cartesian product of Euclidean spaces and nonnegative or nonpositive orthants, i.e., $\mathcal C = \mathbb %R_{\Ical_1} \times (\mathbb R_+)_{\Ical_+} \times (\mathbb R_-)_{\Ical_-} \times \{ 0 \}_{\Ical_0}$. It is %easy to see that $\mathcal C$ is closed and convex.
%

%
% It is easy to see that $\mathcal C$ is closed and convex. We show that $\mathcal C$ is CP admissible as %follows.
%
%
%  For each $x \in \mathcal P$ and $\Jcal \subseteq \supp(x)$, let $\Jcal_1:=\Jcal \cap \Ical_1$, %$\Jcal_+:=\Jcal \cap \Ical_+$, $\Jcal_-:=\Jcal \cap \Ical_-$, and we notice that $\Jcal \cap %\Ical_0=\emptyset$. Hence, $\Jcal_1, \Jcal_+, \Jcal_-$ are disjoint and $\Jcal=\Jcal_1 \cup \Jcal_+ \cup %\Jcal_-$. Hence, $x_\Jcal = (x_{\Jcal_1}, x_{\Jcal_+}, x_{\Jcal_-})$. Further, $(x_{\Jcal_1}, 0, 0, 0) \in %\mathbb R^{\Ical_1}$, $(0, x_{\Jcal_+}, 0, 0) \in \mathbb R^{\Ical_+}$, and $(0, 0, x_{\Jcal_-}, 0) \in %\mathbb R^{\Ical_-}$. Therefore, $\pi_\Jcal(x)= (x_{\Jcal_1}, 0, 0, 0) + (0, x_{\Jcal_+}, 0, 0)+(0, 0, %x_{\Jcal_-}, 0) \in \mathcal C$. As a result, $\mathcal C$ is CP admissible.
%

``Only If''. Let $\mathcal C$ be a closed convex cone which is CP admissible. For an arbitrary index $i\in \{1, \ldots, N\}$, let $\pi_i(\mathcal C):=\{ \pi_i(x) \, | \, x \in \mathcal C\} \subseteq \mathbb R^N$ and $[\pi_i(\mathcal C)]_i:=\{ \big(\pi_i(x))_i \, | \, x \in \mathcal C\} \subseteq \mathbb R$. Since $\mathcal C$ is a closed convex cone, it is easy to show via a similar argument for Lemma~\ref{lem:CP_adm_closedness} that $[\pi_i(\mathcal C)]_i$ is a closed convex cone in $\mathbb R$. This implies that $[\pi_i(\mathcal C)]_i$ equals either one of the following (polyhedral) cones in $\mathbb R$: $\mathbb R$, $\mathbb R_+$, $\mathbb R_-$, or $\{ 0 \}$. Define the index sets $\mathcal I_1:=\{ i \, | \, [\pi_i(\mathcal C)]_i = \mathbb R\}$, $\mathcal I_+:=\{ i \, | \, [\pi_i(\mathcal C)]_i = \mathbb R_+\}$, $\mathcal I_-:=\{ i \, | \, [\pi_i(\mathcal C)]_i = \mathbb R_-\}$, and $\mathcal I_0:=\{ i \, | \, [\pi_i(\mathcal C)]_i = \{ 0 \} \}$. Clearly, these index sets form a disjoint union of $\{1, \ldots, N\}$. Furthermore, since $\mathcal C$ is CP admissible, we have $\mathbb R^{\Ical_1} \subseteq \mathcal C$, $\mathbb R^{\Ical_+}_+ \subseteq \mathcal C$, $\mathbb R^{\Ical_-}_- \subseteq \mathcal C$, and $\{ 0 \}^{\Ical_0} \subseteq \mathcal C$. Since $\mathcal C$ is a convex cone, $ \mathbb R^{\Ical_1}+ (\mathbb R_+)^{\Ical_+} + (\mathbb R_-)^{\Ical_-} + \{0\}^{\Ical_0} \subseteq \mathcal C$. Conversely, for any $x \in \mathcal C$, it follows from the definition of $[\pi_i(\mathcal C)]_i$ and the disjoint property of the index sets $\Ical_1, \Ical_+, \Ical_-$ and $\Ical_0$ that $x \in \mathbb R^{\Ical_1}+ (\mathbb R_+)^{\Ical_+} + (\mathbb R_-)^{\Ical_-} + \{0\}^{\Ical_0}$. This shows that $\mathcal C = \mathbb R^{\Ical_1}+ (\mathbb R_+)^{\Ical_+} + (\mathbb R_-)^{\Ical_-} + \{0\}^{\Ical_0}$.
\end{proof}

\mycut{
\begin{remark} \rm \label{remark:cone_case}
Using the notation $\mathcal U_\Ical$ defined at the beginning of this section, the above proposition shows that a closed, convex and CP admissible  cone $\mathcal C$ can be written a Cartesian product of Euclidean spaces and nonnegative or nonpositive orthants, i.e., $\mathcal C = \mathbb R_{\Ical_1} \times (\mathbb R_+)_{\Ical_+} \times (\mathbb R_-)_{\Ical_-} \times \{ 0 \}_{\Ical_0}$. Hence, for any $x \in \mathcal C$, $A x = A_{\bullet \Ical_1} x_{\Ical_1} + A_{\bullet \Ical_+} x_{\Ical_+} +  A_{\bullet \Ical_-} x_{\Ical_-}$. Consequently, without loss of generality, we may assume that $\mathcal C= \mathbb R_{\Ical_1} \times (\mathbb R_+)_{\mathcal I_+} \times (\mathbb R_-)_{\Ical_-}$ through the rest of the development.
%
%, where $-x_{\Ical_-} \ge 0$. Consequently, without loss of generality, we may assume that $\mathcal C= %\mathbb R_{\Ical_1} \times (\mathbb R_+)_{\mathcal I_+}$ by removing the redundant zero part and changing the %sign of $A_{\bullet\Ical_-}$.
%
\end{remark}
}

The next proposition presents a decomposition of a closed, convex and CP admissible set.

%
%{\bf Conjecture}: {\it A closed, convex, and CP admissible set $\mathcal P = \mathcal W + \mathcal K$, where %$\mathcal W$ is a compact convex and CP admissible set and $\mathcal K$ is a closed convex and CP admissible %cone.} \\
%

\begin{proposition} \label{prop:CP_adm_decomposition}
Let $\Pcal \subseteq \mathbb R^N$ be closed, convex and CP admissible. Then  $\Pcal = \mathcal W + \mathcal K$, where $\mathcal W\subseteq \Pcal$ is a compact, convex and CP admissible set, and $\mathcal K\subseteq \Pcal$ is a closed, convex and CP admissible cone.
\end{proposition}

\begin{proof}
  For a given closed, convex and CP admissible set $\mathcal P$, we first construct a compact, convex and CP admissible set $\mathcal W$ contained in $\mathcal P$. It follows from the similar argument for Lemma~\ref{lem:CP_adm_closedness} and Proposition~\ref{prop:CP_admissible_cone} that for each $i \in \{1, \ldots, N\}$, $[\pi_{i} (\Pcal)]_i$ is a closed convex set in $\mathbb R$ which  contains $0$ . Hence, each $[\pi_{i} (\Pcal)]_i$ must be in one of the following forms: $\mathbb R$,   $[a_i, \infty)$ with $a_i \le 0$, $(-\infty, b_i]$ with $b_i \ge 0$, and  $[a_i, b_i]$ with $a_i\le 0 \le b_i$, where in the last case, $a_i=b_i=0$ if $a_i=b_i$. These four forms respectively correspond to an unbounded set without lower and upper bounds, an unbounded set that is bounded from below, an unbounded set that is bounded frow above, and a bounded set. Define the following disjoint index sets whose union is $\{1, \ldots, N\}$:
  \begin{align*}
     \Ical_1 & \, := \, \{ i \, | \, [\pi_{i} (\Pcal)]_i = \mathbb R\}, \qquad \qquad \quad  \Ical_+ \, := \, \{ i \, | \, [\pi_{i} (\Pcal)]_i \mbox{ is unbounded but bounded from below } \}, \\
     \Ical_0 & \, := \, \{ i \, | \, [\pi_{i} (\Pcal)]_i \mbox{ is bounded} \}, \qquad
      \Ical_ -  \, := \, \{ i \, | \, [\pi_{i} (\Pcal)]_i \mbox{ is unbounded but bounded from above } \}.
  \end{align*}
  Define the closed convex cone $\mathcal K:=\mathbb R^{\Ical_1}+ (\mathbb R_+)^{\Ical_+} + (\mathbb R_-)^{\Ical_-} + \{0\}^{\Ical_0}$. Since $\Pcal$ is CP admissible and convex, we have $\mathcal K \subseteq \Pcal$.
  Further, $\mathcal K$ is CP admissible in view of Proposition~\ref{prop:CP_admissible_cone}. Moreover, define the set
  \begin{equation} \label{eqn:def_W_set}
     \mathcal W \, := \, \Pcal \cap \, \underbrace{ \Big \{ \, x =(x_{\Ical_1}, x_{\Ical_+}, x_{\Ical_-}, x_{\Ical_0} ) \, | \,  x_{\Ical_1} =0, \ x_{\Ical_+} \le 0, \ x_{\Ical_-} \ge 0 \, \Big\} }_{:= \, \mathcal C}.
  \end{equation}
  Clearly, $\mathcal W \subseteq \Pcal$. Since the set $\mathcal C$ defined in (\ref{eqn:def_W_set})
  is closed and convex, $\mathcal W$ is also closed and convex. We show next that $\mathcal W$ is bounded and CP admissible. To proved the boundedness of $\mathcal W$,
  %
  % it is sufficient to show that $\pi_{\Ical_+ \cup \Ical_-}(\mathcal W)$ is bounded.
  %
   recall that (i) for each $i\in \Ical_+$, $[\pi_{i} (\Pcal)]_i=[a_i, \infty)$ for some $a_i \le 0$; (ii) for  each $i\in \Ical_-$, $[\pi_{i} (\Pcal)]_i=(-\infty, b_i]$ for some $b_i \ge 0$; and (iii) for each $i \in \Ical_0$, $[\pi_{i} (\Pcal)]_i=[a_i, b_i]$ for some $a_i \le 0 \le b_i$. Hence, $\pi_i(\mathcal W)=\{0\}$ for each $i\in \Ical_1$,    $\pi_i (\mathcal W) \in [a_i, 0]$ for each $i\in \Ical_+$, $\pi_i (\mathcal W) \in [0, b_i]$ for each $i\in \Ical_-$, and $\pi_i (\mathcal W) \in [a_i, b_i]$ for each $i\in \Ical_0$. Therefore, for each $x \in \mathcal W$, we have  $\| x \|_1 = \| x_{\Ical_+}\|_1 + \| x_{\Ical_-}\|_1 + \| x_{\Ical_0}\|_1\le \sum_{i\in \Ical_+}|a_i| + \sum_{i\in \Ical_-}|b_i| +\sum_{i\in \Ical_0} \max(|a_i|, b_i)$. This shows that $\mathcal W$ is bounded and thus compact. Lastly, it is easy to see that the set $\mathcal C$ defined in (\ref{eqn:def_W_set})
  %
  %on the right hand side of the definition of $\mathcal W$
  %
  is CP admissible. Since $\Pcal$ is CP admissible, by statement (i) of Proposition~\ref{prop:properties_CP_admissible}, $\mathcal W$ is also CP admissible.

  We show that $\mathcal P = \mathcal W + \mathcal K$ as follows. We first show that $\mathcal W + \mathcal K \subseteq \mathcal P$. Consider an arbitrary $z \in \mathcal W + \mathcal K$, i.e., $z=x+y$ with $x \in \mathcal W$ and $y \in \mathcal K$. Since $\mathcal W$ and $\mathcal K$ are both contained in the convex set $\Pcal$ and since $\mathcal K$ is a cone, we see that for any $\lambda \in [0, 1)$,
  \[
    \lambda x + y = \lambda x + (1-\lambda) \frac{y}{1-\lambda} \in \mathcal P.
  \]
  Furthermore, since $\Pcal$ is closed, $x+y=\lim_{\lambda \uparrow 1} \big (\lambda x + y \big) \in \mathcal P$. This shows that $z \in \Pcal$ and thus $\mathcal W + \mathcal K \subseteq \Pcal$. We finally show that $\Pcal \subseteq \mathcal W + \mathcal K$.  Toward this end, consider an arbitrary $z =(z_1, \ldots, z_N)^T \in \Pcal$, and define the vectors $x=(x_1, \ldots, x_N)^T$ and $y=(y_1, \ldots, y_N)^T$ as follows:
  \[
      x_i \, := \, \left\{ \begin{array}{lcc} 0 & \mbox{ if } \ \ i \in \Ical_1 \\ -(z_i)_- & \mbox{ if } \ \ i \in \Ical_+ \\  (z_i)_+ & \mbox{ if } \ \ i \in \Ical_- \\ z_i & \mbox{ if } \ \ i \in \Ical_0         \end{array} \right., \qquad \quad
      y_i \, := \, \left\{ \begin{array}{lcc} z_i & \mbox{ if } \ \ i \in \Ical_1 \\ (z_i)_+ & \mbox{ if } \ \ i \in \Ical_+ \\  -(z_i)_- & \mbox{ if } \ \ i \in \Ical_- \\ 0 & \mbox{ if } \ \ i \in \Ical_0         \end{array} \right..
  \]
  Clearly, $z=x+y$, $y \in \mathcal K$, and $x \in \mathcal C$, where $\mathcal C$ is defined in (\ref{eqn:def_W_set}). Moreover, letting the index set $\Jcal := \{ i \in \Ical_+ \, | \, z_i < 0 \} \cup \{ i \in \Ical_- \, | \, z_i > 0 \} \cup \Ical_0$, we have $x=\pi_\Jcal(z)$. Since $\Pcal$ is CP admissible, it follows from Lemma~\ref{lem:CP_adm_geometry} that $x \in \Pcal$, leading to $x\in \mathcal W$. This shows that $z \in \mathcal W + \mathcal K$, and thus $\Pcal \subseteq \mathcal W + \mathcal K$.
\end{proof}

The above proposition shows that $\mathcal K$ is the asymptotic cone (or recession cone) of $\mathcal P$.
%
%\cite{AuslenderT_book02}.
%
Furthermore,
by using this proposition, we show the existence of an optimal solution of the underlying minimization problem given in Line 7 of Algorithm~\ref{algo:constrained_MP} for an arbitrary index set $\Jcal$ as follows.

\begin{corollary} \label{coro:sol_existence_CP_adm}
  Let $\mathcal P \subseteq \mathbb R^N$ be a closed, convex and CP admissible set. Then for any matrix $A \in \mathbb R^{m \times N}$, any index set $\mathcal J \subseteq \{1, \ldots, N\}$, and any $y \in \mathbb R^m$, $\min_{w \in \Pcal, \supp(w) \subseteq \Jcal} \| A w - y\|^2_2$ attains an optimal solution.
\end{corollary}

\begin{proof}
 We first show that $A \mathcal P$ is a closed set for any matrix $A \in \mathbb R^{m \times N}$. It follows from Proposition~\ref{prop:CP_adm_decomposition} that $A \Pcal = A \mathcal W + A \mathcal K$, where $\mathcal W$ is compact and $\mathcal K$ is a polyhedral cone.  Note that $A \mathcal W$ is compact, and $A \mathcal K$ is a polyhedral cone and thus is closed. This implies that $A \Pcal$ is closed. The desired result thus follows readily from statement (i) of Lemma~\ref{lemma:sol_existence}.
\end{proof}

%
%{\bf Under this conjecture, $H \mathcal P$ is closed for any matrix $H$, which is critical to the solution %existence of some optimization problem.}
%

In what follows, we let $\mbox{cone}(\mathcal U)$ denote the conic hull of a nonempty set $\mathcal U$ in $\mathbb R^N$, i.e., $\mbox{cone}(\mathcal U)$ is the collection of all nonnegative combinations of finitely many vectors in $\mathcal U$.
%
% or $\mbox{cone}(\mathcal U)=\{ \sum^k_{i=1} \lambda_i \cdot x^i \, | \, k \in \mathbb N, x^i \in \mathcal U, %\lambda_i \ge 0,  \forall \, i=1, \ldots, k \}$.
%

\begin{proposition} \label{prop:conic_hull}
  Let $\Pcal$ be a closed, convex and CP admissible set in $\mathbb R^N$. Then $\mbox{cone}(\Pcal)=\{ \lambda x \, | \, \lambda \ge 0, x \in \Pcal \}$, and $\mbox{cone}(\Pcal)$ is a closed, convex and CP admissible cone.
\end{proposition}

\begin{proof}
 Since $\Pcal$ is a convex set, it follows from a standard argument in convex analysis, e.g., \cite[Corollary 2.6.3]{Rockafellar_book70}, that $\mbox{cone}(\Pcal)=\{ \lambda x \, | \, \lambda \ge 0, x \in \Pcal \}$.
Define the disjoint index sets whose union is $\{1, \ldots, N\}$:
  \begin{align}
     \Lcal_1 & \, := \, \{ i \, | \ \mbox{$0$ is in the interior of } [\pi_{i} (\Pcal)]_i \}, \qquad \qquad
      \Lcal_0  \, := \, \{ i \, | \ [\pi_{i} (\Pcal)]_i =\{ 0 \}  \}, \notag \\
       \Lcal_+ & \, := \, \{ i \, | \ \inf [\pi_{i} (\Pcal)]_i =0, \,  \mbox{ and $[\pi_{i} (\Pcal)]_i$ contains a positive number }   \},  \label{eqn:conic_hll_indices} \\
      \Lcal_ -  & \, := \, \{ i \, | \ \sup [\pi_{i} (\Pcal)]_i =0, \,  \mbox{ and $[\pi_{i} (\Pcal)]_i$ contains a negative number } \}.  \notag
  \end{align}
  Let  $\mathcal C:=\mathbb R^{\Lcal_1}+ (\mathbb R_+)^{\Lcal_+} + (\mathbb R_-)^{\Lcal_-} + \{0\}^{\Lcal_0}$. In view of Proposition~\ref{prop:CP_admissible_cone}, $\mathcal C$ is a closed, convex and CP admissible cone. In what follows, we show that $\mathcal C=\mbox{cone}(\Pcal) $ in two steps.

   (i) We first show that $\mbox{cone}(\Pcal) \subseteq \mathcal C$.
  For a given $x \in \Pcal$, we write it as $x=(x_{\Lcal_1}, x_{\Lcal_0}, x_{\Lcal_+}, x_{\Lcal_-})$. Hence, $x=\pi_{\Lcal_1}(x)  + \pi_{\Lcal_+}(x) + \pi_{\Lcal_-}(x) + \pi_{\Lcal_0}(x)$, where $\pi_{\Lcal_1}(x) \in \mathbb R^{\Lcal_1}$, $\pi_{\Lcal_+}(x) \in (\mathbb R_+)^{\Lcal_+}$, $\pi_{\Lcal_-}(x) \in (\mathbb R_-)^{\Lcal_-}$, and $\pi_{\Lcal_0}(x) = 0 \in  \{0\}^{\Lcal_0}$. By the definition of $\mathcal C$, we have that $x \in \mathcal C$. Therefore, $\Pcal \subseteq \mathcal C$. Since $\mbox{cone}(\Pcal)$ is the smallest convex cone containing $\Pcal$, we have  $\mbox{cone}(\Pcal) \subseteq \mathcal C$.
%
%  . By the definition of $\mathcal C$, each vector on the right side belongs to $\mathcal C$. Since $\mathcal %C$ is a convex cone, the sum of the vectors on the right hand side is also in $\mathcal C$, yielding $x \in %\mathcal C$. Therefore, $\Pcal \subseteq \mathcal C$. Since $\mbox{cone}(\Pcal)$ is the smallest convex cone %containing $\Pcal$, we have  $\mbox{cone}(\Pcal) \subseteq \mathcal C$.
  %

  (ii) We next show that $\mathcal C \subseteq \mbox{cone}(\Pcal)$. Consider a vector $x \in \mathbb R^{\Lcal_1}$, where $x=(x_{\Lcal_1}, x_{\Lcal^c_1})=(x_{\Lcal_1}, 0)$.
  %
  %If $x=0$, then $x \in \mbox{cone}(\Pcal)$. Now suppose $x \ne 0$.
  %
  By the definition of the index set $\mathcal L_1$ given in (\ref{eqn:conic_hll_indices}), we see that there exists a sufficiently small positive number $\lambda$ such that $\lambda x_i \in [\pi_i(\Pcal)]_i$ for each $i \in \Lcal_1$.
   Let $v=(v_{\Lcal_1}, v_{\Lcal^c_1})$ with $v_{\Lcal_1} := \lambda x_{\Lcal_1}$ and $v_{\Lcal^c_1}:=0$. Hence, $v\in \pi_{\Lcal_1}(\Pcal)$.
   Since $\Pcal$ is CP admissible,  $\pi_{\Lcal_1}(\Pcal) \subseteq \Pcal$ such that $v \in \Pcal$. In view of  $x = (1/\lambda) v$ and $\mbox{cone}(\Pcal)=\{ \lambda x \, | \, \lambda \ge 0, x \in \Pcal \}$, we deduce that $x \in \mbox{cone}(\Pcal)$. Therefore, $\mathbb R^{\Lcal_1} \subseteq \mbox{cone}(\Pcal)$. It follows from a similar argument that $\mathbb R^{\Lcal_+}_+\subseteq \mbox{cone}(\Pcal)$, $\mathbb R^{\Lcal_-}_-\subseteq \mbox{cone}(\Pcal)$, and $\{0\}^{\Lcal_0} \subseteq \mbox{cone}(\Pcal)$. Since $\mbox{cone}(\Pcal)$ is convex, we see that $\mathbb R^{\Lcal_1}+ \mathbb R^{\Lcal_+}_+ + \mathbb R^{\Lcal_-}_- + \{0\}^{\Lcal_0} \subseteq \mbox{cone}(\Pcal)$. Hence, $\mathcal C \subseteq \mbox{cone}(\Pcal)$.

 Consequently, $\mathcal C = \mbox{cone}(\Pcal)$. Finally, since $\mathcal C$ is closed and CP admissible, so is $\mbox{cone}(\Pcal)$.
%
%
%   Let $\mathcal C'$ be a convex cone containing $\Pcal$. Since $\Pcal$ contains the zero vector,
%
%   $\mathcal C'$ contains $\mathbb R^{\Lcal_1}$,  $\mathbb R^{\Lcal_+}_+$, $\mathbb R^{\Lcal_-}_-$, and %$\{0\}^{\Lcal_0}$ ({\bf more ...}). Since $\mathcal C'$ is convex, we have $\mathbb R^{\Lcal_1}+ \mathbb %R^{\Lcal_+}_+ + \mathbb R^{\Lcal_-}_- + \{0\}^{\Lcal_0} \subseteq \mathcal C'$. This shows that $\mathcal C %\subseteq \mathcal C'$.  Therefore, $\mathcal C$ is the smallest convex cone containing $\Pcal$. This implies %that $\mathcal C=\mbox{cone}(\Pcal)$. \\
%
% The convexity of $\mbox{cone}(\Pcal)$ is trivial. We then show that $\mbox{cone}(\Pcal)$ is CP admissible. %Consider an arbitrary nonzero vector $z \in \mbox{cone}(\Pcal)$. Since $\mbox{cone}(\Pcal)= \{ \lambda x \, | %\, \lambda \ge 0, x \in \Pcal \}$, we have $z = \lambda \cdot x$ for some $\lambda>0$ and $x \ne 0$. %Therefore, for any index set $\Jcal \subset \supp(z)$, we have $\Jcal \subset \supp(x)$ and $\pi_\Jcal(z) = %\lambda \cdot \pi_\Jcal(x) = \lambda \cdot ( x_\Jcal, 0)$. Since $\Pcal$ is CP admissible, $ ( x_\Jcal, 0) %\in \Pcal$. This shows that $\pi_\Jcal(z) \in \mbox{cone}(\Pcal)$.
%
\end{proof}

Note that if $\Pcal$ is not CP admissible (even though closed and convex), its conic hull may {\em not} be closed in general. An example is the closed unit $\ell_2$-ball in $\mathbb R^N$ centered at $\mathbf e_1 \in \mathbb R^N$.

%
% although the closure of its conic hull is CP admissible.
%{\bf Hence the positive homogeneity based technique can be extended to the closure of a conic hull of a %non-CP admissible set, as long as this closure (which is a closed convex cone) is CP admissible.}
%
%
%{\bf Additional facts}: (1) $\mbox{cone}(\Pcal)=T_\Pcal(0)$, where $T_\Pcal(z)$ denotes the tangent cone of %$\Pcal$ at $z$, i.e., $T_\Pcal(z) = \mbox{cl}\{ d \in \mathbb R^N \, | \, \mbox{there exists $\lambda>0$ such %that $z+ \lambda d \in \Pcal$} \, \}$; (2) if $\Pcal=\mathcal W + \mathcal K$ (following the decomposition of %$\Pcal$), then $\mbox{cone}(\Pcal) = \mbox{cone}(\mathcal W) + \mathcal K$.
%
%\gap

\begin{definition} \rm \label{def:irreducible_CP_set}
 A closed, convex and CP admissible set $\Pcal$ is {\em irreducible} if the index set $\{ i \, | \, [\pi_{i} (\Pcal)]_i =\{ 0 \}  \}$ is the empty set.
\end{definition}
In light of Proposition~\ref{prop:conic_hull}, it is easy to see that a closed, convex and CP admissible set $\Pcal$ is irreducible  if and only if $\mbox{cone}(\Pcal)$ is irreducible.

\gap

The above development shows that the class of CP admissible sets enjoy favorable properties indicated at the beginning of this section. For example, each CP admissible set contains sufficiently many sparse vectors due to the CP admissible property. Moreover, $\mathbb R^N$, $\mathbb R^N_+$ and their alikes belong to the class of CP admissible sets. In what follows, we show an additional important implication of CP admissible sets in Proposition~\ref{prop:negative_second_term}, which is crucial to the development of sufficient conditions for uniform exact recovery in Section~\ref{sect:suff_cond_exact_recovery}.
%
%given in Proposition~\ref{prop:negative_second_term} below. To prove this proposition,
%
To this end, we first present a technical result on the support of vectors.

\begin{lemma} \label{lem:support_result}
 Let $u, v \in \mathbb R^N$ and $\Jcal\subseteq \{1, \ldots, N\}$ be such that $\supp(v) \subseteq \Jcal \subseteq \supp(u)$. Then $\supp(u-v)\setminus \Jcal \, = \, \supp(u) \setminus \Jcal$.
\end{lemma}

\begin{proof}
   We show $\supp(u-v) \subseteq \supp(u)$ first. Let $i \in \supp(u-v)$. Hence, $u_i - v_i \ne 0$. We claim that $u_i \ne 0$, because otherwise, $u_i=0$ and $v_i \ne 0$, which implies $i\in \supp(v) \subseteq \supp(u)$, yielding a contradiction.  Hence, $\supp(u-v) \subseteq \supp(u)$. This leads to $\supp(u-v)\setminus \Jcal \subseteq \supp(u)\setminus \Jcal$.
  Conversely, for any $i \in \supp(u)\setminus \Jcal$, we have $v_i =0$ (due to $\supp(v) \subseteq \Jcal$) so that $(u-v)_i = u_i \ne 0$. Hence, $i \in \supp(u-v)$. Since $i\notin\Jcal$, we have $i \in  \supp(u-v)\setminus \Jcal$. Therefore,  $\supp(u)\setminus \Jcal \subseteq \supp(u-v)\setminus \Jcal$. As a result, $\supp(u-v)\setminus \Jcal = \supp(u)\setminus \Jcal$.
\mycut{
  We show the second equality next.
  Since $\supp(u-v) \subseteq \supp(u)$ as shown above, we have $[\supp(u-v)]^c \supseteq [\supp(u)]^c$. Hence, $[\supp(u-v)]^c\setminus \Jcal \supseteq [\supp(u)]^c \setminus \Jcal = [\supp(u)]^c$, where the last equation follows from the assumption that $\Jcal \subset \supp(u)$. Conversely, in view of $\supp(u)\setminus \Jcal \subseteq \supp(u-v)$ shown in the first part, we have $[\supp(u-v)]^c \setminus \Jcal \subseteq  [\supp(u)\setminus \Jcal]^c \setminus \Jcal \subseteq  [\supp(u)]^c \setminus \Jcal = [\supp(u)]^c$, where we use $\Jcal \subset \supp(u)$ again for the last equation. This yields the second equality.
 }
\end{proof}

\begin{proposition} \label{prop:negative_second_term}
  Let $\mathcal P$ be a closed, convex and CP admissible set in $\mathbb R^N$. Given a matrix $A \in \mathbb R^{m\times N}$, a vector $0\ne u \in \Sigma_K \cap \mathcal P$,  and any index set $\Jcal \subset \supp(u)$, let $v$ be an arbitrary solution to $\mathbf Q: \, \min_{ w \in \mathcal P, \, \supp(w) \subseteq \Jcal }  \, \| A (w - u)\|^2_2 $. Then the following hold:
  \[
    \sum_{j \in \supp(u-v)\cap\Jcal } \langle A(u-v), A_{\bullet j}  \rangle \cdot (u-v)_j \, \le \, 0,
  \]
  and
  \[
   \|A (u-v)\|^2_2 \, \le \, \sum_{j \in \supp(u)\setminus\Jcal } \langle A(u-v), A_{\bullet j}  \rangle \cdot (u-v)_j.
  \]
\end{proposition}

\begin{proof}
  Note that such an optimal solution $v$ exists due to Corollary~\ref{coro:sol_existence_CP_adm}.
  Define the convex function $g(z):= \| A_{\bullet \Jcal} z - A u \|^2_2$ with $z\in \mathbb R^{|\Jcal|}$, and the constraint set $\mathcal W := \{ z \, | \, (z, 0) \in \pi_\Jcal(\mathcal P) \}$.
  It follows from Lemma~\ref{lem:CP_adm_closedness} that $\pi_\Jcal(\mathcal P)$ is closed. Since $\Pcal$ is convex, so is $\pi_\Jcal(\mathcal P)$. Hence, $\pi_\Jcal(\mathcal P)$ is closed and convex. This shows that $\mathcal W$ is also a closed convex set. Moreover, the underlying optimization problem $\mathbf Q$ can be equivalently formulated as the convex optimization problem: $\displaystyle   \min_{ z \in \mathcal W } g(z)$.
 %
 % Since $\mathcal P \cap \{ (z, 0) \, | \, \in \pi_\Jcal(\mathcal P) \}$ is closed and convex, $\mathcal W$ %is also a closed convex set. Furthermore, the underlying optimization problem becomes
 % $
 %  \displaystyle   \min_{ z_\Jcal \in \mathcal W } g(z_\Jcal).
 % $
 %
  Therefore, the optimal solution $v=(v_\Jcal, 0)$ satisfies the necessary and sufficient optimality condition given by the following variational inequality: $\langle \nabla g(v_\Jcal), z - v_\Jcal \rangle \ge 0$ for all $z \in \mathcal W$. Since $\mathcal P$ is CP admissible, we have $(u_\Jcal, 0) \in \mathcal P$ so that $u_\Jcal \in \mathcal W$. In view of $\nabla g(v_\Jcal) = A^T_{\bullet \Jcal}(A_{\bullet \Jcal} v_\Jcal - A u)=A^T_{\bullet \Jcal} (Av-Au)$, we have
  \[
     0 \, \le \, \langle A^T_{\bullet \Jcal}(A_{\bullet \Jcal} \, v_\Jcal - A u), u_\Jcal - v_\Jcal \rangle  = \langle A v - A u, A_{\bullet \Jcal} (u - v)_\Jcal \rangle.
  \]
  This implies that $\langle A (u -v), A_{\bullet \Jcal} (u - v)_\Jcal \rangle \le 0$. Consequently, we obtain
  \begin{eqnarray*}
  \lefteqn{ \sum_{j \in \supp(u-v)\cap\Jcal } \langle A(u-v), A_{\bullet j}  \rangle \cdot (u-v)_j} \notag \\ [5pt]
     & = & \sum_{j \in \supp(u-v)\cap\Jcal } \langle A(u-v), A_{\bullet j}  \rangle \cdot (u-v)_j
      + \sum_{j \in [\supp(u-v)]^c\cap\Jcal } \langle A(u-v), A_{\bullet j}  \rangle \cdot (u-v)_j \notag \\ [5pt]
    & = & \sum_{j \in \Jcal } \langle A(u-v), A_{\bullet j}  \rangle \cdot (u-v)_j \, = \, \langle A(u-v), A_{\bullet \Jcal}  (u - v)_\Jcal \rangle \notag \\ [5pt]
    & \, \le \, &  0. \notag %%\label{eqn:ineq_negative_2nd_term}
\end{eqnarray*}
Furthermore, we have
   \begin{eqnarray*}
   \| A(u - v) \|^2_2 & = & \sum^N_{j=1} \langle A(u-v), A_{\bullet j} (u-v)_j \rangle = \sum_{j \in \supp(u-v)} \langle A(u-v), A_{\bullet j}  \rangle \cdot (u-v)_j \\
   & = & \sum_{j \in \supp(u-v)\setminus\Jcal } \langle A(u-v), A_{\bullet j}  \rangle \cdot (u-v)_j +  \sum_{j \in \supp(u-v)\cap\Jcal } \langle A(u-v), A_{\bullet j}  \rangle \cdot (u-v)_j \\
   & \le & \sum_{j \in \supp(u-v)\setminus\Jcal } \langle A(u-v), A_{\bullet j}  \rangle \cdot (u-v)_j \\
   & = & \sum_{j \in \supp(u)\setminus\Jcal } \langle A(u-v), A_{\bullet j}  \rangle \cdot (u-v)_j,
\end{eqnarray*}
 where %%the first inequality follows from (\ref{eqn:ineq_negative_2nd_term}),
  the last equation follows from Lemma~\ref{lem:support_result}.
\end{proof}

%\gap
%
%{\bf Comment on bounded CP admissible}: may use the scaling technique in view of $x \in \mathcal P$ if and %only if $\lambda \cdot x \in \lambda \cdot \mathcal P$ for any $\lambda \ne 0$. Note that this scaling does %not affect an optimal solution (particularly its support)  since $u, v \in \mathcal P$ (up to scaling). When %$\mathcal P$ is bounded, choose a small $\lambda>0$ to make   $\lambda \cdot \mathcal P$ small such that the %sufficient condition is easy to establish.
%

%--------------------------------------------------------------------------------------
%
\section{Exact Vector Recovery on Closed, Convex, CP Admissible Cones for a Fixed Support via Constrained Matching Pursuit} \label{sect:exact_vector_recovery}
%

%
%It is shown in Section~\ref{sect:suff_cond_exact_recovery} that  closed, convex and CP admissible cones play %an important role in characterizing exact recovery, even for general closed, convex and CP admissible sets %(cf. Section~\ref{subsect:suff_cond_noncone}).
%
We first introduce the definition of exact vector recovery.

\begin{definition} \rm
\tblue{
Let a matrix $A\in \mathbb R^{m\times N}$ and a constraint set $\Pcal$ be given. For a fixed $z \in \Sigma_K \cap \mathcal P$, we say that {\it the exact vector recovery} of $z$ is achieved from $y=A z$ via Algorithm~\ref{algo:constrained_MP} if (i) the exact support recovery of $z$ is achieved, and (ii) along  any sequence $\big( (x^k, j^*_k, \Jcal_k) \big)_{k \in \mathbb N}$, once $\Jcal_s=\supp(z)$ is reached, then the minimization problem in Line 7 of Algorithm~\ref{algo:constrained_MP} yields the {\em unique} solution $x^s=z$.
If the exact vector recovery of any $z \in \Sigma_K \cap \, \mathcal P$ is achieved, then we say that  {\it the exact vector recovery on $\Sigma_K \cap \mathcal P$} (or simply the exact vector recovery) is achieved.
}
\end{definition}

We also say that a matrix $A$ achieves exact vector (resp. support) recovery on $\Pcal$ if  the exact vector (resp. support) recovery on $\Sigma_K \cap \mathcal P$ is achieved using $A$.
For a fixed index set $\Scal$, we say that {\it the exact vector recovery on $\Pcal$ for $\Scal$} is achieved if exact vector recovery of any $z \in \Pcal$ with $\supp(z)=\Scal$ is achieved.

This section is focused on the exact vector recovery on closed, convex and CP admissible cones for a fixed support. By Proposition~\ref{prop:CP_admissible_cone}, such a cone is a Cartesian product of copies of $\mathbb R$, $\mathbb R_+$ and $\mathbb R_-$, which includes $\mathbb R^N$ and $\mathbb R^N_{+}$.

%--------------------------------------------------------------------------
%
\subsection{Revisit of Exact Vector Recovery on $\mathbb R^N$ for a Fixed Support via OMP: A Counterexample to a Necessary Exact Recovery Condition in the Literature} \label{subsect:OMP_counterexample}

When the sparse recovery problem (\ref{eqn:constrained_L0}) is constraint free, i.e., $\Pcal=\mathbb R^N$, the constrained matching pursuit scheme given by Algorithm~\ref{algo:constrained_MP} reduces to the OMP \cite{PRKrishnaprasad_Asilomar93}. The OMP has been extensively studied in the signal processing and compressed sensing literature, and many results have been developed for support or vector recovery using the OMP \cite{FoucartRauhut_book2013, MoS_TIT12}. In particular, ``necessary'' and sufficient conditions are established in \cite[Proposition 3.5]{FoucartRauhut_book2013} for exact vector recovery via the OMP for a fixed support; the same ``necessary'' and sufficient conditions are also given by Tropp \cite[Theorems 3.1 and 3.10]{Tropp_ITI04}.
%
% In what follows, for a fixed index set $\Scal$, we say that {\it a matrix $A \in \mathbb R^{m\times N}$  %achieves the exact vector recovery via the OMP} if any nonzero $x \in \mathbb R^N$ with $\supp(x)=\Scal$ is %uniquely recovered from $y=A x$ in at most $|\Scal|$ iterations via the OMP.
%
For the sake of completeness and the ease of the subsequent discussions, we present the real version of \cite[Proposition 3.5]{FoucartRauhut_book2013} as follows, i.e., $A \in \mathbb R^{m\times N}$, $y \in \mathbb R^m$, and $x \in \mathbb R^N$, using slightly modified wording.

\begin{proposition} \cite[Proposition 3.5]{FoucartRauhut_book2013} \label{prop:Foucart_prop3.5}
 Given a matrix $A \in \mathbb R^{m\times N}$ with unit columns, every nonzero vector $x \in \mathbb R^N$ supported on a given index set $\Scal$ of size $s$ (i.e., $\supp(x)=\Scal$ and $|\supp(x)|=s$) is recovered from $y=A x$ after at most $s$ iterations of OMP if and only if the following two conditions hold:
 \begin{itemize}
   \item [(i)] The matrix $A_{\bullet \Scal}$ is injective (i.e., $A_{\bullet \Scal}$ has full column rank), and
    \item [(ii)]  \begin{equation} \label{eqn:condition_Prop_3.5}
    \max_{j \in \Scal} \big| (A^T A z)_j \big | \, > \, \max_{j \in \Scal^c} \big |(A^T A z)_j \big|, \qquad \quad \forall \ 0\ne z \in \mathbb R^{N} \ \mbox{ with } \ \supp(z) \subseteq \Scal.
 \end{equation}
 \end{itemize}
  Further, under condition (i), condition (\ref{eqn:condition_Prop_3.5}) holds  if and only if
  \begin{equation} \label{eqn:exact_recovery_condition}
    \big\| (A^T_{\bullet \Scal} A_{\bullet \Scal} )^{-1} A^T_{\bullet \Scal} A_{\bullet \Scal^c} \big\|_1 < 1,
  \end{equation}
 where $\|\cdot \|_1$ denotes the matrix 1-norm.
\end{proposition}

The ``proof'' of this proposition can be found on page 68 of the well received monograph \cite{FoucartRauhut_book2013} by Foucart and Rauhut, and its equivalent condition (\ref{eqn:exact_recovery_condition}) in term of the matrix 1-norm follows from  \cite[Remark 3.6]{FoucartRauhut_book2013}.
Also see a similar sufficiency proof in \cite[Theorem 3.1]{Tropp_ITI04} and a ``necessity'' proof in \cite[Theorem 3.10]{Tropp_ITI04}, where condition (\ref{eqn:exact_recovery_condition}) is referred to as the {\em exact recovery condition} coined by Tropp in \cite{Tropp_ITI04}.
Clearly, conditions (i) and (ii) are sufficient for the exact vector recovery.
%
%i.e., conditions (i) and (ii) imply the exact vector recovery of any nonzero vector $x \in \mathbb R^N$ with %$\supp(x)=\Scal$ from $y=Ax$ via the OMP.
%
Further, condition (i) is necessary for the exact vector recovery. However, we find that condition (ii) only {\em partially} holds for the necessity of the exact vector recovery.
Specifically, condition (ii) is necessary  when the index set $\Scal$ satisfies $|\Scal|=1$ or $|\Scal|=2$; when $|\Scal|=3$, we construct a nontrivial counterexample (i.e., a matrix $A$) such that any nonzero vector $x \in \mathbb R^N$ with $\supp(x)=\Scal$ is exactly recovered via the OMP using the matrix $A$ but this $A$ does not satisfy (\ref{eqn:condition_Prop_3.5}) or its equivalence (\ref{eqn:exact_recovery_condition}).

The construction of our counterexample is motivated by an unsuccessful attempt to justify the following implication, which is the last key step given in the necessity proof for \cite[Proposition 3.5]{FoucartRauhut_book2013}:
\begin{align}
\left[ \ \max_{j \in \Scal} \big| (A^T A z)_j \big | \, > \, \max_{j \in \Scal^c} \big |(A^T A z)_j \big|, \quad \forall \ 0 \ne z \in \mathbb R^{N} \ \mbox{ with } \ \supp(z) = \Scal \ \right] \, \Longrightarrow  \notag \\
\left[ \ \max_{j \in \Scal} \big| (A^T A z)_j \big | \, > \, \max_{j \in \Scal^c} \big |(A^T A z)_j \big|, \quad \forall \ 0\ne z \in \mathbb R^{N} \ \mbox{ with } \ \supp(z) \subseteq \Scal \ \right ],  \label{eqn:necessity_implication}
\end{align}
where we assume that the exact vector recovery is achieved and $A_{\bullet \Scal}$ has full column rank. Note that the hypothesis of the implication given by (\ref{eqn:necessity_implication})  holds since it follows from the first step of the OMP using $A$.
%
%Here for a fixed index set $\Scal$, we say that {\it a matrix $A \in \mathbb R^{m\times N}$  achieves the %exact vector recovery via the OMP} if any nonzero $x \in \mathbb R^N$ with $\supp(x)=\Scal$ is uniquely %recovered from $y=A x$ in at most $|\Scal|$ iterations via the OMP.
%
To elaborate an underlying reason for the failure of this implication, we define the function $q(z):=  \max_{j \in \Scal} | (A^T A z)_j | - \max_{j \in \Scal^c} |(A^T A z)_j |$ for $z \in \mathbb R^N$ and the set $\mathcal R:=\{ z \in \mathbb R^N \, | \, z \ne 0, \ \supp(z) = \Scal \}$. Clearly, $q(\cdot)$ is continuous. Further, any nonzero $\wt z \in \mathbb R^N$ with $\supp(\wt z) \subset \Scal$ is on the boundary of $\mathcal R$ such that there exists a sequence $(z_k)$ in $\mathcal R$ converging to $\wt z$. Hence, the sequence $(q(z_k))$ converges to $q(\wt z)$, where each $q(z_k)>0$ in view of the hypothesis of the implication (\ref{eqn:necessity_implication}). However, one can only conclude that $q(\wt z) \ge 0$ instead $q(\wt z)>0$.
%
%%\gap
%
 The counterexample we construct shows that when $|\Scal|=3$, there exists a matrix $A$ achieving the exact vector recovery via the OMP but the corresponding $q(\wt z)=0$ for some $0\ne \wt z \in \mathbb R^N$ with $\supp(\wt z) \subset \Scal$; see Remark~\ref{remark:counter_example} for details. This example invalidates the  implication (\ref{eqn:necessity_implication}).

A similar argument also explains the failure of Tropp's necessity proof in \cite[Theorem 3.10]{Tropp_ITI04}. In fact, the (nonzero) signal ${\mathbf s}_{bad}$ constructed in that proof is shown to satisfy $\rho({\mathbf s}_{bad})\ge 1$, which is equivalent to $q({\mathbf s}_{bad}) \le 0$. However, if $\supp({\mathbf s}_{bad})$ is a proper subset of the index set $\Lambda_{opt}$, which is equivalent to the index set $\Scal$ defined above, then the argument based on the first step of the OMP used in the proof for \cite[Theorem 3.10]{Tropp_ITI04} becomes invalid. In fact,
the counterexample we construct shows that when $|\Scal|=3$, there exists a matrix $A$ achieving the exact vector recovery via the OMP but a nonzero $\wt z$ with $\supp(\wt z) \subset \Scal$ exists such that the corresponding $q(\wt z)=0$ or equivalently $\rho(\wt z) =1$. See Remark~\ref{remark:counter_example} for details.
\footnote{In a private communication, Dr. Joel A. Tropp pointed out to the authors that this issue may be related to the borderline case indicated in Footnote 2 in his paper \cite{Tropp_ITI04}.}

%
%we show that the above proposition is valid when the index set $S$ satisfies $|S|=1$ or $|S|=2$; however, %when $|S|=3$, we construct a (nontrivial) counterexample where the condition~(\ref{eqn:condition_Prop_3.5}) %does not hold.
%

We introduce more assumptions and notation through the rest of the development in this section. Consider a matrix $A \in \mathbb R^{m\times N}$ with unit columns, i.e., $\|A_{\bullet i}\|_2=1$ for each $i=1, \ldots, N$. Define $\vartheta_{ij}:=\langle A_{\bullet i}, A_{\bullet j} \rangle$ for $i, j \in \{1, \ldots, N\}$, and for each $i$, define the function
\begin{equation} \label{eqn:g_i_def}
  g_i(z) \, :=  \, \big | \langle A_{\bullet i}, A z \rangle  \big | \, = \, \Big | \sum^N_{j=1} \vartheta_{ij} z_j   \Big |, \qquad \quad \forall \, z =(z_1, \ldots, z_N)^T \in \mathbb R^N.
\end{equation}

%------------------------------------------------------------------------
%
\subsubsection{Positive Necessity Results and Their Implications}

 This subsection presents certain cases where condition (\ref{eqn:exact_recovery_condition}) (or equivalently (\ref{eqn:condition_Prop_3.5})) is indeed necessary for the exact vector recovery for a given support $\Scal$. The first result shows that \cite[Proposition 3.5]{FoucartRauhut_book2013} (or Proposition~\ref{prop:Foucart_prop3.5} of the present paper) holds when the index set $\Scal$ is of size 1 or 2.

\begin{theorem} \label{thm:nonconstrained_S2}
  For a matrix $A \in \mathbb R^{m\times N}$ with unit columns and an index set $\Scal$ with $|\Scal|=1$ or $|\Scal|=2$,  the exact vector recovery of every nonzero vector $x \in \mathbb R^N$ with $\supp(x)=\Scal$ is achieved from $y=A x$ via the OMP if and only if the conditions (i) and (ii) in Proposition~\ref{prop:Foucart_prop3.5} hold.
\end{theorem}

\begin{proof}
 In light of the prior discussions and the argument for \cite[Proposition 3.5]{FoucartRauhut_book2013}, we only need to show that the implication (\ref{eqn:necessity_implication}) holds when $A$ achieves the exact vector recovery via the OMP and $A_{\bullet \Scal}$ has full column rank. The case of $|\Scal|=1$ is trivial, and we focus on the case of $|\Scal|=2$ as follows. Without loss of generality, let $\Scal=\{ 1, 2\}$. In view of $g_i$'s defined in (\ref{eqn:g_i_def}), it suffices to show that if
 $\max( g_1(z), g_2(z)) > \max_{i \ge 3} g_i (z), \forall \, z \mbox{ with } \supp(z)=\{1, 2\}$, then $ \max (g_1(z), g_2(z)) > \max_{i \ge 3} g_i (z), \forall \,  z$ with $\supp(z)=\{1\}$ or $\supp(z)=\{2 \}$.  Since $A_{\bullet \Scal}$ has full column rank, the $2\times 2$ matrix
 $
   A^T_{\bullet \Scal} A_{\bullet \Scal} = \begin{bmatrix} 1 & \vartheta_{12} \\ \vartheta_{12} & 1 \end{bmatrix}
 $
 is positive definite. Hence, $|\vartheta_{12}| <1$. For any $z$ with  $\supp(z)=\{1\}$, we have $\max (g_1(z), g_2(z)) = \max( |z_1|, |\vartheta_{12} z_1|)= g_1(z)> g_2(z)$ because $z_1 \ne 0$ and $|\vartheta_{12}| <1$. Similarly,  $\max (g_1(z), g_2(z)) = g_2(z)> g_1(z)$ when $\supp(z)=\{2\}$.

In what follows, we consider an arbitrary $z^*$ with $\supp(z^*)=\{1\}$ first. Note that $g_j(z^*)=|\vartheta_{j1} z^*_1|$ for each $j$, where $z^*_1\ne 0$.
Since $z^*$ is on the boundary of $\mathcal R:=\{ z \in \mathbb R^N \, | \, \supp(z)=\{1,2\} \}$ on which $\max( g_1(z), g_2(z)) > \max_{i \ge 3} g_i (z)$, we deduce via the continuity of $g_i$'s that $g_1(z^*)=\max (g_1(z^*), g_2(z^*)) \ge g_i(z^*)$ for each $i \ge 3$.
We show next that $g_1(z^*)> g_i (z^*)$ for all $i \ge 3$ by contradiction. Suppose, in contrast, $g_1(z^*) = g_i(z^*)$ for some $i \ge 3$, i.e., $|z^*_1| = | \vartheta_{i 1} z^*_1|=\gamma$. For any $v \in \mathbb R^N$ with $\supp(v)=\{1, 2\}$ and $\| v\|_2>0$ sufficiently small, $\max (g_1(z^*+v), g_2(z^*+v) ) = g_1(z^*+v)$ due to $g_1(z^*)> g_2(z^*)$,  and $z^*+v \in \mathcal R$ so that $g_1(z^*+v) > g_i(z^*+v)$. Therefore, we have
\begin{equation} \label{eqn:R^N_S=1,2}
 |z^*_1 + p^T v_{\Scal} | \, > \, | \vartheta_{i 1} z^*_1 + q^T v_{\Scal}|,
\end{equation}
where $p=(1, \vartheta_{12})^T$, $q=(\vartheta_{i1},  \vartheta_{i2})^T$, and $v_\Scal=(v_1, v_2)^T \in \mathbb R^2$.
 Letting $\gamma:=|z^*_1|>0$, we obtain four possible cases from $|z^*_1| = | \vartheta_{i 1} z^*_1|$: (i) $(z^*_1,  \vartheta_{i 1} z^*_1)=(\gamma, \gamma)$; (ii) $(z^*_1,  \vartheta_{i 1} z^*_1)=(\gamma, -\gamma)$; (iii) $(z^*_1,  \vartheta_{i 1} z^*_1)=(-\gamma, \gamma)$; and (iv) $(z^*_1,  \vartheta_{i 1} z^*_1)=(-\gamma, -\gamma)$. In each of these cases, it follows from (\ref{eqn:R^N_S=1,2}) that  $( \sgn(z^*_1) \cdot p - \sgn(\vartheta_{i 1} z^*_1) \cdot q)^T v_{\Scal}>0$ for all $\| v_\Scal\|>0$ sufficiently small, where $\sgn(\cdot)$ is the signum function. In view of $\supp(v_\Scal)=\supp(-v_\Scal)$, we have $( \sgn(z^*_1) \cdot p - \sgn(\vartheta_{i 1} z^*_1) \cdot q)^T v_{\Scal}>0$ and $( \sgn(z^*_1) \cdot p - \sgn(\vartheta_{i 1} z^*_1) \cdot q)^T (-v_{\Scal})>0$ for all $\| v_\Scal \|_2>0$ sufficiently small. This yields a contradiction. Hence, $\max (g_1(z^*), g_2(z^*))> g_i (z^*)$ for all $i \ge 3$ when $\supp(z^*)=\{ 1\}$. The case of $\supp(z^*)=\{ 2\}$ also follows by interchanging the roles of $g_1$ and $g_2$. Consequently, the implication (\ref{eqn:necessity_implication}) holds, which leads to condition (ii) in  Proposition~\ref{prop:Foucart_prop3.5}.
\end{proof}

By leveraging the necessary and sufficient recovery conditions in Theorem~\ref{thm:nonconstrained_S2} for a fixed support of size 2, we show that condition $(\mathbf H)$ is necessary for the exact vector or support recovery on $\Sigma_2$.

\begin{corollary} \label{coro:condition_H'_necessary_RN_S2}
   Let $A \in \mathbb R^{m\times N}$ have unit columns. Then $A$ achieves the exact vector recovery on $\Sigma_2$ if and only if (i) condition $(\mathbf H)$ holds on $\Sigma_2$, and (ii) any two distinct columns of $A$ are linearly independent.
\end{corollary}

\begin{proof}
``If''. In view of  Proposition~\ref{prop:condition_H_suppt_recovery}, condition $(\mathbf H)$ yields the exact support recovery on $\Sigma_2$.  Besides, condition (ii) guarantees that the unique $x^2$ equals $z$ for any $z \in \Sigma_2$ with $|\supp(z)|=2$. It also ensures that the unique $x^1=z$ for any $z \in \Sigma_2$ with $|\supp(z)|=1$. This yields the exact vector recovery on $\Sigma_2$.

``Only if''. Suppose $A$ achieves the exact vector recovery on $\Sigma_2$. Clearly, condition (ii) is necessary as shown before. To show that condition (i) is also necessary, consider a vector $z \in \Sigma_2$ with $|\supp(z)|=2$.
Without loss of generality, we assume that $\supp(z)=\{ 1, 2\}$. Since $A$ achieves the exact vector recovery on $\Sigma_2$, it must achieve the exact support recovery for the fixed support $\Scal=\{1, 2 \}$. Hence it follows from Theorem~\ref{thm:nonconstrained_S2} that  $\| (A^T_{\bullet \Scal} A_{\bullet \Scal} )^{-1} A^T_{\bullet \Scal} A_{\bullet \Scal^c} \|_1 < 1$, which is equivalent to
\begin{equation} \label{eqn:condtion_H'_necessary_RN_S2}
   1-\vartheta^2_{12} \, > \, \max_{j\in \Scal^c} \big( \, |\vartheta_{j1}- \vartheta_{j2} \vartheta_{12}| + |\vartheta_{j2}- \vartheta_{j1} \vartheta_{12}| \, \big).
\end{equation}
Consider the three proper subsets of $\supp(z)=\{ 1, 2\}$, i.e., $\Jcal=\emptyset$, $\Jcal=\{1 \}$, and $\Jcal=\{2 \}$. When $\Jcal=\emptyset $, the inequality (\ref{eqn:condition_H'}) holds for $u=z$ and $v=0$ in light of  $\max_{j \in \Scal} \big| (A^T A z)_j \big | \, > \, \max_{j \in \Scal^c} \big |(A^T A z)_j \big|$ obtained from the first step of the OMP. Moreover, we have either $|z_1 + \vartheta_{12} z_2| \ge |\vartheta_{12} z_1 + z_2|$ or $|z_1 + \vartheta_{12} z_2| \le |\vartheta_{12} z_1 + z_2|$. For the former case, we deduce from the exact support recovery of $z$ via the OMP that $j^*_1=1$ and $\Jcal_1=\{ 1 \}$ such that $x^1=(A^T_{\bullet 1} A z) \mathbf e_1$ is the unique optimal solution to $\min_{\supp(w) \subseteq \Jcal_1} \| A(z - w) \|^2_2$. Hence, by Corollary~\ref{coro:Nec_Suf_suppt_recovery_RN_RN+}, the exact support recovery shows that $f^*_2(z, x^1)< \min_{j \in \Scal^c} f^*_j(z, x^1)$, leading to the inequality (\ref{eqn:condition_H'}) for $u=z$ and $v=x^1$ when $\Jcal=\{ 1 \}$. We then consider $\Jcal=\{2\}$. In this case, the unique optimal solution $v^*$ to $\min_{\supp(w) \subseteq \Jcal} \| A(z - w) \|^2_2$ is given by $v^*=(A^T_{\bullet 2} A z) \mathbf e_2 = (\vartheta_{12} z_1 + z_2) \mathbf e_2$. Therefore, $A^T_{\bullet j} A (z - v^*) = (\vartheta_{j1} -  \vartheta_{j2} \vartheta_{12} ) z_1$ for any $j$. We thus have $|A^T_{\bullet 1} A (z-v^*)|= |1 - \vartheta^2_{12}| \cdot |z_1|$ and $|A^T_{\bullet j} A (z-v^*)|= |\vartheta_{j1} - \vartheta_{12} \vartheta_{j2}| \cdot |z_1|$, where $z_1 \ne 0$.
Noting that $f^*_1(z, v^*) < \min_{j\in \Scal^c} f^*_j(z, v^*)$ if and only if $|A^T_{\bullet 1} A (z-v^*)| > \max_{j\in \Scal^c}  |A^T_{\bullet j} A (z-v^*)|$, we deduce via the above results and (\ref{eqn:condtion_H'_necessary_RN_S2}) that $f^*_1(z, v^*) < \min_{j\in \Scal^c} f^*_j(z, v^*)$,
leading to the inequality (\ref{eqn:condition_H'}) for $u=z$ and $v=v^*$ when $\Jcal=\{ 2 \}$.  The other case where $|z_1 + \vartheta_{12} z_2| \le |\vartheta_{12} z_1 + z_2|$ can be established in a similar way. Further, for any $u\in \Sigma_2$ with $|\supp(u)|=1$ and $\Jcal=\emptyset$,  the inequality (\ref{eqn:condition_H'}) also holds. Thus condition $(\mathbf H)$ holds on $\Sigma_2$.
\end{proof}

The next result shows that even though condition (\ref{eqn:exact_recovery_condition}) (or equivalently  (\ref{eqn:condition_Prop_3.5})) may fail to be necessary, it is necessary for {\em almost all} the matrices achieving the exact vector recovery associated with a fixed support $\Scal$. This result also illustrates the challenge of constructing a counterexample.
Toward this end, let $\mathcal U$ be the set of all matrices in $\mathbb R^{m\times N}$ with unit columns, i.e.,
$
   \mathcal U \, := \, \big\{ A \in \mathbb R^{m\times N} \ | \, \|A_{\bullet i} \|_2=1, \ \forall \ i=1, \ldots, N \big\}.
$
Note that $\mathcal U$ is the Cartesian product of $N$ copies of unit $\ell_2$-spheres in $\mathbb R^m$. Hence, $\mathcal U$ is a compact manifold of dimension $(m-1)N$, and it attains a (finite) positive measure with a Lebesgue measure $\mu$ on $\mathcal U$. For a fixed index set $\Scal$, define the set
\[
 \mathcal D \, := \, \big\{ A \in \mathcal U \, | \, \mbox{ $A$ achieves the exact vector recovery for the given support $\Scal$ } \big \}.
\]
Recall that for any $A \in \mathcal D$, $A_{\bullet \Scal}$ has full column rank.

\begin{proposition}
 Let the set $\mathcal D' :=\{ A \in \mathcal D \, | \, \| (A^T_{\bullet \Scal} A_{\bullet \Scal} )^{-1} A^T_{\bullet \Scal} A_{\bullet \Scal^c} \|_1 = 1\}$, and $\mu$ be a Lebesgue measure  on $\mathcal U$. Then $\mu(\mathcal D)>0$ and $\mu(\mathcal D')=0$.
\end{proposition}

\begin{proof}
For a given matrix $A\in \mathcal D$, we recall the function $q(z):=  \max_{j \in \Scal} | (A^T A z)_j | - \max_{j \in \Scal^c} |(A^T A z)_j |$ for $z \in \mathbb R^N$ and the set $\mathcal R:=\{ z \in \mathbb R^N \, | \, z \ne 0, \ \supp(z) = \Scal \}$ given below (\ref{eqn:necessity_implication}). Since $A$ achieves  the exact vector recovery for the given support $\Scal$, we have $q(z)>0$ for all $z \in \mathcal R$. Moreover,
it follows from the discussioins below (\ref{eqn:necessity_implication}) that $q(\wt z) \ge 0$ for any nonzero $\wt z \in \mathbb R^N$ with $\supp(\wt z) \subset \Scal$.
%
%Furthermore, since
%  $q$ is continuous and any nonzero $\wt z \in \mathbb R^N$ with $\supp(\wt z) \subset \Scal$ is in the %closure of $\mathcal R$, it follows from the argument below (\ref{eqn:necessity_implication}) that $q(\wt z) %\ge 0$ for any nonzero $\wt z \in \mathbb R^N$ with $\supp(\wt z) \subset \Scal$.
%
  By a similar argument for \cite[Remark 3.6]{FoucartRauhut_book2013}, we have $\| (A^T_{\bullet \Scal} A_{\bullet \Scal} )^{-1} A^T_{\bullet \Scal} A_{\bullet \Scal^c} \|_1 \le  1$ for any $A \in \mathcal D$.

  Define the set $\mathcal D'' :=\{ A \in \mathcal D \, | \, \| (A^T_{\bullet \Scal} A_{\bullet \Scal} )^{-1} A^T_{\bullet \Scal} A_{\bullet \Scal^c} \|_1 < 1\}$.
  In view of the above argument, we see that $\mathcal D$ is the disjoint union of $\mathcal D'$ and $\mathcal D''$. Since $\mathcal D''$ is a (relatively) open subset in $\mathcal U$, we deduce that $\mu(\mathcal D'')>0$. Therefore, $\mu(\mathcal D) \ge \mu(\mathcal D'')>0$. Moreover, define
  \begin{align*}
   \wt{\mathcal D} & :=  \Big\{ \, A \in \mathcal U \, | \, \mbox{ $A_{\bullet \Scal}$ has full column rank, and $\| (A^T_{\bullet \Scal} A_{\bullet \Scal} )^{-1} A^T_{\bullet \Scal} A_{\bullet \Scal^c} \|_1 = 1$} \Big\}, \\
    \mathcal W_j & :=  \Big\{ \, A \in \mathcal U \, | \, \mbox{ $A_{\bullet \Scal}$ has full column rank, and $\big\| \big[(A^T_{\bullet \Scal} A_{\bullet \Scal} )^{-1} A^T_{\bullet \Scal} A_{\bullet \Scal^c}\big]_{\bullet j} \big\|_1 = 1$} \Big\}, \quad j=1, \ldots, |\Scal^c|.
  \end{align*}
  Hence, $\mathcal D' \subseteq  \wt{\mathcal D} \subset \bigcup^{|\Scal^c|}_{j=1} \mathcal W_j$. Let $\mathbf a\in \mathbb R^{mN}$ be the vectorization of $A \in \mathbb R^{m\times N}$, i.e., $\mathbf a$ is generated by stacking the columns of $A$ on top of one another. For each fixed $j$, $\big\| \big[(A^T_{\bullet \Scal} A_{\bullet \Scal} )^{-1} A^T_{\bullet \Scal} A_{\bullet \Scal^c}\big]_{\bullet j} \big\|_1 = 1$ holds if and only if the piecewise polynomial function $G_j(\mathbf a)=0$, where $G_j(\mathbf a):= \sum^{|\Scal|}_{k=1} | G_{j, k}(\mathbf a) | - G_{j, k+1}(\mathbf a)$, and each $G_{j, k}(\cdot):\mathbb R^{mN} \rightarrow \mathbb R$ is a polynomial function. In view of this result, it is easy to verify that $\mathcal W_j$ is a subset of a finite union of the sets of the following form: $\big\{ \, A \in \mathcal U \, | \, \mbox{ $A_{\bullet \Scal}$ has full column rank, and $H_s(\mathbf a)=0$} \big\}$, where $H_s(\cdot)$ is a (nonzero) polynomial function. Clearly, each set of this form is a lower dimensional sub-manifold of $\mathcal U$ and thus is of zero measure. Thus $\mu(\mathcal W_j)=0$ for each $j$, and we thus have $\mu(\mathcal D')=0$.
\end{proof}

%%\gap

%-----------------------------------------------------------------------
%
\subsubsection{Construction of a Counterexample for a Fixed Support of Size 3} \label{subsect:counterexample_S=3}

In this subsection, we construct a nontrivial counterexample to show that condition (\ref{eqn:exact_recovery_condition}) (or equivalently (\ref{eqn:condition_Prop_3.5})) fails to be necessary.
The main result is given by the following theorem.

\begin{theorem} \label{thm:counterexample}
 For an index set $\Scal$ with $|\Scal|=3$, there exists an $A \in \mathbb R^{4\times 4}$ with unit columns such that $A$ achieves exact vector recovery for the fixed support $\Scal$ via the OMP, $A_{\bullet \Scal}$ has full column rank, and $\big\| (A^T_{\bullet \Scal} A_{\bullet \Scal} )^{-1} A^T_{\bullet \Scal} A_{\bullet \Scal^c} \big\|_1 = 1$.
\end{theorem}

To construct such a matrix $A$ indicated in the above theorem, we first present some preliminary results. Without loss of generality, let $\Scal=\{1, 2, 3\}$ and $\Scal^c=\{ 4 \}$. In view of the function $g_i$'s defined in (\ref{eqn:g_i_def}), we introduce the following functions for $i=1, \ldots, 4$:
\[
   \wh g_i(v) \, : = \, \big | h^T_i v |, \ \forall \ v \in \mathbb R^3, \quad \mbox{ where } \quad h_i := \big( \vartheta_{i1}, \vartheta_{i2}, \vartheta_{i3} )^T \in \mathbb R^3,
\]
where we recall that $\vartheta_{ij}=\langle A_{\bullet i}, A_{\bullet j} \rangle$ for $i, j\in \{1, \ldots, 4\}$.
Hence,   $\max_{j \in \Scal} \big| (A^T A z)_j \big | \, > \, \max_{j \in \Scal^c} \big |(A^T A z)_j \big|$ for all $0 \ne z \in \mathbb R^{N}$ with $\supp(z) = \Scal$ if and only if the following holds:
\[
  (\mathbf P): \quad \max_{i=1, 2, 3} \wh g_i(v) \, > \, \wh g_4(v), \quad \forall \ v=(v_1, v_2, v_3)^T\in \mathbb R^3 \ \mbox{ with } v_1\cdot v_2\cdot v_3 \ne 0.
\]
Since each $\wh g_i$ and $\max_{i=1, 2, 3} \wh g_i(v)$ are convex piecewise affine functions \cite{MouShen_COCV18}, it is not surprising that the feasibility of ($\mathbf P$) can be characterized by that of certain linear inequalities. The following lemma gives a necessary and sufficient condition for ($\mathbf P$) in term of the feasibility of some linear inequalities.

\begin{lemma} \label{lem:feasbility_sufficiency}
 Let the matrix $H:=\begin{bmatrix} h_4 + h_1 & h_4-h_1 & h_4 + h_2 & h_4-h_2 & h_4 + h_3 & h_4-h_3 \end{bmatrix} \in \mathbb R^{3\times 6}$. Then $\max_{i=1, 2, 3} \wh g_i(v) > \wh g_4(v)$ for all $v=(v_1, v_2, v_3)^T\in \mathbb R^3$ with $v_1\cdot v_2\cdot v_3 \ne 0$ holds if and only if for each $\sigma :=(\sigma_1, \sigma_2, \sigma_3)\in \{ (\pm 1, \pm 1, \pm 1)\}$, there exist vectors $0\ne u \ge 0$ and $w \ge 0$ such that $ u + D_\sigma H w =0$, where the diagonal matrix $D_\sigma:=\mbox{diag}( \sigma_1, \sigma_2, \sigma_3) \in \mathbb R^{3\times 3}$.
\end{lemma}

\begin{proof}
 Note that ($\mathbf P$) fails if and only if there exists $\wh  v \in \mathbb R^3$ with $\wh v_1\cdot \wh v_2\cdot \wh v_3 \ne 0$ such that  $\wh g_i(\wh v) \le \wh g_4(\wh v)$ for each $i=1,2, 3$. We claim that the latter statement holds if and only if there exists $v^*\in \mathbb R^3$ with $v^*_1\cdot v^*_2\cdot v^*_3 \ne 0$ such that $| h^T_i v^* | \le h^T_4 v^*$ for each $i=1, 2, 3$. The ``if'' part is obvious since $h^T_4 v^* \le |h^T_4 v^*|=\wh g_4(v^*)$. To show the ``only if'' part, we let $v^*= \sgn( h^T_4 \wh v) \cdot \wh v$, where $\wh v$ satisfies the specified conditions. In view of $\wh g_i(v^*)=|h^T_i v^*|=|h^T_i \wh v|$ for $i=1, 2, 3$, $h^T_4 v^*= |h^T_4 \wh v|=\wh g_4(\wh v)$, and each $v^*_i \ne 0$, we conclude that the desired result holds. This completes the proof of the claim.
% \gap
%
% Further, due to the positive homogeneity of each $\wt g_i$, it is easy to show that
%

By using the above claim, we see that  ($\mathbf P$) fails if and only if there exists $v^*\in \mathbb R^3$ with $v^*_1\cdot v^*_2\cdot v^*_3 \ne 0$ such that $| h^T_i v^* | \le h^T_4 v^*$ for each $i=1, 2, 3$, where the latter is further equivalent to $\pm h^T_i v^* \le h^T_4 v^*$ for each $i=1, 2, 3$ or equivalently $H^T v^* \ge 0$. Therefore, ($\mathbf P$) fails if and only if there exist $\sigma  \in \{ (\pm 1, \pm 1, \pm 1)\}$ and $\wt v \in \mathbb R^3_{++}$ (i.e., $\wt v_i>0$ for each $i=1, 2, 3$)  such that $H^T D_\sigma \wt v \ge 0$. By virtue of the Motzkin's Transposition Theorem, we see that for a fixed $\sigma$, the linear inequality system $(D_\sigma H)^T \wt v \ge 0, \wt v >0$ has a solution $\wt v$ if and only if the linear inequality system $u + D_\sigma H w =0, 0\ne u \ge 0$ and $w \ge 0$ has no solution $(u, w)$. In other words, ($\mathbf P$) fails if and only if there exist $\sigma  \in \{ (\pm 1, \pm 1, \pm 1)\}$ such that the linear inequality system $u + D_\sigma H w =0$, $0\ne u \ge 0$, and $w \ge 0$ has no solution. As a result, ($\mathbf P$) holds if and only if for any $\sigma  \in \{ (\pm 1, \pm 1, \pm 1)\}$, there exist vectors $0\ne u \ge 0$ and $w \ge 0$ such that $ u + D_\sigma H w =0$.
\end{proof}

By making use of the above preliminary results, we prove Theorem~\ref{thm:counterexample} as follows.

\begin{proof}[Proof of Theorem~\ref{thm:counterexample}]
 Consider the matrix
 \begin{equation} \label{eqn:A_matrix_counterexample}
   \displaystyle  A \, = \, \begin{bmatrix} 1 & -\frac{1}{3} & -\frac{1}{3} & \frac{1}{3} \\
         0 &\frac{2\sqrt{2}}{3} & -\frac{\sqrt{2}}{3} & \frac{\sqrt{2}}{3} \\
         0 & 0 & \frac{\sqrt{6}}{3} & -\frac{\sqrt{6}}{12} \\
         0 & 0 & 0 & \frac{\sqrt{10}}{4}\end{bmatrix} \in \mathbb R^{4\times 4}.
 \end{equation}
 Recall that $\Scal=\{1, 2, 3\}$ and $\Scal^c=\{4\}$. It is easy to verify that $A$ is invertible with unit columns (i.e., $\|A_{\bullet i}\|_2=1$ for each $i$), $A_{\bullet\Scal}$ has full column rank, and
\begin{equation} \label{eqn:vartheta_value}
 \vartheta_{12}=\vartheta_{21}=\vartheta_{13}=\vartheta_{31}=\vartheta_{23}=\vartheta_{32}=-\frac{1}{3},
 \quad \vartheta_{41}=\vartheta_{42}=\frac{1}{3}, \quad \vartheta_{43} = -\frac{1}{2}.
\end{equation}
Hence, $A^T_{\bullet \Scal} A_{\bullet \Scal}=\begin{bmatrix} h_1 & h_2 & h_3\end{bmatrix} \in \mathbb R^{3\times 3}$ and $A^T_{\bullet \Scal} A_{\bullet \Scal^c} = h_4$, where
\[
   h_1 = \begin{bmatrix} 1 \\ -\frac{1}{3} \\ -\frac{1}{3} \end{bmatrix}, \quad
   h_2 = \begin{bmatrix}  -\frac{1}{3} \\ 1 \\ -\frac{1}{3} \end{bmatrix}, \quad
   h_3 = \begin{bmatrix}  -\frac{1}{3} \\ -\frac{1}{3} \\ 1 \end{bmatrix}, \quad
   h_4 = \begin{bmatrix}  \frac{1}{3} \\ \frac{1}{3} \\ -\frac{1}{2} \end{bmatrix}.
\]
Furthermore,
\[
   \Big( A^T_{\bullet \Scal} A_{\bullet \Scal} \Big)^{-1} = \begin{bmatrix} h_1 & h_2 & h_3\end{bmatrix}^{-1} =  \begin{bmatrix} 1 & -\frac{1}{3} & -\frac{1}{3} \\ -\frac{1}{3} & 1 & -\frac{1}{3} \\ -\frac{1}{3} & -\frac{1}{3} & 1 \end{bmatrix}^{-1} = \frac{3}{4} \begin{bmatrix} 2 & 1 & 1 \\ 1 & 2 & 1 \\ 1 & 1 & 2 \end{bmatrix}
\]
such that $\big\| (A^T_{\bullet \Scal} A_{\bullet \Scal} )^{-1} A^T_{\bullet \Scal} A_{\bullet \Scal^c}\big\|_1 = \frac{3}{8}+\frac{3}{8}+\frac{1}{4}= 1$.
The rest of the proof consists of two parts: the first part shows that $\max_{j \in \Scal} \big| (A^T A z)_j \big| >  \max_{j \in \Scal^c} \big |(A^T A z)_j \big|$ for all $z$ with $\supp(z) = \Scal$, and the second part shows that $A$ achieves the exact vector recovery for the index set $\Scal$.

We first show the following claim:
\begin{equation} \label{eqn:claim_I}
  \mbox{ Claim I: $\quad \max_{j \in \Scal} \big| (A^T A z)_j \big| >  \max_{j \in \Scal^c} \big |(A^T A z)_j \big|$ \ for all $z$ with $\supp(z) = \Scal$}.
\end{equation}
In view of Lemma~\ref{lem:feasbility_sufficiency}, we only need to show that for each $\sigma \in \{ (\pm 1, \pm 1, \pm 1)\}$, there exist vectors $0\ne u \ge 0$ and $w \ge 0$ such that $ u + D_\sigma H w =0$, where the matrix
\[
  H = \begin{bmatrix} h_4 + h_1 & h_4-h_1 & h_4 + h_2 & h_4-h_2 & h_4 + h_3 & h_4-h_3 \end{bmatrix} = \begin{bmatrix} \frac{4}{3} & -\frac{2}{3} & 0 & \frac{2}{3} & 0 & \frac{2}{3} \\
                       0 & \frac{2}{3} & \frac{4}{3} &  -\frac{2}{3} & 0 &  \frac{2}{3} \\
                       -\frac{5}{6} & -\frac{1}{6} & -\frac{5}{6} & -\frac{1}{6} & \frac{1}{2} & -\frac{3}{2} \end{bmatrix}.
\]
Toward this end, we give a specific solution $(u, w)$ to the above linear inequality system for each $\sigma$:
\begin{itemize}
  \item [(1)] $\sigma=(1, 1, 1)$. A solution is given by $u=-(H_{\bullet 2} + H_{\bullet 4})=(0, 0, \frac{1}{3})^T$ and $w=\mathbf e_2 + \mathbf e_4$;
  \item [(2)] $\sigma=(1, 1, -1)$. A solution is given by $u=H_{\bullet 5}=(0, 0, \frac{1}{2})^T$ and $w=\mathbf e_5$;
  \item [(3)] $\sigma=(1, -1, 1)$. A solution is given by $u=(\frac{2}{3}, \frac{2}{3}, \frac{1}{6})^T$ and $w=\mathbf e_2$;
   \item [(4)] $\sigma=(-1, 1, 1)$. A solution is given by $u=(\frac{2}{3}, \frac{2}{3}, \frac{1}{6})^T$ and $w=\mathbf e_4$;
  \item [(5)] $\sigma=(1, -1, -1)$. A solution is given by $u=H_{\bullet 5}=(0, 0, \frac{1}{2})^T$ and $w=\mathbf e_5$;
  \item [(6)] $\sigma=(-1, 1, -1)$. A solution is given by $u=H_{\bullet 5}=(0, 0, \frac{1}{2})^T$ and $w=\mathbf e_5$;
  \item [(7)] $\sigma=(-1, -1, 1)$. A solution is given by $u=(\frac{3}{4}, 0, \frac{5}{6})^T$ and $w=\mathbf e_1$;
  \item [(8)] $\sigma=(-1, 1, -1)$. A solution is given by $u=H_{\bullet 5}=(0, 0, \frac{1}{2})^T$ and $w=\mathbf e_5$.
\end{itemize}
Hence, Claim I holds in light of Lemma~\ref{lem:feasbility_sufficiency}.

We show next that the matrix $A$ achieves the exact vector recovery via the OMP for the given index set $\Scal=\{ 1, 2, 3\}$.
%
%i.e., each $z \in \mathbb R^4$ with $\supp(z)=\Scal$ is exactly recovered from $y=A z$.
%
Let $z$  be an arbitrary vector in $\mathbb R^4$ with $\supp(z)=\Scal$, and $y=A z = A_{\bullet \Scal} z_\Scal$. Consider the following three steps of the OMP:
%%\begin{itemize}
  %\item

 \indent $\bullet$  Step 1: Since $x^0=0$ and $y=A z$, it follows from (\ref{eqn:claim_I}) that $\max_{i=1, 2, 3} | A^T_{\bullet i} A (z - x^0)| > | A^T_{\bullet 4} A( z - x^0)|$.
%
%  $\max_{i=1, 2, 3} | A^T_{\bullet i} A_{\bullet \Scal} z_\Scal| > | A^T_{\bullet 4} A_{\bullet \Scal} %z_\Scal|$.
%
  Hence, by Corollary~\ref{coro:Nec_Suf_suppt_recovery_RN_RN+}, $j^*_1\in \Scal= \{1, 2, 3\}$ and $\Jcal_1=\{ j^*_1\}$. Also, $x^1:=\argmin_{\supp(w) \subseteq \Jcal_{1} } \, \| y - A w\|^2_2$ is given by $x^1 =  (A^T_{\bullet j^*_1} A_{\bullet \Scal} z_{\Scal}) \cdot \mathbf e_{j^*_1}$.
 %
 % $x^1_{j^*_1} = A^T_{j^*_1} y = A^T_{\bullet j^*_1} A_{\bullet \Scal} z_{\Scal}$, and $x^1_i=0$ for all %$i\ne j^*_1$.
 %
  Note that $x^1_{j^*_1} \ne 0$ in view of Proposition~\ref{prop:index_set}.
  %\item

 \indent $\bullet$ Step 2: We first prove the following claim: for any $j_1 \in \Scal= \{1, 2, 3\}$ and $u = (A^T_{\bullet j_1} A z) \cdot \mathbf e_{j_1} \in \mathbb R^4$, $\max_{i=1, 2, 3} | A^T_{\bullet i} A ( z-  u) | > | A^T_{\bullet 4} A ( z -  u) |$.
%
%  $u_{j_1} = A^T_{\bullet j_1} A z$ and $u_i=0$ for each $i\ne j_1$.
     %
      \begin{proof}[Proof of the above claim]
        For any $j_1\in \{1, 2, 3\}$ and its corresponding $u$, let $v:=z - u$. Note that $v_i=z_i \ne 0$ for each $i\in \Scal \setminus \{ j_1 \}$. Therefore, if $z_{j_1} \ne A^T_{\bullet j_1} A z$, then $\supp(v)=\Scal$ so that the claim holds by virtue of  (\ref{eqn:claim_I}). To handle the case where  $z_{j_1} = A^T_{\bullet j_1} A z$, we consider $j_1=1$ first. Since $A^T_{\bullet 1} A z = A^T_{\bullet 1} A_{\bullet \Scal} z_\Scal = h^T_1 z_\Scal = z_1 - \frac{1}{3} z_2 - \frac{1}{3}z_3$, we must have $z_3=-z_2 \ne 0$. Therefore, $v_\Scal = z_\Scal - u_\Scal = z_2 \cdot (0, 1, -1)^T$. It follows from  $\wh g_i(v_\Scal)=| h^T_i v_\Scal|$ and $h_i$'s given before that $\wh g_1(v_\Scal)= 0$, $\wh g_2(v_\Scal)= \wh g_3(v_\Scal)= \frac{4}{3}|z_2|$, and $\wh g_4(v_\Scal)=\frac{5}{6}|z_2|$. Consequently, $\max_{i=1, 2, 3} | A^T_{\bullet i} A( z- u) | =\max_{i=1,2,3} \wh g_i(v_\Scal) >  \wh g_4(v_\Scal)=| A^T_{\bullet 4} A (z - u) |$. Due to the symmetry of the matrix $A$, it can be shown via a similar argument that the above result also holds for $z_{j_1} = A^T_{\bullet j_1} A z$ with $j_1=2$ or $j_1=3$. This completes the proof of the claim.
      \end{proof}

       By the above claim, we see that $\max_{i=1, 2, 3} | A^T_{\bullet i} A (z- x^1) | > | A^T_{\bullet 4} A (z - x^1) |$ for the vector $x^1$ obtained from Step 1. Therefore, $j^*_2 \in \Scal$ and $j^*_2 \ne j^*_1$ in view of  Lemma~\ref{lem:exact_suppt_recovery_index}.
       %
       %Theorem~\ref{thm:nec_suf_condition_for_exact_supp_recovery}.
       %
       %%% Proposition~\ref{prop:optimal_index}.
%
       Hence, $\Jcal_2=\{j^*_1, j^*_2\}$, and $x^2:=\argmin_{\supp(w) \subseteq \Jcal_{2} } \, \| y - A w\|^2_2$ is given by $x^2_{\Jcal_2}= \big( A^T_{\bullet \Jcal_2} A_{\bullet \Jcal_2} \big)^{-1} A^T_{\bullet \Jcal_2} A_{\bullet \Scal} z_\Scal$, and $x^2_i=0$ for $i\notin \Jcal_2$.

   \gap

   \indent $\bullet$ Step 3: Note that for any index set $\Ical\in \big\{ \{1, 2\}, \{1, 3\}, \{2, 3\}\big\}$, it follows from a direct calculation on the matrix $A$ that
        \[
          w = \big( A^T_{\bullet \Ical} A_{\bullet \Ical} \big)^{-1} A^T_{\bullet \Ical} A_{\bullet \Scal} z_\Scal = \begin{bmatrix} z_s - \frac{1}{2} z_p \\ z_t -\frac{1}{2} z_p \end{bmatrix},
        \]
        where $s, t \in \Ical$ with $s<t$, and $p\in \Scal \setminus \Ical$. Hence, (i) if $\Jcal_2=\{1, 2\}$, then $(z-x^2)_\Scal = z_3\cdot (\frac{1}{2}, \frac{1}{2}, 1)^T$; (ii) if $\Jcal_2=\{1, 3\}$, then $(z-x^2)_\Scal = z_2\cdot (\frac{1}{2}, 1, \frac{1}{2})^T$; and (iii) if
        $\Jcal_3=\{2, 3\}$, then $(z-x^2)_\Scal = z_1\cdot (1, \frac{1}{2}, \frac{1}{2})^T$. Therefore, for the vector $x^2$ obtained from Step 2, we have $\supp(z-x^2)=\Scal$. It follows from (\ref{eqn:claim_I}) that $\max_{i=1, 2, 3} | A^T_{\bullet i} A (z- x^2) | > | A^T_{\bullet 4} A (z - x^2) |$. This shows that $j^*_3 \in \Scal$ with $j^*_3 \notin \Jcal_2$. Hence, $\Jcal_3=\Jcal_2\cup\{j^*_3\}=\Scal$. Since $A_{\bullet \Scal}$ has full column rank, we see that $x^3:=\argmin_{\supp(w) \subseteq \Jcal_{3} } \, \| y - A w\|^2_2$ satisfies $x^3=z$. This shows that $z$ is uniquely recovered via the OMP using the matrix $A$.
%%\end{itemize}
 %
\end{proof}

%%\newpage

\begin{remark} \rm \label{remark:counter_example}
We make a few remarks about the counterexample constructed above.
\begin{itemize}
  \item [(a)] It is easy to verify that for the given matrix $A$ in (\ref{eqn:A_matrix_counterexample}),  when $v=\alpha \cdot (1, 1, 0)^T$ for any $0 \ne \alpha \in \mathbb R$, $\wh g_1(v)=\wh g_2(v)= \wh g_3(v)= \wh g_4(v)= \frac{2}{3}|\alpha|$. Hence, $\max_{i=1,2,3} \wh g_i(v) = \wh g_4(v)$. Letting $z=(z_\Scal, z_{\Scal^c})\in \mathbb R^4$ with $z_\Scal =v$ and $z_{\Scal^c}=0$, we have $\max_{i=1, 2, 3} g_i(z) = g_4(z)$, leading to a counterexample to the implication (\ref{eqn:necessity_implication}) used in the necessity proof for \cite[Proposition 3.5]{FoucartRauhut_book2013}. Besides, letting $\wt m=|\Scal|=3$, since $A$ is invertible, all the $\wt m$-term representations are unique, and the condition $\| (A^T_{\bullet \Scal} A_{\bullet \Scal} )^{-1} A^T_{\bullet \Scal} A_{\bullet \Scal^c} \|_1 = 1$ implies the failure of the ``Exact Recovery Condition'' defined in Tropp's paper \cite{Tropp_ITI04} (i.e., $\| (A^T_{\bullet \Scal} A_{\bullet \Scal} )^{-1} A^T_{\bullet \Scal} A_{\bullet \Scal^c} \|_1 < 1$). However, any $z$ with $\supp(z)=\Scal$ can be exactly recovered via the OMP, yielding a counterexample to  \cite[Theorem 3.10]{Tropp_ITI04}.

  \item [(b)] There are multiple $4\times 4$ real matrices satisfying the conditions specified in Theorem~\ref{thm:counterexample} as long as their columns are unit and the inner products of their distinct columns defined by $\vartheta_{ij}$ equal to the values given in (\ref{eqn:vartheta_value}). In particular, for the matrix $A$ given in (\ref{eqn:A_matrix_counterexample}) and any orthogonal matrix $P \in \mathbb R^{4 \times 4}$, $PA$ also satisfies the  conditions in Theorem~\ref{thm:counterexample}.
\end{itemize}
\end{remark}

The counterexample constructed in the previous theorem can be extended to one with a larger size.

\begin{corollary}
 Suppose an index set $\Scal \subseteq\{ 1, \ldots, N\}$ is of size 3, i.e., $|\Scal|=3$. Then for any $m \ge 4$ and $N \ge 4$, there exists a matrix $\wh A \in \mathbb R^{m\times N}$ with unit columns such that $\wh A$ achieves the exact vector recovery for the fixed support $\Scal$ via the OMP, $\wh A_{\bullet \Scal}$ has full column rank, and $\big\| (\wh A^T_{\bullet \Scal} \wh A_{\bullet \Scal} )^{-1} \wh A^T_{\bullet \Scal} \wh A_{\bullet \Scal^c} \big\|_1 = 1$.
\end{corollary}

\begin{proof}
Without loss of generality, let $\Scal=\{1, 2, 3\}$. For any $N \ge 4$, define the matrix $B \in \mathbb R^{4\times N}$ as
$
   B := \begin{bmatrix} A & B_{\bullet 5} & \cdots & \cdots & B_{\bullet N} \end{bmatrix},
$
where the matrix $A$ is given in (\ref{eqn:A_matrix_counterexample}), and $B_{\bullet k}=\pm A_{\bullet 4}$ for each $k \ge 5$. Then let
$
    \wh A := \begin{bmatrix} B \\ 0_{(m-4)\times N} \end{bmatrix} \in \mathbb R^{m\times N}.
$
Straightforward calculations show that $\wh A$ satisfies the desired properties by observing that almost all the required properties of $\wh A$ rely on $\langle \wh A_{\bullet i}, \wh A_{\bullet j}\rangle$'s, which are defined by $\vartheta_{ij}$'s or $h_i$'s of the matrix $A$.
\end{proof}

%--------------------------------------------------------------------------
%
\subsection{Exact Vector Recovery on the Nonnegative Orthant $\mathbb R^N_+$ for a Fixed Support}

We consider the exact vector recovery on the nonnegative orthant $\mathbb R^N_+$ for a fixed support $\Scal$ using constrained matching pursuit. Without loss of generality, we assume that the matrix $A \in \mathbb R^{m\times N}$ has unit columns, i.e., $\|A_{\bullet i}\|_2=1$ for each $i=1, \ldots, N$. A necessary condition is given as follows.

\begin{lemma} \label{lem:nonnegative orthant_Nec01}
 Given a matrix $A \in \mathbb R^{m\times N}$ with unit columns and an index set $\Scal$ of size $s$, the exact vector recovery of every nonzero vector $x \in \mathbb R^N_+$ with $\supp(x)=\Scal$ is achieved via constrained matching pursuit  only if $A_{\bullet \Scal}$ has full column rank.
\end{lemma}

\begin{proof}
Assume, in contrast, that $A_{\bullet \Scal}$ does not have full column rank. Let $r:=|\Scal|$. Then there exist a nonzero vector $v \in \mathbb R^{r}$ such that $A_{\bullet \Scal} v =0$.
For a given nonzero $x \ge 0$ with $\supp(x)=\Scal$, suppose at the $r$th step, the exact support of $x$ is recovered from $y=A x$ via  constrained matching pursuit. It follows from Algorithm~\ref{algo:constrained_MP}
  that one need to solve the constrained minimization problem
$
 {\bf Q}: \ \min_{w \in \mathbb R^{r}_+} \| A_{\bullet \Scal} w -  y \|^2_2,
$
where $y=A_{\bullet \Scal} x_\Scal$, to recover $x_\Scal$. Since $x_\Scal>0$ and $v \ne 0$, there exists a small positive constant $\varepsilon$ such that $x_\Scal + \varepsilon v >0$. Noting that $A_\Scal  (x_\Scal + \varepsilon v)= A_\Scal x_\Scal = y$, we see that $x_\Scal + \varepsilon v$ is a solution to the minimization problem ${\bf Q}$. Hence, ${\bf Q}$ has multiple optimal solutions which can be different from the desired solution $x_\Scal$. This leads to a contradiction. Consequently, $A_{\bullet \Scal}$ has full column rank.
\end{proof}

In light of statement (ii) of Corollary~\ref{coro:Nec_Suf_suppt_recovery_RN_RN+} for $x^0=0$, we easily obtain another necessary condition for the exact support recovery (and thus exact vector recovery) of any $z \in \mathbb R^N_+$ with $\supp(z)=\Scal$:
%%%%\begin{equation} \label{eqn:R+_x0=0}
\[
     \max_{j \in \supp(z)} (A^T_{\bullet j} A z )_+ \, > \, \max_{j\in [\supp(z)]^c} ( A^T_{\bullet j} A z )_+, \qquad \forall \ z \in \mathbb R^N_+ \ \mbox{ with } \supp(z)=\Scal,
\]
%%%\end{equation}
which is equivalent to $\| (A^T_{\bullet \Scal} A_{\bullet \Scal} v )_+ \|_\infty > \| (A^T_{\bullet \Scal^c} A_{\bullet \Scal} v )_+ \|_\infty$  for all $v \in \mathbb R^{|\Scal|}_{++}$.

\mycut{
In what follows, we show that under mild assumptions, the negation of this condition leads to the failure of the exact support recovery in the strong sense defined in Remark~\ref{remark:failure_recovery_def} for those  $z \in \mathbb R^N_+$ with $\supp(z)=\Scal$ and $z_\Scal$ violating (\ref{eqn:R+_x0=0}) subject to arbitrarily small perturbations. For this purpose, we present a lemma first.

\begin{lemma} \label{lem:positive_sign}
 Let $M \in \mathbb R^{m\times m}$ be a positive definite matrix. Then for any $z \in \mathbb R^m$ with $z > 0$, there exists $i\in \{1, \ldots, m\}$ such that $(M z)_i>0$.
\end{lemma}

\begin{proof}
  %%%We prove this result by contradiction.
  Suppose there exists $z>0$ such that $M z \le 0$. Since $z>0$, we have $z^T Mz \le 0$. On the other hand, since $M$ is positive definite, we deduce that $z^T M z =0$ so that $z =0$. This yields a contradiction.
\end{proof}

\begin{proposition}
 Given a matrix $A \in \mathbb R^{m\times N}$ with unit columns and an index set $\Scal$, suppose $A_{\bullet \Scal}$ has full column rank and $A_{\bullet i} \ne A_{\bullet j}$ for any $i \ne j$. Then
 for almost all $z \in \mathbb R^N_+$ with $\supp(z)=\Scal$ satisfying $\max_{j \in \Scal} (A^T_{\bullet j} A z )_+  \le \max_{j\in \Scal^c} ( A^T_{\bullet j} A z)_+$, the exact support recovery of $z$ fails in the strong sense, i.e., $\max_{j \in \Scal} (A^T_{\bullet j} A z )_+  < \max_{j\in \Scal^c} ( A^T_{\bullet j} A z)_+$.
%
%
% if $z \in \mathbb R^N_+$ with $\supp(z)=\Scal$ is such that $\max_{j \in \Scal} (A^T_{\bullet j} A z )_+  %\le \max_{j\in \Scal^c} ( A^T_{\bullet j} A z)_+$, then for any $\varepsilon>0$, almost all $z' \in \mathbb %R^N_+$ with $\supp(z')=\Scal$ and $\| z' - z \| < \varepsilon$ are such that the exact support recovery of %$z'$ fails in the strong sense, i.e., $\max_{j \in \Scal} (A^T_{\bullet j} A z' )_+  < \max_{j\in \Scal^c} ( %A^T_{\bullet j} A z')_+$.
%%
\end{proposition}

\begin{proof}
 % The case where $\max_{j \in \Scal} (A^T_{\bullet j} A z )_+  < \max_{j\in \Scal^c} ( A^T_{\bullet j} A %z)_+$ is trivial due to the continuity of the functions in $z$ on the two sides of the inequality. We next %consider the case where the equality holds at $z$, i.e., $\max_{j \in \Scal} (A^T_{\bullet j} %A_{\bullet\Scal} z_\Scal )_+  = \max_{j\in \Scal^c} ( A^T_{\bullet j} A_{\bullet\Scal} z_\Scal )_+$, where %$z_\Scal>0$.
 %

 %
  Define the two sets
 \begin{align*}
    \mathcal U' & \, := \, \left\{ \, u \in \mathbb R^{|\Scal|}_{++} \, | \, \max_{j\in \Scal} ( A^T_{\bullet j} A_{\bullet\Scal} u )_+ \le \max_{j \in \Scal^c}( A^T_{\bullet j} A_{\bullet\Scal} u )_+ \, \right\}, \\
   \mathcal U & \, := \, \left\{ \, u \in \mathbb R^{|\Scal|}_{++} \, | \, \max_{j\in \Scal} ( A^T_{\bullet j} A_{\bullet\Scal} u )_+ = \max_{j \in \Scal^c}( A^T_{\bullet j} A_{\bullet\Scal} u )_+ \, \right\}.
  \end{align*}
 Since $A_{\bullet \Scal}$ has full column rank, $A^T_{\bullet \Scal} A_{\bullet \Scal}$ is positive definite. Hence, by Lemma~\ref{lem:positive_sign}, we see that $\max_{j \in \Scal} (A^T_{\bullet j} A_{\bullet\Scal} u )_+>0$ for all $u \in \mathbb R^{|\Scal|}$.
 In view of this result, we easily show that $\mathcal U$ is contained in the intersection of sets of the following form:
 \[
 \mathcal V \, := \, \Big \{ u \in \mathbb R^{|\Scal|}_{++} \, | \, A^T_{\bullet j_1} A_{\bullet\Scal} u = A^T_{\bullet j_2} A_{\bullet\Scal} u, \mbox{ for some $j_1 \in \Scal$ and  $j_2 \in \Scal^c$} \Big\}.
 \]
 Recall that $\vartheta_{ij}:=\langle A_{\bullet i}, A_{\bullet j} \rangle$ for $i, j \in \{1, \ldots, N\}$. Hence, the $j_1$-th element of $A^T_{\bullet\Scal} (A_{\bullet j_1} - A_{\bullet j_2})$ is given by $1-\vartheta_{j_1, j_2}$. Since $A_{\bullet j_1}, A_{\bullet j_2}$ are distinct unit vectors, we have $\vartheta_{j_1, j_2}<1$ so that $A^T_{\bullet j_1} A_{\bullet\Scal} \ne A^T_{\bullet j_2} A_{\bullet\Scal}$. This implies that $\mathcal V$ is a lower-dimensional closet set of zero measure, and so is $\mathcal U$. Therefore, $\mathcal U' \setminus \mathcal U$ is open and $v \in \mathcal U' \setminus \mathcal U$ for almost all $v \in \mathcal U'$. This shows that for almost all $v \in \mathcal U'$,  $\max_{j \in \Scal} (A^T_{\bullet j} A_{\bullet \Scal} v )_+  < \max_{j\in \Scal^c} ( A^T_{\bullet j}  A_{\bullet \Scal} v )_+$.
 %
 %Since $A^T_{\bullet \Scal} A_{\bullet \Scal}$ is positive definite and $A_{\bullet i} \ne A_{\bullet j}$ for %any $i \ne j$, we must have $A^T_{\bullet j_1} A_{\bullet\Scal} \ne A^T_{\bullet j_2} A_{\bullet\Scal}$ so %that $\mathcal V$ is a lower-dimensional closet set of zero measure, and so is $\mathcal U$. Therefore, if %$z \in \mathbb R^N_+$ with $\supp(z)=\Scal$ is such that $\max_{j \in \Scal} (A^T_{\bullet j} A z )_+  \le %\max_{j\in \Scal^c} ( A^T_{\bullet j} A z)_+$,
 %
 %$\mathcal U^c$ is an open set and $v \in \mathcal U^c$ for almost all $v \in \mathbb R^{|\Scal|}_{++}$. %Hence, for any given $\varepsilon>0$, all but those a set of zero measure with $\| u - z_\Scal\| < %\varepsilon$
 %
\end{proof}
}

%---------------------------------------------------------------------------------
%
\subsubsection{Necessary and Sufficient Conditions for Exact Vector Recovery for a Fixed Support of Size 2}

We derive necessary and sufficient conditions for exact vector recovery on $\mathbb R^N_+$ for a given support $\Scal$ with $|\Scal|=2$. Recall that $\vartheta_{ij}:=\langle A_{\bullet i}, A_{\bullet j} \rangle$ for $i, j \in \{1, \ldots, N\}$. Besides, the following lemma is needed.

%
%Recall that $\vartheta_{ij}:=\langle A_{\bullet i}, A_{\bullet j} \rangle$ for $i, j \in \{1, \ldots, N\}$, %and for each $i$, define the function
%\begin{equation} \label{eqn:gi_+_def}
%  g^+_i(z) \, :=  \,  \langle A_{\bullet i}, A z \rangle_+ \, = \, \Bigg( \sum^N_{j=1} \vartheta_{ij} z_j   %\Bigg)_+, \qquad \quad \forall \, z =(z_1, \ldots, z_N)^T \in \mathbb R^N.
%\end{equation}
%

\begin{lemma} \label{lem:positive_sign}
 Let $M \in \mathbb R^{m\times m}$ be a positive definite matrix. Then for any $z \in \mathbb R^m$ with $z > 0$, there exists $i\in \{1, \ldots, m\}$ such that $(M z)_i>0$.
\end{lemma}

\begin{proof}
  %%%We prove this result by contradiction.
  Suppose, in contrast, that there exists $z>0$ such that $M z \le 0$. Since $z>0$, we have $z^T Mz \le 0$. As $M$ is positive definite, we deduce that $z^T M z =0$ so that $z =0$. This yields a contradiction.
\end{proof}

\begin{theorem} \label{thm:nonnegative constraint_S2}
 Given a matrix $A \in \mathbb R^{m\times N}$ with unit columns and the index set $\Scal=\{1, 2\}$, every nonzero vector $x \in \mathbb R^N_+$ with $\supp(x)=\Scal$ is recovered from $y=A x$ via constrained matching pursuit if and only if the following conditions hold:
 \begin{itemize}
   \item [(i)] $A_{\bullet \Scal}$ has full column rank or equivalently $|\vartheta_{12}|<1$;
    \item [(ii)]
   $ \max \big( (z_1 + \vartheta_{12} z_2)_+, \, (\vartheta_{12} z_1 + z_2)_+  \big) > \max_{j \in \Scal^c} \big(\vartheta_{j1} z_1 + \vartheta_{j2} z_2 \big)_+, \ \forall \, (z_1, z_2)^T \in \mathbb R^2_{++}$;
    \item [(iii)] $1-\vartheta^2_{12} \, > \, \max_{j\in \Scal^c} \big( \, (\vartheta_{j2}- \vartheta_{12}\vartheta_{j1})_+, \ (\vartheta_{j1}- \vartheta_{12}\vartheta_{j2})_+ \, \big)$.
 %%%\end{equation}
 \end{itemize}
\end{theorem}

\begin{proof}
 ``Only if''. Clearly, the condition that $A_{\bullet \Scal}$ has full column rank is necessary for the exact vector recovery in view of Lemma~\ref{lem:nonnegative orthant_Nec01}. Since $A_{\bullet \Scal}$ has full column rank if and only if $A^T_{\bullet \Scal} A_{\bullet \Scal} = \begin{bmatrix} 1 & \vartheta_{12} \\ \vartheta_{12} & 1 \end{bmatrix}$ is positive definite, we see that $A_{\bullet \Scal}$ has full column rank if and only if $|\vartheta_{12}|<1$. For an arbitrary $z \in \mathbb R^N$ with $z_\Scal=(z_1, z_2)>0$, let $y= A z = A_{\bullet \Scal} z_\Scal$. At Step 1, since $x^0=0$, it follows from statement (ii) of Corollary~\ref{coro:Nec_Suf_suppt_recovery_RN_RN+} that any $j^*_1 \in \Scal$ if and only if
 $
   \max_{j\in \Scal} \langle A_{\bullet j}, A z \rangle_+ > \max_{j\in \Scal^c} \langle A_{\bullet j}, A z \rangle_+.
 $
 This leads to condition (ii), in light of $\langle A_{\bullet 1}, A z \rangle_+ = (z_1 + \vartheta_{12} z_2)_+$,  $\langle A_{\bullet 2}, A z \rangle_+ = (\vartheta_{12}z_1 +  z_2)_+$, and $\langle A_{\bullet j}, A z \rangle_+ = (\vartheta_{j1} z_1 + \vartheta_{j2} z_2 )_+$.
 Since $\begin{bmatrix} 1 & \vartheta_{12} \\ \vartheta_{12} & 1 \end{bmatrix}$ is positive definite, it follows from Lemma~\ref{lem:positive_sign} that for any $(z_1, z_2)>0$, at least one of $\vartheta_{12}z_1 +  z_2$ and $\vartheta_{j1} z_1 + \vartheta_{j2} z_2$ is positive.
%
% We claim that for any $(z_1, z_2)>0$, at least one of $\vartheta_{12}z_1 +  z_2$ and $\vartheta_{j1} z_1 + %\vartheta_{j2} z_2$ is positive. In fact,
% %
% if   $z_1 + \vartheta_{12} z_2\le 0$ and $\vartheta_{12} z_1 + z_2\le 0$, then
% \[
%    \begin{pmatrix} z_1 & z_2 \end{pmatrix} \begin{bmatrix} 1 & \vartheta_{12} \\ \vartheta_{12} & 1 %\end{bmatrix} \begin{pmatrix} z_1 \\ z_2 \end{pmatrix} \le 0 \ \Longrightarrow \ z_1=z_2=0,
% \]
% which yields a contradiction.
%
% Therefore, at least one of $(z_1 + \vartheta_{12} z_2)_+$ and $(\vartheta_{12} z_1 + z_2)_+$ is positive for %any $(z_1, z_2)>0$.
 %
 %
 Further, in view of $|\vartheta_{12}|<1$ and  the fact that for $a, b \in \mathbb R$, $b_+>a_+$ if and only if $b>0$ and $b>a$, it is easy to verify
 that for any $(z_1, z_2)>0$, (a) $(z_1 + \vartheta_{12} z_2)_+ > (\vartheta_{12} z_1 + z_2)_+$ if and only if $z_1>z_2$; (b) $(z_1 + \vartheta_{12} z_2)_+ < (\vartheta_{12} z_1 + z_2)_+$ if and only if $z_1 < z_2$; and (c) $(z_1 + \vartheta_{12} z_2)_+ = (\vartheta_{12} z_1 + z_2)_+>0$ if and only if $z_1 =z_2$.
  Hence, we have that $j^*_1=1$ if $z_1>z_2>0$, $j^*_1=2$ if $z_2>z_1>0$, and $j^*_1\in\{1, 2\}$ if $z_1=z_2>0$. Moreover,  $\Jcal_1=\{ j^*_1\}$, and $x^1:=\argmin_{w \ge 0, \supp(w) \subseteq \Jcal_{1} } \, \| A_{\bullet \Scal} z_\Scal - A w\|^2_2$ is given by $x^1= \langle A_{\bullet \Scal} z_{\Scal}, A_{\bullet j^*_1} \rangle_+ \cdot \mathbf e_{j^*_1}$, where $\langle A_{\bullet \Scal} z_{\Scal}, A_{\bullet j^*_1} \rangle_+ >0$ by Proposition~\ref{prop:index_set}.
%
%  $x^1_{j^*_1} = \langle A_{\bullet \Scal} z_{\Scal}, A_{\bullet j^*_1} \rangle >0$, and $x^1_i=0$ for all %$i\ne j^*_1$. Therefore, $x^1= \langle A_{\bullet 1} z_1 + A_{\bullet 2} z_2, A_{\bullet j^*_1} \rangle_+ %\cdot \mathbf e_{j^*_1}$.
%
  In what follows, we consider $j^*_1=1$ corresponding to $z_1\ge z_2>0$ first. In this case, $x^1= (z_1 + \vartheta_{12} z_2) \cdot \mathbf e_1$. Hence, $(z - x^1)_\Scal = (-\vartheta_{12}, 1)^T \cdot z_2$. It follows from  statement (ii) of Corollary~\ref{coro:Nec_Suf_suppt_recovery_RN_RN+}
   that a necessary and sufficient condition to select $j^*_2=2$ at Step 2 is
 \begin{equation}  \label{eqn:Step2_Nec}
    \langle A(z - x^1), A_{\bullet 2} \rangle_+  \, >  \, \max_{j \in \Scal^c} \langle A(z - x^1), A_{\bullet j} \rangle_+,
 \end{equation}
  where $\langle A(z - x^1), A_{\bullet 2} \rangle_+= (1-\vartheta^2_{12}) \cdot z_2$ and $\langle A(z - x^1), A_{\bullet j} \rangle_+ = (\vartheta_{j2}- \vartheta_{12}\vartheta_{j1})_+ \cdot z_2$ for each $j\in \Scal^c$. Hence, when $z_1\ge z_2>0$, an equivalent condition for (\ref{eqn:Step2_Nec}) is $1-\vartheta^2_{12} > \max_{j \in \Scal^c} (\vartheta_{j2}- \vartheta_{12}\vartheta_{j1})_+$. When $j^*_1=2$ corresponding to $z_2 \ge z_1>0$, we deduce via a similar argument that a necessary and sufficient condition for $j^*_2=1$ at Step 2 is $1-\vartheta^2_{12} > \max_{j \in \Scal^c}  (\vartheta_{j1}- \vartheta_{12}\vartheta_{j2})_+$. This gives rise to condition (iii).

``If''. As indicated in the ``only if'' part, condition (ii) is sufficient for $j^*_1 \in \Scal$ at Step 1, and condition (iii) is sufficient for $j^*_2 \in \Scal\setminus \{ j^*_1\}$ at Step 2. Hence, under conditions (ii) and (iii), the exact support $\Scal$ is recovered from $y=A z$ in two steps for any $z\in \mathbb R^N$ with $z_\Scal>0$, i.e., $\Jcal_2=\Scal$. Note that the optimality condition for $x^2:=\argmin_{w \ge 0, \supp(w) \subseteq \Jcal_{2} } \, \| A_{\bullet \Scal} z_\Scal - A w\|^2_2$ is given by the linear complementarity problem (LCP): $0 \le x^2_\Scal \perp A^T_{\bullet \Scal} A_{\bullet \Scal} (x^2_{\Scal} - z_\Scal) \ge 0$. Since $A_{\bullet \Scal}$ has full column rank, $A^T_{\bullet \Scal} A_{\bullet \Scal}$ is positive definite such that the LCP has a unique solution $x^2_\Scal = z_\Scal$ or equivalently $x^2=z$. This shows that the exact vector recovery is achieved for any $z\in \mathbb R^N$ with $z_\Scal>0$ under conditions (i)-(iii).
\end{proof}

Applying the necessary and sufficient  conditions given in Theorem~\ref{thm:nonnegative constraint_S2}, it is shown in the next corollary that condition $(\mathbf H)$ is necessary for the exact vector or support recovery on $\Sigma_2 \cap \mathbb R^N_+$.
%
%We omit the proof for this corollary since its argument is similar to that for %Corollary~\ref{coro:condition_H'_necessary_RN_S2}.
%

\begin{corollary} \label{coro:condition_H'_necessary_RN+_S2}
 Let $A \in \mathbb R^{m\times N}$ be a matrix with unit columns. Then  the exact vector recovery on $\Sigma_2 \cap \, \mathbb R^N_+$ is achieved if and only if (i) condition $(\mathbf H)$ holds on $\Sigma_2 \cap \, \mathbb R^N_+$, and (ii) any two distinct columns of $A$ are linearly independent.
\end{corollary}

\begin{proof}
The ``if'' part is similar to that given in the proof of Corollary~\ref{coro:condition_H'_necessary_RN_S2}. For the ``only if'' part, let $A$ achieve the exact vector recovery on $\Sigma_2 \cap \, \mathbb R^N_+$. Clearly, condition (ii) is necessary in light of Lemma~\ref{lem:nonnegative orthant_Nec01}. To show that condition (i) is necessary, we consider an arbitrary $z \in \Sigma_2 \cap \, \mathbb R^N_+$ with $\supp(z)=\{1, 2\}:=\Scal$. Hence, $A$ achieves the exact support recovery for the fixed support $\Scal$. Therefore, conditions (ii) and (iii) of Theorem~\ref{thm:nonnegative constraint_S2} hold.
 Consider the three proper subsets of $\Scal$, i.e., $\Jcal=\emptyset$, $\Jcal=\{1\}$, and $\Jcal=\{2 \}$. When $\Jcal=\emptyset$, we see that the inequality (\ref{eqn:condition_H'}) holds for $u=z$ and $v=0$ in light of statement (ii) of Corollary~\ref{coro:Nec_Suf_suppt_recovery_RN_RN+} and  conditions (ii) of Theorem~\ref{thm:nonnegative constraint_S2}.
 Furthermore, we have either $(z_1 + \vartheta_{12} z_2)_+ \ge (\vartheta_{12} z_1 + z_2)_+$ or $(z_1 + \vartheta_{12} z_2)_+ \le (\vartheta_{12} z_1 + z_2)_+$. For the former case, we deduce from Algorithm~\ref{algo:constrained_MP} that $j^*_1=1$ and $\Jcal_1=\{ 1 \}$ such that $x^1=(A^T_{\bullet 1} A z)_+ \mathbf e_1$ is the unique optimal solution to $\min_{w\ge 0, \supp(w) \subseteq \Jcal_1} \| A(z - w) \|^2_2$. Hence, the exact support recovery of $z$ shows that $f^*_2(z, x^1)< \min_{j \in \Scal^c} f^*_j(z, x^1)$, yielding (\ref{eqn:condition_H'}) for $u=z$ and $v=x^1$ when $\Jcal=\{ 1 \}$.
 We then consider $\Jcal=\{2\}$. Similarly, the unique optimal solution $v^*$ to $\min_{w \ge 0, \supp(w) \subseteq \Jcal} \| A(z - w) \|^2_2$ is given by $v^*=(A^T_{\bullet 2} A z)_+ \mathbf e_2 = (\vartheta_{12} z_1 + z_2)_+ \mathbf e_2$. Consider two sub-cases:
 \begin{itemize}
   \item [(a)] $(\vartheta_{12} z_1 + z_2)_+ \le 0$. In this case, $v^*=0$ such that $z-v^*=z$. Hence, $(A^T_{\bullet 1} A (z - v^*))_+=(z_1+\vartheta_{12} z_2)_+$ and $(A^T_{\bullet j} A (z-v^*))_+= (\vartheta_{j1}z_1 + \vartheta_{j2} z_2)_+$ for $j \in \Scal^c$. Since $ \max \big( (z_1 + \vartheta_{12} z_2)_+, \, (\vartheta_{12} z_1 + z_2)_+  \big) = (z_1 + \vartheta_{12} z_2)_+$, we deduce via condition (ii) of Theorem~\ref{thm:nonnegative constraint_S2} that $f^*_1(z, v^*)  < \min_{j\in \Scal^c} f^*_j(z, v^*)$, yielding the inequality (\ref{eqn:condition_H'}) for $u=z$ and $v=v^*$ when $\Jcal=\{ 2 \}$.
   \item [(b)]  $(\vartheta_{12} z_1 + z_2)_+ \ge 0$.   In this case, $z-v^*= (1, -\vartheta_{12}) z_1$ such that $(A^T_{\bullet 1} A (z - v^*))_+=(1 - \vartheta^2_{12}) \cdot z_1$ and $(A^T_{\bullet j} A (z-v^*))_+= (\vartheta_{j1} - \vartheta_{12} \vartheta_{j2})_+ \cdot z_1$ for $j \in \Scal^c$, where $z_1>0$. By condition (iii) of Theorem~\ref{thm:nonnegative constraint_S2} that $f^*_1(z, v^*)  < \min_{j\in \Scal^c} f^*_j(z, v^*)$, yielding (\ref{eqn:condition_H'}) for $u=z$ and $v=v^*$ when $\Jcal=\{ 2 \}$.
 \end{itemize}
The other case where $(z_1 + \vartheta_{12} z_2)_+ \le (\vartheta_{12} z_1 + z_2)_+$ can be established similarly. In addition, for any $u\in \Sigma_2\cap \mathbb R^N_+$ with $|\supp(u)|=1$ and $\Jcal=\emptyset$,  (\ref{eqn:condition_H'}) also holds. Hence, condition $(\mathbf H)$ holds on $\Sigma_2 \cap \, \mathbb R^N_+$.
\end{proof}

%---------------------------------------------------------------------------------
%
\subsubsection{Necessary and Sufficient Conditions for Exact Vector Recovery for a Fixed Support of Size 3}

We first present some preliminary results. Given a (possibly non-square) matrix
\[
 M = \begin{bmatrix} M_{11} & M_{12} \\ M_{21} & M_{22} \end{bmatrix},
\]
where $M_{ij}$'s are submatrices of $M$ with $M_{11}$ being invertible, the Schur complement of $M_{11}$ in $M$, denoted by $M/M_{11}$, is given by $M/M_{11}:= M_{22} - M_{21} M^{-1}_{11} M_{12}$. When $M$ is square, the Schur determinant formula says that $\det (M/M_{11}) = \det M / \det M_{11}$ \cite[Proposition 2.3.5]{CPStone_book92}. Particularly, when $M$ is positive definite, any of its Schur complement is also positive definite.

\begin{lemma} \label{lem:R+_optimal_solution}
Given a matrix $A\in \mathbb R^{m\times N}$ and an index set $\Scal$ such that $A_{\bullet \Scal}$ has full column rank, let the matrix $M:=A^T_{\bullet \Scal} A_{\bullet \Scal}$. For a nonempty index set $\Jcal \subset \Scal$ and $z \in \mathbb R^N_+$ with $\supp(z)=\Scal$, let $x^*$ be the unique solution to $\min_{w \ge 0, \supp(w) \subseteq \Jcal} \| A (z-w) \|^2_2$ whose support is given by $\Jcal^*$, i.e., $\supp(x^*)=\Jcal^*$. Define the index set $\Ical :=\Scal \setminus \Jcal^*$. Then $A^T_{\bullet \Jcal^*}A (z- x^*)=0$,  and
\[
   A^T_{\bullet \Ical}A (z- x^*) = \big( M/M_{\Jcal^* \Jcal^*} \big) \cdot z_{\Ical}, \quad
   A^T_{\bullet \Scal^c}A (z- x^*) =  A^T_{\bullet \Scal^c} \big[ I - A_{\bullet \Jcal^*} \big( A^T_{\bullet \Jcal^*} A_{\bullet \Jcal^*} \big)^{-1} A^T_{\bullet \Jcal^*} \big] A_{\bullet \Ical} \cdot z_{\Ical}.
\]
Moreover, $\max_{j\in\Ical} [A^T_{\bullet j}A(z - x^*)]_+=\max_{j\in\Scal\setminus \Jcal} [A^T_{\bullet j}A(z - x^*)]_+>0$.
\end{lemma}

\begin{proof}
Since $x^*$ is the unique optimal solution to $\min_{w \ge 0, \supp(w) \subseteq \Jcal} \| A (z- w) \|^2_2$, we have $x^*=(x^*_\Jcal, 0)$, where $x^*_\Jcal$ is the solution to $\min_{u \ge 0} \| A_{\bullet \Jcal} u - A_{\bullet \Scal} z_\Scal\|^2$, and $A_{\bullet \Jcal}$ has full column rank.
 Therefore, $x^*_\Jcal$ is a solution to the linear complementarity problem: $0 \le u \perp A^T_{\bullet \Jcal} A_{\bullet \Jcal} u - A^T_{\bullet \Jcal} A_{\bullet \Scal} z_\Scal \ge 0$. In view  of $\supp(x^*)=\Jcal^* \subseteq \Jcal$, we deduce that $x^*_{\Jcal^*} = \big( A^T_{\bullet \Jcal^*} A_{\bullet \Jcal^*} )^{-1}  A^T_{\bullet \Jcal^*} A_{\bullet \Scal} z_\Scal > 0$. Using $\Ical = \Scal \setminus \Jcal^*$ and $A_{\bullet \Scal} z_\Scal= A_{\bullet \Jcal^*} z_{\Jcal^*} + A_{\bullet \Ical} z_\Ical$, we further have $x^*_{\Jcal^*} = z_{\Jcal^*} + \big( A^T_{\bullet \Jcal^*} A_{\bullet \Jcal^*} )^{-1}  A^T_{\bullet \Jcal^*} A_{\bullet \Ical} z_\Ical$. Hence,
\[
  A (z - x^*) = A_{\bullet \Scal} (z_\Scal - x^*_\Scal ) = A_{\bullet \Jcal^*} \big( z_{\Jcal^*} - x^*_{\Jcal^*} \big) + A_{\bullet \Ical}  z_\Ical  = \big[ -A_{\bullet \Jcal^*} \big( A^T_{\bullet \Jcal^*} A_{\bullet \Jcal^*} )^{-1}  A^T_{\bullet \Jcal^*} A_{\bullet \Ical} + A_{\bullet\Ical} \big] z_\Ical.
%
%  \begin{bmatrix} A_{\bullet \Jcal^*} & A_{\bullet \Ical} \end{bmatrix} \begin{pmatrix} z_{\Jcal^*} - %x^*_{\Jcal^*} \\ z_\Ical \end{pmatrix}
%
\]
Direct calculations yield $A^T_{\bullet \Ical}A (z- x^*)  =  \big[ A^T_{\bullet \Ical} A_{\bullet\Ical} - A^T_{\bullet\Ical}A_{\bullet \Jcal^*} \big( A^T_{\bullet \Jcal^*} A_{\bullet \Jcal^*} )^{-1}  A^T_{\bullet \Jcal^*} A_{\bullet \Ical} \big] z_\Ical = \big( M/M_{\Jcal^* \Jcal^*} \big) \cdot z_{\Ical}$; the other equation also follow readily.

Since $M/M_{\Jcal^* \Jcal^*}$ is positive definite and $z_\Ical>0$, it follows from Lemma~\ref{lem:positive_sign} and the expression for $A^T_{\bullet \Ical}A (z- x^*)$ derived above that there exists an index $j\in \Ical$ such that $A^T_{\bullet j}A(z - x^*)>0$. Hence,  $\max_{j\in\Ical} [A^T_{\bullet j}A(z - x^*)]_+>0$. Furthermore, since $\Ical=\Scal\setminus \Jcal^*$ and $\Jcal^* \subseteq \Jcal \subset \Scal$, we have $\Ical = (\Scal \setminus \Jcal) \cup (\Jcal \setminus \Jcal^*)$. However, it follows from the linear complementarity condition for $x^*_\Jcal$ that $A^T_{\bullet \Jcal}A( x^* - z) = A^T_{\bullet \Jcal} A_{\bullet \Jcal} x^*_\Jcal - A^T_{\bullet \Jcal} A_{\bullet \Scal} z_\Scal \ge 0$, which implies that $A^T_{\bullet \Jcal}A(z- x^* ) \le 0$ or equivalently $[A^T_{\bullet j}A(z- x^* )]_+ =0$ for all $j \in \Jcal$. Therefore, $\max_{j\in\Ical} [A^T_{\bullet j}A(z - x^*)]_+=\max_{j\in\Scal\setminus \Jcal} [A^T_{\bullet j}A(z - x^*)]_+$.
%
%
%it suffices to show that there exist an index $j \in \Scal\setminus \Jcal$ such that $A^T_{\bullet j}A(z - %x^*)>0$. Since $M/M_{\Jcal^* \Jcal^*}$ is positive definite and $z_\Ical>0$, it follows from %Lemma~\ref{lem:positive_sign} and the expression for $A^T_{\bullet \Ical}A (z- x^*)$ derived before that %there exists an index $j\in \Ical$ such that $A^T_{\bullet j}A(z - x^*)>0$.
%
%
%$A^T_{\bullet \Jcal^*}A (z- x^*)=0$, and
%\begin{eqnarray*}
%A^T_{\bullet \Ical}A (z- x^*) & = & \big[ A^T_{\bullet \Ical} A_{\bullet\Ical} - A^T_{\bullet\Ical}A_{\bullet %\Jcal^*} \big( A^T_{\bullet \Jcal^*} A_{\bullet \Jcal^*} )^{-1}  A^T_{\bullet \Jcal^*} A_{\bullet \Ical} %\big] z_\Ical \, = \, \big( M/M_{\Jcal^* \Jcal^*} \big) \cdot z_{\Ical}, \\
% A^T_{\bullet \Scal^c}A (z- x^*) & = &  A^T_{\bullet \Scal^c} \Big[ I - A_{\bullet \Jcal^*} \big( %A^T_{\bullet \Jcal^*} A_{\bullet \Jcal^*} \big)^{-1} A^T_{\bullet \Jcal^*} \Big] A_{\bullet \Ical} \, %z_{\Ical}.
%\end{eqnarray*}
%This completes the proof.
%
\end{proof}

\begin{lemma} \label{lem:R+_S3_Step2}
 Let $U:= \begin{bmatrix} \alpha & \gamma \\ \gamma & \beta \end{bmatrix} \in \mathbb R^{2\times 2}$ be a positive definite matrix for real numbers $\alpha, \beta$ and $\gamma$. Define the set $\mathcal W :=\big\{ (u_1, u_2) \in \mathbb R^2_{++} \, | \, \big( \alpha u_1 + \gamma u_2 \big)_+ \ge \big( \gamma u_1 + \beta u_2 \big)_+ \big\}$. Then $\mathcal W$ is nonempty if and only if $\alpha > \gamma$. Furthermore, if $\mathcal W$ is nonempty, then $\{ u_2 \, | \, (u_1, u_2) \in \mathcal W \}=\mathbb R_{++}$.
\end{lemma}

\begin{proof}
  Since $U$ is positive definite, we have $\alpha>0$, $\beta>0$, and $\alpha\beta > \gamma^2$. To show the ``if'' part, suppose $\alpha> \gamma$. Then for a fixed $u_1>0$, we have $\alpha u_1 > \gamma u_1$ and $\alpha u_1>0$. Therefore, for a sufficiently small $u_2>0$, it is easy to see that $( \alpha u_1 + \gamma u_2 )_+ \ge ( \gamma u_1 + \beta u_2 )_+$. This shows that $\mathcal W$ is nonempty. To prove the ``only if'' part, suppose $\mathcal W$ is nonempty but $\alpha \le \gamma$. Note that this implies that $\gamma>0$. Since $\alpha \cdot \beta > \gamma^2$ (due to the positive definiteness of $U$), we have
  $
  \displaystyle   \beta > \frac{\gamma}{\alpha} \cdot \gamma \ge \gamma.
  $
  Therefore, $\beta>\gamma\ge \alpha>0$. Hence, for any $(u_1, u_2)>0$, we have $\alpha u_1 \le \gamma u_1$ and $\gamma u_2 < \beta u_2$ such that
  $
     0 < \alpha u_1 + \gamma u_2 < \gamma u_1 + \beta u_2.
  $
  This implies that $\mathcal W$ is empty, yielding a contradiction. Finally, when $\mathcal W$ is nonempty, we see, in view of $\alpha > \gamma$ proven above, that for any $u_2>0$, there exists a sufficiently large $u_1>0$ such that $\alpha u_1 + \gamma u_2>0$ and  $\alpha u_1 + \gamma u_2> \gamma u_1 + \beta u_2$. This shows that $( \alpha u_1 + \gamma u_2 )_+ > ( \gamma u_1 + \beta u_2 )_+$. Hence, $\{ u_2 \, | \, (u_1, u_2) \in \mathcal W \}=\mathbb R_{++}$.
\end{proof}

\begin{theorem} \label{thm:nonnegative constraint_S3}
 Given a matrix $A \in \mathbb R^{m\times N}$ with unit columns and the index set $\Scal=\{1, 2, 3\}$, let $M:=A^T_{\bullet \Scal} A_{\bullet \Scal}$. Then every nonzero vector $x \in \mathbb R^N_+$ with $\supp(x)=\Scal$ is recovered from $y=A x$ via constrained matching pursuit if and only if each of the following conditions holds:
 \begin{itemize}
   \item [(i)] $A_{\bullet \Scal}$ has full column rank;
    \item [(ii)] $\big\| (A^T_{\bullet \Scal} A_{\bullet \Scal} u )_+ \big\|_\infty > \big\| (A^T_{\bullet \Scal^c} A_{\bullet \Scal} u )_+ \big\|_\infty$  for all $u \in \mathbb R^3_{++}$;
    \item [(iii)] For any $\Jcal\in \{ \{1\}, \{2\}, \{3\}\}$, $\big\| (M/M_{\Jcal\Jcal} \, v )_+ \big\|_\infty > \big\| (A^T_{\bullet \Scal^c}[I- A^T_{\bullet \Jcal} A_{\bullet \Jcal}] A_{\Scal\setminus\Jcal}v  )_+ \big\|_\infty$  for all $v \in \mathbb R^2_{++}$;
       %
       % , where $M/M_{\Jcal\Jcal}$ is the Schur complement of $M_{\Jcal\Jcal}$ in $M$;
    \item [(iv)]  All the following implications hold:
      \begin{align*}
        \Big[ 1-\vartheta^2_{12} > \min(\Delta_{13}, \Delta_{23}) \Big] & \ \Longrightarrow \ \Big[ \det M > \max_{i \in \Scal^c} \big( \vartheta_{i3} (1-\vartheta^2_{12}) - \vartheta_{i1} \Delta_{13} - \vartheta_{i2}\Delta_{23} \big)_+ \Big], \\
        \Big[ 1-\vartheta^2_{13} > \min(\Delta_{12}, \Delta_{23}) \Big] & \ \Longrightarrow \ \Big[ \det M > \max_{i \in \Scal^c} \big( \vartheta_{i2} (1-\vartheta^2_{13}) - \vartheta_{i1} \Delta_{12} - \vartheta_{i3}\Delta_{23} \big)_+ \Big], \\
        \Big[ 1-\vartheta^2_{23} > \min(\Delta_{12}, \Delta_{13}) \Big] & \ \Longrightarrow \ \Big[ \det M > \max_{i \in \Scal^c} \big( \vartheta_{i1} (1-\vartheta^2_{23}) - \vartheta_{i2} \Delta_{12} - \vartheta_{i3}\Delta_{13} \big)_+ \Big],
      \end{align*}
       where $\Delta_{12}:=\vartheta_{12}-\vartheta_{13}\vartheta_{23}$, $\Delta_{13}:=\vartheta_{13}-\vartheta_{12}\vartheta_{23}$, and $\Delta_{23}:=\vartheta_{23} - \vartheta_{12}\vartheta_{13}$.
 \end{itemize}
\end{theorem}

\begin{remark} \rm \label{remark:R+3_conditions}
We comment on the above conditions before presenting a proof:
\begin{itemize}
 \item [(a)] Since the matrix $M = \begin{bmatrix} 1 & \vartheta_{12} & \vartheta_{13} \\ \vartheta_{12} & 1 & \vartheta_{23} \\ \vartheta_{13} & \vartheta_{23} & 1 \end{bmatrix}$, its determinant $\det M = 1 + 2\vartheta_{12} \vartheta_{13} \vartheta_{23} - \vartheta^2_{12} - \vartheta^2_{13}-\vartheta^2_{23}$.

 \item [(b)] If the hypothesis of an implication in condition (iv) fails, then that implication holds even when the conclusion statement is false. Hence, that implication is vacuously true and can be neglected.

 \item [(c)] Since each Schur complement of $M:=A^T_{\bullet \Scal} A_{\bullet \Scal}$ is positive definite, we have $(1-\vartheta^2_{13})(1-\vartheta^2_{23}) \ge \Delta^2_{12}$, $(1-\vartheta^2_{12})(1-\vartheta^2_{23}) \ge \Delta^2_{13}$, and $(1-\vartheta^2_{12})(1-\vartheta^2_{13}) \ge \Delta^2_{23}$. By virtue of these inequalities, it is easy to verify that at least two hypotheses of the three implications in condition (iv) must hold.
\end{itemize}

\end{remark}

\begin{proof}[Proof of Theorem~\ref{thm:nonnegative constraint_S3}]
``If''. Suppose conditions (i)-(iv) hold. Fix an arbitrary $z=(z_\Scal, 0) \in \mathbb R^N_+$ with $z_\Scal=(z_1, z_2, z_3)\in \mathbb R^3_{++}$, and let $y = A z = A_{\bullet S} z_\Scal$. Consider the following three steps of Algorithm~\ref{algo:constrained_MP}:

 \indent $\bullet$  Step 1: Let $x^0=0$. Since $y=A_{\bullet \Scal} z_\Scal$, it follows from condition (ii) that
  $\max_{i=1, 2, 3} ( A^T_{\bullet i} A_{\bullet \Scal} z_\Scal)_+ > \max_{j \in \Scal^c} (A^T_{\bullet j} A_{\bullet \Scal} z_\Scal)_+$. Hence, it follows from Algorithm~\ref{algo:constrained_MP} that $j^*_1\in \Scal= \{1, 2, 3\}$, and the index set $\Jcal_1=\{ j^*_1\}$. Further, $x^1:=\argmin_{x  \ge 0, \supp(x) \subseteq \Jcal_{1} } \, \| y - A x\|^2_2$ is given by $x^1 = \langle A_{\bullet j^*_1}, A_{\bullet \Scal} z_{\Scal} \rangle_+ \mathbf e_{j^*_1}$,
 %
 % and $x^1_i=0$ for all $i\ne j^*_1$, where
 %
   where $\langle A_{\bullet j^*_1}, A_{\bullet \Scal} z_{\Scal} \rangle_+ > 0$ in view of Proposition~\ref{prop:index_set}.

\indent $\bullet$  Step 2: By observing that $x^1$ is the optimal solution obtained from Step 1 with $\supp(x^1)=\Jcal_1=\{ j^*_1\}$, it follows from Lemma~\ref{lem:R+_optimal_solution} and $A^T_{\bullet \Jcal_1} A_{\bullet \Jcal_1}=1$ that by letting  the index set $\Ical := \Scal \setminus \Jcal_1$,
\begin{align*}
  \max_{i\in \Scal} \big( A^T_{\bullet i} A(z - x^1) \big)_+ & \, = \ \big \| \big( A^T_{\bullet  \Scal} A(z - x^1) \big)_+ \big\|_\infty \, = \, \big\| \big( (M/M_{\Jcal_1 \Jcal_1}) \cdot z_\Ical \big)_+ \big\|_\infty, \\
  \max_{j\in \Scal^c} \big( A^T_{\bullet j} A(z - x^1) \big)_+ & \, = \ \big \| \big( A^T_{\bullet  \Scal^c} A(z - x^1) \big)_+ \big\|_\infty \, = \, \big\|  (A^T_{\bullet \Scal^c}[I- A_{\bullet \Jcal_1} A^T_{\bullet \Jcal_1}] A_{\Ical} \cdot z_\Ical  )_+  \big\|_\infty.
\end{align*}
Noting that the Schur complement $M/M_{\Jcal_1 \Jcal_1}$ is positive definite and $z_\Ical>0$, we deduce via Lemma~\ref{lem:positive_sign} that $\big\| \big( (M/M_{\Jcal_1 \Jcal_1}) \cdot z_\Ical \big)_+ \big\|_\infty>0$.  By $z_\Ical>0$ and condition (iii), we have  $ \max_{i\in \Scal} \big( A^T_{\bullet i} A(z - x^1) \big)_+ >  \max_{j\in \Scal^c} \big( A^T_{\bullet j} A(z - x^1) \big)_+$. In light of Algorithm~\ref{algo:constrained_MP}, we see that $j^*_2:= \argmax_{i\in \Scal} \big( A^T_{\bullet i} A(z - x^1) \big)_+$ satisfies $j^*_2 \in \Ical$, and $\Jcal_2=\{ j^*_1, j^*_2\} \subset \Scal$ with $j^*_1 \ne j^*_2$. Moreover, let $x^2$ be the unique optimal solution to $\min_{w \ge 0, \supp(w)\subseteq \Jcal_2} \| y - A w \|^2$. Then it follows from Proposition~\ref{prop:index_set} that $\supp(x^2)=\Jcal_2$.

\indent $\bullet$  Step 3: Let the index $j_3$ be such that $\{ j_3 \}=\Scal \setminus \Jcal_2$.
 Note that $\{ j^*_2, j_3\}=\Ical$. Hence, the Schur complement $U:=M/M_{\Jcal_1 \Jcal_1}$ is one of the following $2\times 2$ positive definite  matrices:
\begin{equation} \label{eqn:U_matrices}
  U^1: = \begin{bmatrix} 1 - \vartheta^2_{12} & \Delta_{23} \\ \Delta_{23} & 1 - \vartheta^2_{13} \end{bmatrix}, \quad
  U^2: = \begin{bmatrix} 1 - \vartheta^2_{12} & \Delta_{13} \\ \Delta_{13} & 1 - \vartheta^2_{23} \end{bmatrix}, \quad
  U^3: = \begin{bmatrix} 1 - \vartheta^2_{13} & \Delta_{12} \\ \Delta_{12} & 1 - \vartheta^2_{23} \end{bmatrix},
\end{equation}
where $\Delta_{ij}$'s are defined in condition (iv), $U^1= M/M_{11}$, $U^2= M/M_{22}$, and $U^3= M/M_{33}$. Hence, $U=\begin{bmatrix} \alpha & \gamma \\ \gamma & \beta \end{bmatrix}$ is positive definite, where $\alpha, \beta \in \{ 1 - \vartheta^2_{12}, 1-\vartheta^2_{13}, 1-\vartheta^2_{23} \}$ with $\alpha \ne \beta$, and $\gamma \in \{ \Delta_{12}, \Delta_{13}, \Delta_{23} \}$. Furthermore, either $(U_{1\bullet} z_\Ical )_+\ge (U_{2\bullet} z_\Ical )_+$ or $(U_{2\bullet} z_\Ical )_+\ge (U_{1\bullet} z_\Ical )_+$, where $U_{i \bullet}$ denotes the $i$th row of $U$.
Since $z_\Ical >0$, it follows from Lemma~\ref{lem:R+_S3_Step2} that either $\alpha > \gamma$ or $\beta > \gamma$. We first consider the case where $\alpha > \gamma$. In this case, $\alpha= 1- \vartheta^2_{j^*_1, j^*_2}$, $\beta = 1 - \vartheta^2_{j^*_1, j_3}$, and $\gamma=\Delta_{j^*_2, j_3}$. In light of the implications given by condition (iv), we have that
\begin{equation} \label{eqn:R+_step_3}
\det M \, > \, \max_{i \in \Scal^c} \big( \vartheta_{i, j_3} (1-\vartheta^2_{j^*_1, j^*_2}) - \vartheta_{i, j^*_1} \Delta_{j^*_1, j_3} - \vartheta_{i, j^*_2}\Delta_{j^*_2, j_3} \big)_+.
\end{equation}
Additionally, since $x^2$ is the unique solution to $\min_{w \ge 0, \supp(w) \subseteq \Jcal_2} \| y - A w \|^2$ with $\supp(x^2)=\Jcal_2$, it follows from Lemma~\ref{lem:R+_optimal_solution} that by letting $\wt\Ical := \Scal\setminus \Jcal_2=\{ j_3\}$,
\begin{align*}
 \max_{i \in \Scal} \big( A^T_{\bullet i} A(z- x ^2) \big)_+ & \,  = \, \big \| \big( (M/M_{\Jcal_2\Jcal_2}) \cdot z_{\wt\Ical} \big)_+ \big\|_\infty, \\
  \max_{i\in \Scal^c} \big( A^T_{\bullet i} A(z - x^2) \big)_+ & \, = \, \big\|  (A^T_{\bullet \Scal^c}[I- A^T_{\bullet \Jcal_2} ( A_{\bullet \Jcal_2} A_{\bullet \Jcal_2} )^{-1}  A^T_{\bullet \Jcal_2}] A_{\wt\Ical} \cdot z_{\wt \Ical}  )_+  \big\|_\infty.
\end{align*}
Note that $z_{\wt\Ical}$ and $M/M_{\Jcal_2\Jcal_2}$ are positive scalars. It follows from the Schur determinant formula that $M/M_{\Jcal_2\Jcal_2} = \det (M/M_{\Jcal_2\Jcal_2} )=\det M/\det(M_{\Jcal_2\Jcal_2})$.  Thus $\max_{i \in \Scal} \big( A^T_{\bullet i} A(z- x ^2) \big)_+ = \det M/\det(M_{\Jcal_2\Jcal_2}) \cdot z_{\wt \Ical}$. Further, direct calculations show that $( A_{\bullet \Jcal_2} A_{\bullet \Jcal_2} )^{-1}  A^T_{\bullet \Jcal_2} A_{\wt\Ical} = ( \Delta_{j^*_1, j_3}, \Delta_{j^*_2, j_3})^T/\det(M_{\Jcal_2 \Jcal_2})$.
In view of this result and $\det(M_{\Jcal_2\Jcal_2})= 1-\vartheta^2_{j^*_1, j^*_2}$, we have, for each $i \in \Scal^c$,
\[
\big( (A^T_{\bullet i}[I- A^T_{\bullet \Jcal_2} ( A_{\bullet \Jcal_2} A_{\bullet \Jcal_2} )^{-1}  A^T_{\bullet \Jcal_2}] A_{\wt\Ical} \cdot z_{\wt \Ical} \big)_+ =  \frac{z_{\wt \Ical}}{\det( M_{\Jcal_2\Jcal_2})} \Big( \vartheta_{i, j_3} (1-\vartheta^2_{j^*_1, j^*_2}) - \vartheta_{i, j^*_1} \Delta_{j^*_1, j_3} - \vartheta_{i, j^*_2}\Delta_{j^*_2, j_3} \Big)_+.
\]
These results and the inequality (\ref{eqn:R+_step_3}) imply that $\max_{i \in \Scal} \big( A^T_{\bullet i} A(z- x ^2) \big)_+ > \max_{i\in \Scal^c} \big( A^T_{\bullet i} A(z - x^2) \big)_+$. The other case where $\beta>\gamma$ can be established by the similar argument.
Therefore, following Algorithm~\ref{algo:constrained_MP},  $j^*_3:= \argmax_{i\in \Scal} \big( A^T_{\bullet i} A(z - x^2) \big)_+$ satisfies $j^*_3 = j_3$. This yields $\Jcal_3= \Scal$. Since $A_{\bullet \Scal}$ has full column rank, the exact vector recovery is achieved.

%\gap

 ``Only if''. Suppose every nonzero vector $x \in \mathbb R^N_+$ with $\supp(x)=\Scal$ is recovered from $y=A x$ via constrained matching pursuit for a given matrix $A \in \mathbb R^{m\times N}$ and the index set $\Scal=\{1, 2, 3\}$. It follows from Lemma~\ref{lem:nonnegative orthant_Nec01} that condition (i) must hold. Besides, by setting $x^0=0$, we see via Corollary~\ref{coro:Nec_Suf_suppt_recovery_RN_RN+}  that $\max_{i \in \Scal} (A^T_{\bullet i} A_{\bullet S} z_\Scal)_+ >  \max_{j \in \Scal^c} (A^T_{\bullet j} A_{\bullet S} z_\Scal)_+$ holds for all $z_\Scal \in \mathbb R^3_{++}$. This yields condition (ii).

 For each $p\in \Scal$, define the set $\mathcal W_p :=\{ z_\Scal \in \mathbb R^3_{++} \, | \,   (A^T_{\bullet p} A_{\bullet S} z_\Scal)_+ = \max_{i \in \Scal} (A^T_{\bullet i} A_{\bullet S} z_\Scal)_+ \}$. Clearly, $\mathbb R^3_{++} = \mathcal W_1 \cup \mathcal W_2 \cup \mathcal W_3$. Since the matrix $M:=A^T_{\bullet \Scal} A_{\bullet \Scal}$ given by (c) of Remark~\ref{remark:R+3_conditions}  is positive definite, we observe $|\vartheta_{ij}| <1$ for any $i \ne j$. Based on this observation, it is easy to show that for any given $(z_2, z_3)>0$, there exists a sufficiently large $z_1>0$ such that $(z_1, z_2, z_3) \in \mathcal W_1$.
  Hence, $\mathcal W_1$ is nonempty and $\{ (z_2, z_3) \, | \, z_\Scal=(z_1, z_2, z_3) \in \mathcal W_1 \}=\mathbb R^2_{++}$. By a similar argument, we deduce that $\mathcal W_2$ and $\mathcal W_3$ are nonempty and  $\{ (z_1, z_3) \, | \, z_\Scal=(z_1, z_2, z_3) \in \mathcal W_2 \}=\mathbb R^2_{++}$ and  $\{ (z_1, z_2) \, | \, z_\Scal=(z_1, z_2, z_3) \in \mathcal W_3 \}=\mathbb R^2_{++}$. Since $\mathbb R^3_{++} = \mathcal W_1 \cup \mathcal W_2 \cup \mathcal W_3$, $z_\Scal$ belongs to one of $\mathcal W_i$'s for any $z_\Scal \in \mathbb R^3_{++}$. For each $p \in \Scal$, it follows from Algorithm~\ref{algo:constrained_MP} that
 for any $z \in \mathcal W_p$,   the corresponding unique $x^1=  (A^T_{\bullet p} A_{\bullet S} z_\Scal)_+ \mathbf e_p$, where $(A^T_{\bullet p} A_{\bullet S} z_\Scal)_+>0$. Moreover,  we must have $\max_{i \in \Scal} \big( A^T_{\bullet i} A (z - x^1) \big)_+> \max_{j \in \Scal^c} \big( A^T_{\bullet j} A (z - x^1) \big)_+$. This condition, as shown at Step 2 of the ``if'' part, is equivalent to
 $
  \big\| (M/M_{\Jcal_1\Jcal_1} \, z_\Ical )_+ \big\|_\infty > \big\| (A^T_{\bullet \Scal^c}[I- A^T_{\bullet \Jcal_1} A_{\bullet \Jcal_1}] A_{\Ical} z_\Ical  )_+ \big\|_\infty,
  $
  where $\Jcal_1=\{ p \}$ and $\Ical = \Scal \setminus \Jcal_1$. Since $\{ z_\Ical \, | \, z_\Scal \in \mathcal W_p\}= \mathbb R^2_{++}$ as shown before, we obtain condition (iii).

To establish condition (iv), we first show the following claim: if $1-\vartheta^2_{12}>\min(\Delta_{13}, \Delta_{23})$ holds true, then there exists $z \in \mathbb R^N_+$ with $\supp(z)=\Scal$ such that when $y=A z$, Algorithm~\ref{algo:constrained_MP} give rises to $\Jcal_2=\{ 1, 2\}$.
 To prove this claim, it is noted that $1-\vartheta^2_{12}>\min(\Delta_{13}, \Delta_{23})$ is equivalent to $1-\vartheta^2_{12} > \Delta_{23}$ or $1-\vartheta^2_{12} > \Delta_{13}$. For the former case, i.e., $1-\vartheta^2_{12} > \Delta_{23}$, it follows from Lemma~\ref{lem:R+_S3_Step2} and $U^1=M/M_{11}$ given in (\ref{eqn:U_matrices}) that there exists $v:=(v_1, v_2)^T \in \mathbb R^2_{++}$ such that $(U^1_{1\bullet} v )_+ \ge (U^1_{2\bullet} v)_+$. Further, as shown previously, there exists a sufficiently large $v_0>0$ such that $\wt z=(\wt z_\Scal, 0)$ with $\wt z_\Scal:=(\wt z_1, \wt z_2, \wt z_3)=(v_0, v_1, v_2)$ satisfies $\wt z \in \mathcal W_1$. This implies via Lemma~\ref{lem:R+_optimal_solution} and the argument for Step 1 of the ``if'' part that when $y= A \wt z$, Algorithm~\ref{algo:constrained_MP} give rises to
 $(j^*_1, j^*_2)=(1, 2)$ and $\Jcal_2=\{ 1, 2\}$. The similar argument can be used to show that if  $1-\vartheta^2_{12} > \Delta_{13}$ holds, then there exists $z \in \mathbb R^N_+$ with $\supp(z)=\Scal$ such that when $y=A z$, Algorithm~\ref{algo:constrained_MP} give rises to $(j^*_1, j^*_2)=(2, 1)$ and $\Jcal_2= \{1, 2\}$. The above proof can be extended to show that  if $1-\vartheta^2_{13}>\min(\Delta_{12}, \Delta_{23})$ (respectively $1-\vartheta^2_{13}>\min(\Delta_{12}, \Delta_{23})$) holds, then there exists $z \in \mathbb R^N_+$ with $\supp(z)=\Scal$ such that when $y=A z$, Algorithm~\ref{algo:constrained_MP}  yields $\Jcal_2=\{ 1, 3 \}$ (respectively $\Jcal_2=\{ 2, 3 \}$).

As indicated in Remark~\ref{remark:R+3_conditions}, if the hypothesis of an implication in condition (iv) is false, then that implication holds true vacuously. Now consider an implication in condition (iv) whose hypothesis holds true. Then there exists $z \in \mathbb R^N_+$ with $\supp(z)=\Scal$ such that  Algorithm~\ref{algo:constrained_MP}  yields $\Jcal_2:=\{ j^*_1, j^*_2\}$ from $y=A z$.
%
%each $\Jcal_2:=\{ j^*_1, j^*_2\}$ which is achieved by some $z \in \mathbb R^N_+$ with $\supp(z)=\Scal$ via %Algorithm~\ref{algo:constrained_MP}.
%
Hence, the corresponding $x^2$ obtained from $y=Az$ via Algorithm~\ref{algo:constrained_MP}  satisfies $\supp(x^2)=\Jcal_2$. Since the exact support recovery implies that $\max_{i \in \Scal} (A^T_{\bullet i} A(z- x^2))_+>\max_{j \in \Scal^c} (A^T_{\bullet i} A(z- x^2))_+$, we deduce, in view of $\supp(x^2)=\Jcal_2$, Lemma~\ref{lem:R+_optimal_solution} and the argument for Step 3 of the ``if'' part, that
\[
 \frac{\det M}{\det(M_{\Jcal_2\Jcal_2}) } \cdot z_{\wt \Ical}  \, > \,  \frac{z_{\wt \Ical}}{\det( M_{\Jcal_2\Jcal_2})} \Big( \vartheta_{i, j_3} (1-\vartheta^2_{j^*_1, j^*_2}) - \vartheta_{i, j^*_1} \Delta_{j^*_1, j_3} - \vartheta_{i, j^*_2}\Delta_{j^*_2, j_3} \Big)_+,
\]
where $\wt \Ical=\{j_3\}=\Scal \setminus \Jcal_2$, $z_{\wt \Ical}\in \mathbb R_{++}$, and $\det(M_{\Jcal_2\Jcal_2}) =1-\vartheta^2_{j^*_1, j^*_2}$. This yields condition (iv).
\end{proof}

%---------------------------------------------------------------------------
%
\subsubsection{Sufficient Conditions for Exact Vector Recovery on $\mathbb R^N_+$ for a Fixed Support}

When a given support $\Scal$ is of size greater than or equal to 4, necessary {\em and} sufficient conditions are difficult to obtain due to increasing complexities.
%
%For example,
%
%for $z-x^k$ may not be nonnegative when step $k \ge 3$; and (ii)
%
%at Step $2$, the projection of the set that is similar to $\mathcal W_p$ in the ``only if'' part of %Theorem~\ref{thm:nonnegative constraint_S3} is no longer $\mathbb R^{|\Scal|-2}_{++}$  in general but is the %(nonconvex) union of finitely many polyhedral cones.
%
Hence, we seek neat sufficient conditions in this subsection.

\begin{theorem} \label{thm:R+_sufficient_cond}
  Given a matrix $A \in \mathbb R^{m\times N}$ with unit columns and the index set $\Scal \subset \{1, \ldots, N\}$, let $M:=A^T_{\bullet \Scal} A_{\bullet \Scal}$. Then every nonzero vector $z \in \mathbb R^N_+$ with $\supp(z)=\Scal$ is recovered from $y=A z$ via constrained matching pursuit if the following conditions hold:
 \begin{itemize}
   \item [(i)] $A_{\bullet \Scal}$ has full column rank or equivalently $M$ is positive definite; and
    \item [(ii)] For any (possibly empty) index set $\Jcal \subset \Scal$,
     \begin{equation} \label{eqn:R+_suff_cond}
       \big\| (M/M_{\Jcal\Jcal} \, x )_+ \big\|_\infty \, > \, \big\| \big(A^T_{\bullet \Scal^c}[I- A^T_{\bullet \Jcal} ( A^T_{\bullet \Jcal} A_{\bullet \Jcal} )^{-1} A_{\bullet \Jcal}] A_{\bullet\Scal\setminus\Jcal} \, x  \big)_+ \big\|_\infty, \quad \ \forall \ x \in \mathbb R^{|\Scal\setminus\Jcal|}_{++},
     \end{equation}
    where $M/M_{\Jcal\Jcal}$ is the Schur complement of $M_{\Jcal\Jcal}$ in $M$.
 \end{itemize}
\end{theorem}

\begin{proof}
 Due to condition (i), it suffices to show the exact support recovery of each $z\in \mathbb R^N_+$ with $\supp(z)=\Scal$ via Algorithm~\ref{algo:constrained_MP} from $y=A z$. Toward this end,
%
%  we
%show as follows that under condition (ii), condition $(\mathbf H)$ given by (\ref{eqn:condition_H'}) holds %for any $u \in \mathbb R^N_+$ with $\supp(u)=\Scal$, any index set $\Jcal\subset \supp(u)$, and the (unique) %optimal solution $v=\argmin_{w \ge 0, \supp(w)\subseteq \Jcal} \| A (u - w )\|^2_2$.
%
we see via a similar argument for Corollary~\ref{coro:Nec_Suf_suppt_recovery_RN_RN+} that condition $(\mathbf H)$ given by (\ref{eqn:condition_H'}) holds if for any $0\ne u \in \mathbb R^N_+$ with $\supp(u)=\Scal$, any  index set $\Jcal \subset \Scal$, and the (unique) optimal solution $v=\argmin_{w \ge 0, \supp(w)\subseteq \Jcal} \| A (u - w )\|^2_2$, the following holds:
\[
   \max_{i\in \Scal} \big( A^T_{\bullet i} A(u - v) \big)_+ = \max_{i\in \Scal\setminus \Jcal} \big( A^T_{\bullet i} A(u - v) \big)_+ \, > \, \max_{j \in \Scal^c} \big( A^T_{\bullet j} A(u - v) \big)_+,
\]
where the first equation follows from Lemma~\ref{lem:exact_suppt_recovery_index}.
Let $\Jcal^*:=\supp(v)$. Hence, $\Jcal^* \subseteq \Jcal \subset \Scal$. Since $v$ is the optimal solution to $\min_{w \ge 0, \supp(w)\subseteq \Jcal} \| A (u - w )\|^2_2$, we deduce via Lemma~\ref{lem:R+_optimal_solution} that
\begin{align*}
  \max_{i\in \Scal} \big( A^T_{\bullet i} A(u - v) \big)_+ & \, = \ \big \| \big( A^T_{\bullet  \Scal} A(u - v) \big)_+ \big\|_\infty \, = \, \big\| \big( (M/M_{\Jcal^* \Jcal^*}) \cdot u_\Ical \big)_+ \big\|_\infty, \\
  \max_{j\in \Scal^c} \big( A^T_{\bullet j} A(u - v) \big)_+ & \, = \ \big \| \big( A^T_{\bullet  \Scal^c} A(u - v) \big)_+ \big\|_\infty \, = \, \big \| \big(A^T_{\bullet \Scal^c}[I- A^T_{\bullet \Jcal^*} ( A^T_{\bullet \Jcal^*} A_{\bullet \Jcal^*} )^{-1} A_{\bullet \Jcal^*}] A_{\bullet\Ical} \cdot u_\Ical  \big)_+ \big\|_\infty,
\end{align*}
where $\Ical := \Scal \setminus \Jcal^*$ is nonempty. Since $u_\Ical >0$, we see that $\max_{i\in \Scal} \big( A^T_{\bullet i} A(u - v) \big)_+ \, > \, \max_{j \in \Scal^c} \big( A^T_{\bullet j} A(u - v) \big)_+$ holds under condition (ii). This leads to the desired result. %%%({\bf may be a little more ...})
\end{proof}

In what follows, we develop conditions to verify the inequality given in (\ref{eqn:R+_suff_cond}), which leads to a numerical scheme to check (\ref{eqn:R+_suff_cond}). Fix an index set $\Jcal \subset \Scal$, and let $r:=|\Scal\setminus \Jcal|$. Further, let $M/M_{\Jcal\Jcal}=[ p_1, \cdots, p_r]$, and $E:=\big(A^T_{\bullet \Scal^c}[I- A^T_{\bullet \Jcal} ( A^T_{\bullet \Jcal} A_{\bullet \Jcal} )^{-1} A_{\bullet \Jcal}] A_{\bullet\Scal\setminus\Jcal} \big)^T=[ q_1, \cdots, q_{|\Scal^c|} ]$, namely, $p_i \in \mathbb R^r$ is the $i$th column of $M/M_{\Jcal\Jcal}$ and $q_j \in \mathbb R^r$ is the $j$th column of $E$.

\begin{lemma} \label{lem:verification_cond}
 The inequality (\ref{eqn:R+_suff_cond}) for a fixed index set $\Jcal \subset \Scal$ holds if and only if for each $q_j\in \mathbb R^r$, there exist $w\in \mathbb R^r_+$ and $0 \ne (w', \beta) \in \mathbb R^r_+ \times \mathbb R_+$ such that $[ p_1 - q_j, p_2-q_j, \cdots, p_r - q_j] w = w'+ \beta \cdot q_j$.
\end{lemma}

\begin{proof}
Since the Schur complement $M/M_{\Jcal\Jcal}$ is symmetric, it is easy to see that
 the inequality (\ref{eqn:R+_suff_cond}) fails if and only if there exists $v>0$ such that $\max_{i=1, \ldots, r} ( p^T_i v)_+ \le (q^T_j v)_+$ for some $j$. In view of Lemma~\ref{lem:positive_sign}, we deduce that $\max_{i=1, \ldots, r}  ( p^T_i v)_+>0$ such that $q^T_j v >0$ for this $j$.  Hence, the inequality system $\max_{i=1, \ldots, r}  ( p^T_i v)_+ \le (q^T_j v)_+, v >0$ is equivalent to the following linear inequality system:
 \[
   \mbox{(I)}: \quad   v>0, \ \ q^T_j v >0, \ \ q^T_j v \ge p^T_i v, \ \ \forall \ i=1, \ldots, r.
 \]
 By  Motzkin's Transposition Theorem, (I) has no solution if and only if there exist $w\in \mathbb R^r_+$ and $0 \ne (w', \beta) \in \mathbb R^r_+ \times \mathbb R_+$ such that $[ p_1 - q_j, p_2-q_j, \cdots, p_r - q_j] w = w'+ \beta \cdot q_j$, yielding the desired result.
%
%  the inequality (\ref{eqn:R+_suff_cond}).
\end{proof}

The  condition derived in the above lemma can be effectively verified via a linear program for the given matrices $M/M_{\Jcal \Jcal}$ and $E$.

\mycut{
Moreover, it is equivalent to  the existence of  $q_j \in \mathbb R^r$, $w\in \mathbb R^r_+$ and $0 \ne (w', \beta) \in \mathbb R^r_+ \times \mathbb R_+$ such that $[ p_1, p_2, \cdots, p_r ] w = w'+ (\beta + \mathbf 1^T w) \cdot q_j$. We consider three possible cases as follows for a fixed $q_j$: (i) $q_j \le 0$. In this case, we can choose $w=0$, $\beta=1$, and $w'=(q_j)_- \ge 0$. Hence, this condition holds. (ii) $q_j \nleq 0$. We first show that if desired $w\in \mathbb R^r_+$ and $0 \ne (w', \beta) \in \mathbb R^r_+ \times \mathbb R_+$ exist, then $w$ must be nonzero. In fact, suppose $w=0$. Then we have $w'+ \beta q_j =0$. Since $p_j$ has a positive element and $0 \ne (w', \beta) \in \mathbb R^r_+ \times \mathbb R_+$, $\beta$ cannot be positive and thus must be zero. This shows that $w' =0$. But this contradicts $(w', \beta) \ne 0$. Therefore, $w \ne 0$. This thus implies that if desired $w$ and $0 \ne (w', \beta) \in \mathbb R^r_+ \times \mathbb R_+$ exist, then $\mathbf 1^T w>0$ such that $\beta+\mathbf 1^T w>0$ and $[ p_1, p_2, \cdots, p_r ] \frac{w}{\beta+\mathbf 1^T w} = \wt w + q_j$, where $0 \ne (\wt w, \beta) \in \mathbb R^r_+ \times \mathbb R_+$, and $\wt w=w'/(\beta+\mathbf 1^T w)$. Two subcases: (a) $\beta=0$ and $\wt w \ngeq 0$; and (b) $\beta>0$ and $\wt w \ge 0$. The former is equivalent to the existence of $0 \ne u \in \mathbb R^r_+$ such that $(q_j + u) \in \mbox{conv}(p_1, \ldots, p_r)$, and the latter is equivalent to the existence of $ u \in \mathbb R^r_+$ such that $(q_j + u) \in \mbox{conv}(0, p_1, \ldots, p_r)$ and the latter is equivalent to ... (noting that $0\notin \mbox{conv}(p_1, \ldots, p_r)$ since
$\{ p_1, p_2, \cdots, p_r \}$ are linearly independent).
}

%--------------------------------------------------------------------------
%
\subsection{Exact Vector Recovery on $\mathbb R^{N_1} \times \mathbb R^{N_2}_+$ for a Fixed Support}

In this subsection, we briefly discuss an extension of the preceding exact vector recovery results
to a Cartesian product of copies of $\mathbb R$ and $\mathbb R_+$. Let $\Ical_1$ and $\Ical_+$ be two nonempty index subsets that form a disjoint union of $\{1, \ldots, N\}$. Consider the constraint set $\Pcal = \mathbb R_{\Ical_1} \times (\mathbb R_+)_{\Ical_+}$.
%
%$\Pcal:=\{ x \in \mathbb R^N \, | \, x_i \in \mathbb R_+, \forall \, i \in \Ical_+\}$. For notational %simplicity, we write $\Pcal$ as $\mathbb R_{\Ical_1} \times (\mathbb R_+)_{\Ical_+}$.
%
The following preliminary result can be easily extended from Corollary~\ref{coro:Nec_Suf_suppt_recovery_RN_RN+} and Lemma~\ref{lem:nonnegative orthant_Nec01}; its proof is thus omitted.

\begin{lemma} \label{lem:preliminary_results_RR+}
 Let $A\in \mathbb R^{m\times N}$ be a matrix with unit columns, and $\Pcal=\mathbb R_{\Ical_1} \times (\mathbb R_+)_{\Ical_+}$. The following hold:
 \begin{itemize}
   \item [(i)] Let $0 \ne z \in \Sigma_K \cap \Pcal$ with $|\supp(z)|=r$. Then  the exact support recovery of $z$ is achieved if and only if for any sequence $\big( (x^k, j^*_k, \Jcal_k) \big)_{k \in \mathbb N}$ generated by Algorithm~\ref{algo:constrained_MP} with $y=A z$,
     \begin{align*}
    & \max\Big( \, \max_{j \in (\supp(z)\setminus \Jcal_k) \cap \Ical_1} |A^T_{\bullet j} A(z - x^k)|, \  \max_{j \in (\supp(z)\setminus \Jcal_k) \cap \Ical_+} [A^T_{\bullet j} A(z - x^k) ]_+  \, \Big) \\
    & \, > \ \max\Big( \, \max_{j\in [\supp(z)]^c\cap \Ical_1} | A^T_{\bullet j} A(z - x^k) |, \ \max_{j\in [\supp(z)]^c\cap \Ical_+} [ A^T_{\bullet j} A(z - x^k) ]_+   \Big), \quad \forall \ k=0, 1, \ldots, r-1.
     \end{align*}
   \item [(ii)]  Let $\Scal$ be a nonempty index subset of $\{1, \ldots, N\}$. The exact vector recovery of every vector $x \in \Pcal$ with $\supp(x)=\Scal$ is achieved via constrained matching pursuit  only if $A_{\bullet \Scal}$ has full column rank.
 \end{itemize}
\end{lemma}

The next result characterizes the exact vector recovery on $\Pcal$ for a given support $\Scal$ of size 2.

\begin{theorem}
  Given a matrix $A \in \mathbb R^{m\times N}$ with unit columns and the index set $\Scal=\{1, 2\}$ with $1 \in \Ical_1$ and $2\in \Ical_+$, every vector $x \in \Pcal=\mathbb R_{\Ical_1} \times (\mathbb R_+)_{\Ical_+}$ with $\supp(x)=\Scal$ is recovered from $y=A x$ via constrained matching pursuit if and only if the following conditions hold:
 \begin{itemize}
   \item [(i)] $A_{\bullet \Scal}$ has full column rank or equivalently $|\vartheta_{12}|<1$;
    \item [(ii)]
   $\displaystyle \max \big( |z_1 + \vartheta_{12} z_2|, \, (\vartheta_{12} z_1 + z_2)_+  \big) > \max\Big( \max_{j \in \Scal^c\cap \Ical_1} |\vartheta_{j1} z_1 + \vartheta_{j2} z_2 \big|, \, \max_{j \in \Scal^c\cap \Ical_+} \big(\vartheta_{j1} z_1 + \vartheta_{j2} z_2 \big)_+ \Big)$, \\
   $\forall \, (z_1, z_2)^T \in \big(\mathbb R \setminus \{ 0 \} \big) \times \mathbb R_{++}$;
    \item [(iii)] $1-\vartheta^2_{12} \, > \,
     \max\Big( \max_{j\in \Scal^c\cap \Ical_1} \, |\vartheta_{j2}- \vartheta_{12}\vartheta_{j1}|, \ \max_{j \in \Scal^c \cap \Ical_+}(\vartheta_{j2}- \vartheta_{12}\vartheta_{j1})_+, \ \max_{j\in \Scal^c} | \vartheta_{j1}- \vartheta_{12}\vartheta_{j2} | \,  \Big)$.
 %%%\end{equation}
 \end{itemize}
\end{theorem}

\begin{proof}
 ``Only if''. Suppose the exact vector recovery is achieved for any $x \in \Pcal$ with $\supp(x)=\Scal$. Condition (i) follows from statement (ii) of Lemma~\ref{lem:preliminary_results_RR+}, and condition (ii) follows from Step 1 of Algorithm~\ref{algo:constrained_MP} and statement (i) of Lemma~\ref{lem:preliminary_results_RR+} with $x^0=0$ and $\Jcal_0=\emptyset$. To establish condition (iii), we first notice via $|\vartheta_{12}|<1$ that for any $z\in \Pcal$ with $\supp(z)=\Scal$, i.e., $z_1 \ne 0$ and $z_2>0$, $|z_1 + \vartheta_{12} z_2 | \ge (\vartheta_{12} z_1 + z_2)_+$ if and only if $|z_1| \ge z_2 >0$, and $|z_1 + \vartheta_{12} z_2 | \le (\vartheta_{12} z_1 + z_2)_+$ if and only if $z_2 \ge |z_1|>0$. When the former holds, i.e., $|z_1|\ge z_2>0$, we have $j^*_1=1$ and $x^1=(z_1+ \vartheta_{12} z_2) \cdot \mathbf e_1$. Hence, $A^T_{\bullet j} A(z - x^1)= (\vartheta_{j2} - \vartheta_{j1} \vartheta_{12}) z_2$. Using Step 2 of Algorithm~\ref{algo:constrained_MP} and statement (i) of Lemma~\ref{lem:preliminary_results_RR+} with $\Jcal_1=\{1 \}$, it is easy to obtain $1-\vartheta^2_{12} > \max\big( \max_{j\in \Scal^c\cap \Ical_1} \, |\vartheta_{j2}- \vartheta_{12}\vartheta_{j1}|, \ \max_{j \in \Scal^c \cap \Ical_+}(\vartheta_{j2}- \vartheta_{12}\vartheta_{j1})_+\big)$. We next consider the case where $z_2 \ge |z_1|>0$. In this case, $j^*_1=2$ such that $x^1=(\vartheta_{12} z_1 + z_2)_+ \cdot \mathbf e_2$, where $\vartheta_{12} z_1 + z_2 > 0$. Hence, $A^T_{\bullet j} A(z - x^1)= (\vartheta_{j1} - \vartheta_{j2} \vartheta_{12}) z_1$. Applying Step 2 of Algorithm~\ref{algo:constrained_MP} and statement (i) of Lemma~\ref{lem:preliminary_results_RR+} with $\Jcal_1=\{2 \}$, we have that
 \[
  (1-\vartheta^2_{12}) |z_1| > \max\big( \max_{j\in \Scal^c\cap \Ical_1} \, |(\vartheta_{j1}- \vartheta_{12}\vartheta_{j2}) z_1|, \ \max_{j \in \Scal^c \cap \Ical_+}[(\vartheta_{j1}- \vartheta_{12}\vartheta_{j2}) z_1 ]_+\big).
  \]
 It is easy to show that  $(1-\vartheta^2_{12}) |z_1| > \max_{j \in \Scal^c \cap \Ical_+}[(\vartheta_{j1}- \vartheta_{12}\vartheta_{j2}) z_1 ]_+$ for any $z_1 \ne 0$ if and only if $1-\vartheta^2_{12} > \max_{j \in \Scal^c \cap \Ical_+}|\vartheta_{j1}- \vartheta_{12}\vartheta_{j2}|$. This yields $1-\vartheta^2_{12}> \max_{j\in \Scal^c} | \vartheta_{j1}- \vartheta_{12}\vartheta_{j2} |$, and  condition (iii).

 ``If''. This part can be shown in a similar way by reversing the previous argument.
\end{proof}

Necessary and sufficient conditions for  the exact vector recovery on $\Pcal$ for a given support $\Scal$ of size 3 can be established via a similar argument for Theorem~\ref{thm:nonnegative constraint_S3}. Instead doing this, we provide a sufficient condition for a given support of arbitrary size. To simplify notation, we define the following function $F_{\Ical, \Jcal}: \mathbb R_\Ical \times (\mathbb R_+)_\Jcal \rightarrow \mathbb R$ for given index sets $\Ical$ and $\Jcal$: $F_{\Ical, \Jcal}( v ):=\max\big( \max_{i\in \Ical} |v_i|, \max_{i\in \Jcal} (v_i)_+\big)$.

\begin{theorem} \label{thm:RR+_sufficient_cond}
  Given a matrix $A \in \mathbb R^{m\times N}$ with unit columns and the index set $\Scal \subset \{1, \ldots, N\}$, let $M:=A^T_{\bullet \Scal} A_{\bullet \Scal}$, $\Scal_1:=\Scal \cap \Ical_1$, and $\Scal_+:=\Scal \cap \Ical_+$. Then every vector $z \in \Pcal$ with $\supp(z)=\Scal$ is recovered from $y=A z$ via constrained matching pursuit if the following conditions hold:
 \begin{itemize}
   \item [(i)] $A_{\bullet \Scal}$ has full column rank or equivalently $M$ is positive definite; and
    \item [(ii)] For any (possibly empty) index sets $\Lcal_1 \subset \Scal_1$ and $\Lcal_+ \subset \Scal_+$, letting $\wt \Lcal:=\Lcal_1 \cup \Lcal_+$,
     \begin{align*} %\label{eqn:RR+_suff_cond}
      & F_{\Scal_1\setminus \Lcal_1, \, \Scal_+\setminus \Lcal_+}\left( M/M_{\wt \Lcal \wt\Lcal} \begin{pmatrix} v_{\Scal_1 \setminus \Lcal_1} \\  v_{\Scal_+ \setminus \Lcal_+} \end{pmatrix} \right) \notag \\
      & > \
       F_{\Scal^c\cap \Ical_1, \, \Scal^c\cap \Ical_+}\left( \big(A^T_{\bullet \Scal^c}[I- A^T_{\bullet \wt\Lcal} ( A^T_{\bullet \wt\Lcal} A_{\bullet \wt\Lcal} )^{-1} A_{\bullet \wt\Lcal}] A_{\bullet\Scal\setminus\wt\Lcal}
        \begin{pmatrix} v_{\Scal_1 \setminus \Lcal_1} \\  v_{\Scal_+ \setminus \Lcal_+} \end{pmatrix} \right)
     \end{align*}
     for all $v_{\Scal_+ \setminus \Lcal_+} >0$ and  all $v_{\Scal_1 \setminus \Lcal_1}$ whose each element is nonzero.
   %
   % where $M/M_{\Jcal\Jcal}$ is the Schur complement of $M_{\Jcal\Jcal}$ in $M$.
 \end{itemize}
\end{theorem}

\begin{proof}
 Let $\Jcal \subset \Scal$ be a nonempty index set. Since $\Pcal=\mathbb R_{\Ical_1} \times (\mathbb R_+)_{\Ical_+}$ is a closed convex cone, it follows from the discussions at the end of Section~\ref{sect:Constrained_MP} that
 the necessary and sufficient optimality condition for an optimal solution $x^*=(x^*_\Jcal, 0)$ of the underlying minimization problem $\min_{w \in \Pcal, \supp(w) \subseteq \Jcal} \| A w - A z \|^2_2$ is given by:
 $\mathcal C \in x^*_\Jcal \perp   A^T_{\bullet \Jcal}( A_{\bullet \Jcal} x^*_\Jcal- A z ) \in \mathcal C^*$, where $z \in \Pcal$ is such that $\supp(z)=\Scal$, the convex cone $\mathcal C:=\{ w_\Jcal \, | \,  (w_{\Jcal}, 0) \in \Pcal \}= \mathbb R_{\Ical_1 \cap \Jcal} \times (\mathbb R_+)_{\Ical_+\cap \Jcal}$ and the dual cone $\mathcal C^*$ is given by $\mathcal C^*=\{ 0 \}\times (\mathbb R_+)_{\Ical_+\cap \Jcal}$. Hence, we have that
 \[
   A^T_{\bullet \Ical_1 \cap \Jcal} ( A_{\bullet \Jcal} x^*_\Jcal- A z ) = A^T_{\bullet \Ical_1 \cap \Jcal} A(x^*-z)=0,
 \]
 where $(\Ical_1 \cap \Jcal) \subset \Scal_1$, and
 \[
    0 \le x^*_{\Ical_+ \cap \Jcal} \perp A^T_{\bullet \Ical_+ \cap \Jcal} ( A_{\bullet \Jcal} x^*_\Jcal- A z ) \ge 0,
 \]
 where $x^*_\Jcal=(x^*_{\Ical_1 \cap \Jcal}, x^*_{\Ical_+ \cap \Jcal})$ with $x^*_{\Ical_+ \cap \Jcal} \ge 0$. %
 Let the index set $\mathcal L_+:=\{ i \in \Ical_+ \cap \Jcal \, | \, x^*_i>0\}$. Thus $\mathcal L_+ \subset \Scal_+$ and $A^T_{\bullet \mathcal L_+} A(x^*-z)=0$. Set $\Lcal_1:=\Ical_1 \cap \Jcal$, and $\wt \Lcal:=\Lcal_1 \cup \Lcal_+$. Hence, $\Lcal_1$ and $\Lcal_+$ are disjoint subsets of $\Scal$ with $A^T_{\bullet \wt \Lcal} A(z - x^*)=0$. Further, $x^*_{\Scal\setminus \wt \Lcal}=0$.
 %
 %
 %Let the index set $\mathcal K:=\{ i \in \Ical_+ \cap \Jcal \, | \, x^*_i>0\}$. Thus $\mathcal K \subset %\Scal_+$ and $A^T_{\bullet \mathcal K} A(x^*-z)=0$. Set $\Lcal_1:=\Ical_1 \cap \Jcal$, $\Lcal_+:=\mathcal %K$, and $\wt \Lcal:=\Lcal_1 \cup \Lcal_+$. Hence, $\Lcal_1$ and $\Lcal_+$ are disjoint subsets of $\Scal$ %with $A^T_{\bullet \wt \Lcal} A(z - x^*)=0$. Further, $x^*_{\Scal\setminus \wt \Lcal}=0$.
 %
%
 Hence, $ A^T_{\bullet \Scal \setminus \wt \Lcal} A (z - x^*) = M/M_{\wt\Lcal \wt \Lcal} (z - x^*)_{\Scal\setminus \wt \Lcal} = M/M_{\wt\Lcal \wt \Lcal} \, z_{\Scal \setminus \wt \Lcal}$, and $A^T_{\bullet \Scal^c} A (z - x^*) = A^T_{\bullet \Scal^c}[I- A^T_{\bullet \wt\Lcal} ( A^T_{\bullet \wt\Lcal} A_{\bullet \wt\Lcal} )^{-1} A_{\bullet \wt\Lcal}] A_{\bullet\Scal\setminus\wt\Lcal} \, z_{\Scal \setminus \wt \Lcal}$. Since $\Scal$ is a disjoint union of $\Scal_1$ and $\Scal_+$, $z_{\Scal \setminus \wt \Lcal} = (z_{\Scal_1\setminus \Lcal_1}, z_{\Scal_+\setminus \Lcal_+})$, where $z_{\Scal_+\setminus \Lcal_+}>0$ and each element of $z_{\Scal_1\setminus \Lcal_1}$ is nonzero. Further,
 \[
   \max\Big( \max_{j \in \Scal_1\setminus \Jcal} |A^T_{\bullet j} A (z -x^*)|, \max_{j\in \Scal_+ \setminus \Jcal} [ A^T_{\bullet j} A (z -x^*) ]_+\Big) \, = \, F_{\Scal_1\setminus \Lcal_1, \, \Scal_+\setminus \Lcal_+}\left( M/M_{\wt \Lcal \wt\Lcal} \, z_{\Scal \setminus \wt \Lcal} \right),
 \]
 and
 \begin{align*}
   &\max\Big( \max_{j \in \Scal^c\cap \Ical_1} |A^T_{\bullet j} A (z -x^*)|, \max_{j\in \Scal^c \cap \Ical_+} [ A^T_{\bullet j} A (z -x^*) ]_+\Big) \\
   & \, = \ F_{\Scal^c\cap \Ical_1, \, \Scal^c\cap \Ical_+}\left( \big(A^T_{\bullet \Scal^c}[I- A^T_{\bullet \wt\Lcal} ( A^T_{\bullet \wt\Lcal} A_{\bullet \wt\Lcal} )^{-1} A_{\bullet \wt\Lcal}] A_{\bullet\Scal\setminus\wt\Lcal}\, z_{\Scal \setminus \wt \Lcal} \right).
 \end{align*}
 Consequently, under the condition (ii), condition $(\mathbf H)$ holds, leading to the exact vector recovery.
\end{proof}

%

%--------------------------------------------
%
\section{\tblue{Sufficient Conditions for Uniform Exact Recovery on Convex, CP Admissible Sets via Constrained Matching Pursuit}} \label{sect:suff_cond_exact_recovery}

%
%Through this section, we assume that $\Pcal$ is a closed, convex and CP admissible set unless otherwise %specified, and
%
In this section, we derive sufficient conditions for uniform exact support and vector recovery via constrained matching pursuit using the restricted isometry-like and restricted orthogonality-like constants. For this purpose, we introduce the following constants.

\begin{definition} \label{def:RIP_MC_constants}
For a given (possibly non-CP admissible) set $\Pcal$, a matrix $A \in \mathbb R^{m\times N}$, and disjoin index sets $\Scal_1, \Scal_+, \Scal_-$ whose union is $\{1, \ldots, N\}$,  we say that
\begin{itemize}
 \item [(i)] A real number $\delta$ is {\em of Property RI on $\Pcal$} if $0<\delta <1$ and $(1-\delta) \cdot \| u - v \|^2_2 \le \| A(u-v) \|^2_2 $ for all $u, v\in \Sigma_K \cap \Pcal$ with $\supp(v) \subset \supp(u)$, where \tblue{$\Sigma_K:=\{ x \in \mathbb R^N \, | \, |\supp(x)|\le K\}$};
 \item [(ii)] A real number $\theta$  is {\em of Property RO  on $\Pcal$ corresponding to $\Scal_1, \Scal_+, \Scal_-$} if $\theta>0$ and for all $u, v\in \Sigma_K \cap \Pcal$ with $\supp(v) \subset \supp(u)$, the following holds:
     %\[
     %   \begin{array}{llll}
     %     |\langle A(u-v), A_{\bullet j} \rangle| \ \le \ \theta \cdot \|u-v\|_2, & \quad \forall \ j\in %[\supp(u)]^c \cap \Scal_1, \\
     %     \langle A(u-v), A_{\bullet j} \rangle_+ \ \le \ \theta \cdot \|u-v\|_2, & \quad \forall \ j\in %[\supp(u)]^c\cap \Scal_+, \\
     %     \langle A(u-v), A_{\bullet j} \rangle_- \ \le \ \theta \cdot \|u-v\|_2, & \quad  \forall \ j\in %[\supp(u)]^c\cap \Scal_-.
     %     \end{array}
     %\]
     %or equivalently
     \begin{eqnarray*}
    \lefteqn{
     \max\Big( \max_{j \in [\supp(u)]^c \cap \Scal_1} |\langle A(u-v), A_{\bullet j} \rangle|, \
      \max_{j \in  [\supp(u)]^c \cap \Scal_+} \langle A(u-v), A_{\bullet j} \rangle_+,
      } \qquad \qquad \qquad \qquad \\
      &    \max_{j \in [\supp(u)]^c \cap \Scal_-} \langle A(u-v), A_{\bullet j} \rangle_-    \Big ) \ \le \ \theta \cdot \|u-v\|_2.
     \end{eqnarray*}
 %
 %    $\langle A(u-v), A_{\bullet j} \rangle_+ \le \bar\theta_{k, \Pcal} \cdot \|u-v\|_2$  for $j\in %[\supp(u)]^c\cap \Scal_+$, $\langle A(u-v), A_{\bullet j} \rangle_- \le \bar\theta_{k, \Pcal} \cdot %\|u-v\|_2$  for $j\in [\supp(u)]^c\cap \Scal_-$,
 %and $|\langle A(u-v), A_{\bullet j} \rangle| \le \bar\theta_{k, \Pcal} \cdot \|u-v\|_2$  for $j\in %[\supp(u)]^c \cap \Scal_1$.
 \end{itemize}
We also denote these two constants by $\delta_{K, \Pcal}$ and $\theta_{K, \Pcal}$ respectively to emphasize their dependence on $\Pcal$.
\end{definition}
When $\Pcal=\mathbb R^N$, the constant $\delta_{K, \Pcal}$ resembles the restricted isometry constant, and the constant $\theta_{K, \Pcal}$ is closely related to the $(K, 1)$-restricted orthogonality constant \cite[Definition 6.4]{FoucartRauhut_book2013}.

We consider an irreducible, closed convex and CP admissible set $\Pcal$ (cf. Definition~\ref{def:irreducible_CP_set}) as follows.
By Propositions~\ref{prop:CP_admissible_cone} and \ref{prop:conic_hull},
  its conic hull $\mbox{cone}(\Pcal) = \mathbb R_{\Lcal_1} \times (\mathbb R_+)_{\Lcal_+} \times (\mathbb R_-)_{\Lcal_-}$, where $\Lcal_1, \Lcal_+$ and $\Lcal_-$ form a disjoint union of $\{1, \ldots, N\}$.
For a given $v\in\Pcal$, recall that the interval $\mathbb I_j(v)=[a_j(v), b_j(v)]$, where $a_j(v)\in \mathbb R_-\cup\{-\infty\}$ and $b_j(v) \in \mathbb R_+\cup \{+\infty\}$ take the extended real values; see Section~\ref{sect:Constrained_MP}.
We introduce more notation. For a given index set $\Jcal$ and $u, v \in \Pcal$ with $\supp(v) \subseteq \Jcal \subset \supp(u)$, we define the following  (possibly empty) index sets, where $\wt t_j(u, v)$ is defined in (\ref{eqn:def_tilda_t}):
\begin{eqnarray}
 \Lcal^-_a(u, v) &  := & \{ j \in \supp(u)\setminus \Jcal \ | \ \, \wt t_j(u, v)< a_j(v)<0 \, \},  \notag \\
 \Lcal^0_a(u, v) &  := & \{ j \in \supp(u)\setminus \Jcal \ | \ \, \wt t_j(u, v)\le b_j(v), \mbox{ and } a_j(v) = 0 < b_j(v) \, \},  \notag \\
  \Lcal^+_b(u, v) & := & \{ j \in \supp(u)\setminus \Jcal \ | \ \, \wt t_j(u, v) > b_j(v)>0  \, \},  \label{eqn:LL_index_sets} \\
  \Lcal^0_b(u, v) &  := & \{ j \in \supp(u)\setminus \Jcal \ | \ \, \wt t_j(u, v)\ge a_j(v), \mbox{ and } a_j(v) < 0 = b_j(v)  \, \},  \notag \\
  \Lcal_{uc}(u, v) & := & \{ j \in \supp(u)\setminus \Jcal \ | \ \, \wt t_j(u, v) \in [ a_j(v), b_j(v)], \mbox{ and } a_j(v) < 0 < b_j(v) \, \},  \notag \\
  \Lcal_{0}(u, v) & := & \{ j \in \supp(u)\setminus \Jcal \ | \ \, a_j(v) = 0 = b_j(v) \, \}. \notag
\end{eqnarray}
It is easy to see that the above index sets form a disjoint union of $\supp(u) \setminus \Jcal$.
Further, when $A$ has unit columns,
 it follows from (\ref{eqn:def_tilda_t}) that $\wt t_j(u, v)=\langle A(u-v), A_{\bullet j} \rangle $ for any index $j$ and $u, v \in \Pcal$.
%
%%%We recall some notation.
Using the index sets defined in (\ref{eqn:LL_index_sets}), we present two technical results as follows.

%
%Without loss of generality, we assume that $A$ has unit columns i.e., $\| A_{\bullet i}\|_2=1$ for all $i$.
%

\begin{lemma} \label{lem:W_a_b_bound}
For a given index set $\Jcal$ and $u, v \in \Pcal$ with $\supp(v) \subseteq \Jcal \subset \supp(u)$, the following hold:
\begin{itemize}
 \item [(i)] For any $j\in \Lcal^+_b(u, v)$, $ 2 b_j(v) \wt t_j(u, v) - b^2_j(v) \ge [ b_j(v)/\wt t_j(u,v)] \cdot \big[ \wt t_j(u, v) \big]^2$;
 \item [(ii)] For any $j\in \Lcal^-_a(u, v)$, $ 2 a_j(v) \wt t_j(u, v) - a^2_j(v) \ge  [ a_j(v)/\wt t_j(u,v)] \cdot \big[ \wt t_j(u, v) \big]^2$.
 \end{itemize}
\end{lemma}

\begin{proof}
For any $j\in \Lcal^+_b(u, v)$, we have $\wt t_j(u, v) > b_j(v) > 0$. Therefore,
\begin{eqnarray*}
  2 b_j(v) \wt t_j(u, v) - b^2_j(v) & \ge & b_j(v) [ 2 \wt t_j(u, v) - b_j(v)]  \, = \, b_j(v) [ \wt t_j(u, v) + \wt t_j(u, v) - b_j(v)] \\
   & \ge & b_j(v) \wt t_j(u, v) \, = \, \frac{b_j(v)}{\wt t_j(u, v)} \cdot [ \wt t_j(u, v)]^2. %\\
   %%%& \ge & b_j(v) \cdot \mu^b_j \cdot  [ \wt t_j(u, v)]^2.
\end{eqnarray*}
This gives rise to statement (i). Statement (ii) follows from the similar argument.
%
%Further, by noting that  $2 a_j(v) \wt t_j(u, v) - a^2_j(v) \ge 0$ and using a similar argument, we obtain $ %2 a_j(v) \wt t_j(u, v) - a^2_j(v) \ge \mu^a_j \cdot | a_j(v)| \cdot \big[ \wt t_j(u, v) \big]^2$ for any $(u, %v) \in \Wcal^a_j$.
%
\end{proof}

%%%%A technical lemma is as follows.

%\begin{proof}
%  Let $\gamma:=\min_i y_i$. Since each $x_i \ge 0$ and $y_i \ge \gamma$, we have $x_i \cdot y_i \ge x_i \cdot %\gamma$ for each $i$. Therefore, $\max_i (x_i y_i) \ge \max_i (x_i \gamma) = (\max_i x_i) \cdot \gamma$, %where the last equality is due to the fact that $\gamma \ge 0$.
%\end{proof}
%

\begin{lemma} \label{lem:lower_bd_f*}
  Let $\Pcal$ be an irreducible, closed, convex and CP admissible set whose conic hull is given by $\mbox{cone}(\Pcal)=\mathbb R_{\Lcal_1} \times (\mathbb R^+)_{\Lcal_+}\times (\mathbb R^-)_{\Lcal_-}$, and the matrix $A \in \mathbb R^{m\times N}$ with unit columns. Given an index set $\Jcal$ and vectors $u, v \in \Pcal$ with $\supp(v) \subseteq \Jcal \subset \supp(u)$, the following hold: (i) $f^*_j(u, v) \ge \|A(u-v)\|^2_2 -\wt t^2_j(u, v)$ for any $j \in \Lcal_1$; (ii)
  $f^*_j(u, v) \ge \| A(u- v)\|^2_2 - ([\wt t_j(u, v)]_+)^2$ for any $j \in \Lcal_+$; and (iii) $f^*_j(u, v) \ge \| A(u-v)\|^2_2 - ([\wt t_j(u, v)]_-)^2$ for any $j \in \Lcal_-$.
\end{lemma}

\begin{proof}
 Statement (i) follows from the comment given below (\ref{eqn:f*_j_convex}). To show statement (ii), we see from the relation of the index sets $\Lcal_1, \Lcal_+, \Lcal_-$ for $\mbox{cone}(\Pcal)$ given in (\ref{eqn:conic_hll_indices}) in Proposition~\ref{prop:conic_hull} that  $a_j(v)=0<b_j(v)$ for any $j \in \Lcal_+$, and $a_j(v)<0=b_j(v)$ for all $j \in \Lcal_-$. Hence, for each $j\in \Lcal_+$,  we deduce in view of (\ref{eqn:f*_j_convex}) that  either $f^*_j(u, v) = \| A(u- v)\|^2_2 - ([\wt t_j(u, v)]_+)^2$ when $\wt t_j(u,v) \le b_j(v)$ or $f^*_j(u, v) = \| A(u- v)\|^2_2 - [2 b_j(v) \wt t_j(u, v) - b^2_j(v)] \ge \| A(u- v)\|^2_2- \wt t^2_j(u, v)$ when $\wt t_j(u,v) \ge b_j(v)$. In the latter case, since $b_j(v)>0$, we have $0<\wt t_j(u,v)= [\wt t_j(u, v)]_+$. Therefore, $f^*_j(u, v) \ge \| A(u- v)\|^2_2 - ([\wt t_j(u, v)]_+)^2$ for both cases. This yields statement (ii). Statement (iii) follows from the similar argument and the fact that $a_j(v)<0=b_j(v)$ for all $j \in \Lcal_-$.
\end{proof}

The next result is trivial; its proof is omitted.

\begin{lemma} \label{lem:max_product}
  Let $x_i \in \mathbb R_+$ and $y_i \in \mathbb R_{+}$ for each $i=1, \ldots, n$. Then $\max_i  ( x_i \cdot y_i ) \ge \big(\max_i x_i \big) \cdot \big(\min_i y_i \big)$.
\end{lemma}

%With the help of these technical results,   we present the following result.

The following theorem gives a sufficient condition for condition $(\mathbf H)$ on $\Pcal$, and thus for the exact support recovery on $\Pcal$, in terms of the constants $\theta_{K, \Pcal}$ and $\theta_{K, \Pcal}$ introduced in Definition~\ref{def:RIP_MC_constants}.

%%%the proof of Proposition~\ref{prop:sufficency_general_case} as follows.

\begin{theorem} \label{thm:sufficency_general_case}
\tblue{
  Let $\Pcal$ be an irreducible, closed, convex and CP admissible set whose conic hull is given by $\mbox{cone}(\Pcal)=\mathbb R_{\Lcal_1} \times (\mathbb R^+)_{\Lcal_+}\times (\mathbb R^-)_{\Lcal_-}$, and the matrix $A \in \mathbb R^{m\times N}$ with unit columns.
  Suppose  there exist  constants $\delta_{K,\Pcal}$ of Property RI and $\theta_{K, \Pcal}$ of Property RO corresponding to $\Lcal_1, \Lcal_+, \Lcal_-$ on $\Pcal$ such that for any $u, v \in \Pcal$ with $\supp(v) \subseteq \Jcal \subset \supp(u)$,  the following conditions hold:
  \begin{itemize}
    \item [C1.]  $\Lcal_0(u, v)$ is the empty set,
  $u_j > 0$ for all $j \in \Lcal^0_a(u, v)$,  $u_j < 0$ for all $j \in \Lcal^0_b(u, v)$,
  and
  \item [C2.]
  \begin{equation} \label{eqn:sufficient_condition_general}
   \big(1-\delta_{K, \Pcal} \big) \cdot \min \Big( \min_{j \in \Lcal^-_a(u, v)} \sqrt{a_j(v)/\wt t_j(u, v)}, \ 1, \ \min_{j \in \Lcal^+_b(u, v)} \sqrt{b_j(v)/\wt t_j(u,v)} \, \Big) \, > \, \sqrt{K} \cdot \theta_{K, \Pcal}.
  \end{equation}
\end{itemize}
Then condition $(\mathbf H)$ holds on $\Pcal$.
}
\end{theorem}

\begin{remark} \rm
We give a few comments on the conditions in the above theorem before presenting its proof. First, if any index set stated in Theorem~\ref{thm:sufficency_general_case} is empty, then its associated condition is vacuously true and can be ignored. Second, by the definitions of $\Lcal^-_a(u, v)$ and $\Lcal^+_b(u, v)$, we see that $0<a_j(v)/\wt t_j(u, v) <1$ for any $j \in \Lcal^-_a(u, v)$ and $0<b_j(v)/\wt t_j(u,v) <1$  for any $j \in \Lcal^+_b(u, v)$. Hence, if $\Lcal^-_a(u, v) \cup \Lcal^+_b(u, v)$ is nonempty, then
$
 0 <  \min \Big( \min_{j \in \Lcal^-_a(u, v)} \sqrt{a_j(v)/\wt t_j(u, v)}, \ 1, \ \min_{j \in \Lcal^+_b(u, v)} \sqrt{b_j(v)/\wt t_j(u,v)} \, \Big) < 1.
$
Otherwise, the minimum is one. Lastly, for any irreducible, closed, convex and CP admissible set $\Pcal$ and any constant $\varepsilon>0$, there exists an irreducible, closed, convex and CP admissible set $\wt\Pcal$ with $\Pcal\subseteq \wt \Pcal \subseteq \Pcal + \mathcal B(0, \varepsilon)$ such that
  for any $u, v \in \wt\Pcal$ with $\supp(v) \subseteq \Jcal \subset \supp(u)$, the conditions in C1 hold, i.e., $\Lcal_0(u, v)$ is empty,
  $u_j > 0$ for all $j \in \Lcal^0_a(u, v)$,  $u_j < 0$ for all $j \in \Lcal^0_b(u, v)$, where $\mathcal B(0, \varepsilon)$ is the open $\varepsilon$-ball at the origin.
\end{remark}

\begin{proof}%%[Proof of Proposition~\ref{prop:sufficency_general_case}]
For any $u, v \in \Sigma_K \cap \Pcal $ with $\supp(v) \subseteq \Jcal \subset \supp(u)$ for some index set $\Jcal$, we have $v_j=0$ for each $j \in \supp(u)\setminus \Jcal$. Since $\Lcal_0(u, v)$ is the empty set, the (possibly unbounded) interval $\mathbb I_j(v)=[a_j(v), b_j(v)]$ satisfies exactly one of the following conditions: (i) $a_j(v)<0 < b_j(v)$; (ii) $a_j(v)=0<b_j(v)$; and (iii) $a_j(v)<0=b_j(v)$. Hence, the index sets $\Lcal^-_a(u, v), \Lcal^0_a(u,v), \Lcal^+_b(u,v), \Lcal^0_b(u, v)$ and $\Lcal_{uc}(u, v)$ defined in (\ref{eqn:LL_index_sets}) form a  disjoint union of $\supp(u)\setminus \Jcal$.
%
%In light of these conditions, we see that the index sets $\Lcal^-_a(u, v), \Lcal^0_a(u,v), \Lcal^+_b(u,v), %\Lcal^0_b(u, v)$ and $\Lcal_{uc}(u, v)$ defined in (\ref{eqn:LL_index_sets}) are disjoint and their union %equals $\supp(u)\setminus \Jcal$.
%
Moreover, it follows from the expression for $f^*_j(u,v)$ given in (\ref{eqn:f*_j_convex}), the assumption that $\|A_{\bullet j} \|_2=1, \forall \, j$, and the similar argument for Lemma~\ref{lem:lower_bd_f*} that
\begin{eqnarray*}
  j \in \Lcal^-_a(u,v) & \Rightarrow &  f^*_j(u,v) = \|A(u- v)\|^2_2 -  [ 2 a_j(v) \wt t_j(u, v) - a^2_j(v)],  \\
  j \in \Lcal^0_a(u,v) & \Rightarrow &  f^*_j(u,v) = \|A(u- v)\|^2_2 -  \big([\wt t_j(u, v)]_+ \big)^2, \\
  j \in \Lcal^+_b(u,v) & \Rightarrow &  f^*_j(u,v) = \|A(u- v)\|^2_2 -  [ 2 b_j(v) \wt t_j(u, v) - b^2_j(v)], \\
  j \in \Lcal^0_b(u,v) & \Rightarrow &  f^*_j(u,v) = \|A(u- v)\|^2_2 -   \big([\wt t_j(u, v)]_- \big)^2, \\
  j \in \Lcal_{uc}(u,v) & \Rightarrow &  f^*_j(u,v) = \|A(u- v)\|^2_2 -  [\wt t_j(u, v)]^2.
 %%%% j \in \Lcal_0(u,v) & \Rightarrow &  f^*_j(u,v) = \|A(u- v)\|^2_2.
\end{eqnarray*}
In light of (\ref{eqn:f*_j_convex}) and Lemma~\ref{lem:W_a_b_bound}, we have
\begin{eqnarray}
  \lefteqn{ \min_{j \in \supp(u- v)\setminus \Jcal} f^*(u, v) \,  = \,  \min_{j \in \supp(u)\setminus \Jcal} f^*(u, v) }  \notag \\
   & = & \min \Big( \min_{j \in \Lcal^-_a(u, v)} f^*_j(u, v), \ \min_{j \in \Lcal^0_a(u, v)} f^*_j(u, v),
   %
   %\min_{j \in \Lcal_0(u, v)} f^*_j(u, v), \
   %
   \ \min_{j \in \Lcal_{uc}(u, v)} f^*_j(u, v), \  \min_{j \in \Lcal^0_b(u, v)} f^*_j(u, v), \min_{j \in \Lcal^+_b(u, v)} f^*_j(u, v) \Big) \notag \\
      %%%& & \qquad \qquad \qquad \min_{j \in \Lcal^+_b(u, v)} f^*_j(u, v) \Big)  \notag  \\
   & = & \| A(u-v)\|^2_2 - \max \Big( \max_{j \in \Lcal^-_a(u, v)}   [ 2 a_j(v) \wt t_j(u, v) - a^2_j(v)],  \max_{j \in \Lcal^0_a(u, v)}  \big([\wt t_j(u, v)]_+ \big)^2, \max_{j\in \Lcal_{uc}(u, v)}  \wt t^2_j(u, v), \notag \\
   & &   \qquad \qquad \qquad \max_{j \in \Lcal^0_b(u, v)}  \big([\wt t_j(u, v)]_- \big)^2, \max_{j \in \Lcal^+_b(u, v)}  [ 2 b_j(v) \wt t_j(u, v) - b^2_j(v)] \Big) \notag \\
  & \le &   \| A(u-v)\|^2_2 - \max \Big( \max_{j \in \Lcal^-_a(u, v)}  [a_j(v)/\wt t_j(u, v)] \cdot [\wt t_j(u, v)]^2,  \max_{j \in \Lcal^0_a(u, v)}  \big([\wt t_j(u, v)]_+ \big)^2,  \notag \\
   & &  \qquad \qquad\qquad \quad \quad \max_{j\in \Lcal_{uc}(u, v)}  \wt t^2_j(u, v), \max_{j \in \Lcal^0_b(u, v)}   \big([\wt t_j(u, v)]_- \big)^2,  \max_{j \in \Lcal^+_b(u, v)} [b_j(v)/\wt t_j(v)] \cdot [\wt t_j(u, v)]^2 \Big) \notag \\ %%\label{eqn:upper_bound_f_j_star} \\
  %
%\end{eqnarray}
%
%\begin{eqnarray}
 %
 % \lefteqn{ \min_{j \in \supp(u- v)\setminus \Jcal} f^*(u, v) \,  = \,  \min_{j \in \supp(u)\setminus \Jcal} %f^*(u, v) }  \notag \\
 %
  % & \le &   \| A(u-v)\|^2_2 - \max \Big( \max_{j \in \Lcal^-_a(u, v)}  [a_j(v)/\wt t_j(u, v)] \cdot [\wt %t_j(u, v)]^2,  \max_{j \in \Lcal^0_a(u, v)}  \big([\wt t_j(u, v)]_+ \big)^2,  \notag \\
  % & &  \qquad \qquad\qquad \quad \quad \max_{j\in \Lcal_{uc}(u, v)}  \wt t^2_j(u, v), \max_{j \in %\Lcal^0_b(u, v)}   \big([\wt t_j(u, v)]_- \big)^2,  \max_{j \in \Lcal^+_b(u, v)} [b_j(v)/\wt t_j(v)] \cdot %[\wt t_j(u, v)]^2 \Big) \label{eqn:upper_bound_f_j_star} \\
  % %
   & \le & \| A(u-v)\|^2_2 - \max \Big( \max_{j \in \Lcal^-_a(u, v)}  \wt t^2_j(u, v) \cdot \min_{j \in \Lcal^-_a(u, v)} [ a_j(v)/\wt t_j(u, v)],  \max_{j \in \Lcal^0_a(u, v)}  \big([\wt t_j(u, v)]_+ \big)^2,    \notag  \\
   & &  \qquad \qquad \qquad \max_{j\in \Lcal_{uc}(u, v) }  \wt t^2_j(u, v), \ \max_{j \in \Lcal^0_b(u, v)}   \big([\wt t_j(u, v)]_- \big)^2, \max_{j \in \Lcal^+_b(u, v)}  \wt t^2_j(u, v)  \cdot \min_{j \in \Lcal^+_b(u, v)} [ b_j(v)/ \wt t_j(u, v)]  \, \Big)
   \notag \\
  %
  % & & \qquad \qquad \qquad \qquad  \max_{j \in \Lcal^+_b(u, v)}  \wt t^2_j(u, v)  \cdot \min_{j \in %\Lcal^+_b(u, v)} [ b_j(v)/ \wt t_j(u, v)]  \, \Big) \notag \\
 %
 & \le & \|A(u-v)\|^2_2 -\max\Big( \max_{j\in \Lcal^0_a(u,v)} \big([\wt t_j(u, v)]_+ \big)^2, \max_{j\in \Lcal^-_a(u, v) \cup  \Lcal_{uc}(u,v) \cup \Lcal^+_b(u, v)}   \wt t^2_j(u,v), \notag \\
     & & \qquad \qquad  \max_{j \in \Lcal^0_b(u, v)} \big([\wt t_j(u, v)]_- \big)^2 \Big) \times \min\Big( \min_{j \in \Lcal^-_a(u, v)}  [a_j(v)/\wt t_j(u, v)], \ 1, \, \min_{j \in \Lcal^+_b(u, v)} \ [ b_j(v)/ \wt t_j(u, v)] \Big), \notag
\end{eqnarray}
where the second-to-last and last inequalities follow from Lemma~\ref{lem:max_product}.
%
%
%In view of the relation of the index sets $\Lcal_1, \Lcal_+, \Lcal_-$ for $\mbox{cone}(\Pcal)$ given in %(\ref{eqn:conic_hll_indices}) in Proposition~\ref{prop:conic_hull}, we see that for any $(u, v)$, (i) %$f^*_j(u, v) = A(u-v)\|^2_2 - ([\wt t_j(u, v)]_+)^2$ for any $j \in \Lcal_+$;  (ii) $f^*_j(u, v) = %A(u-v)\|^2_2 - ([\wt t_j(u, v)]_-)^2$ for any $j \in \Lcal_-$; and (iii) $f^*_j(u, v) \ge \|A(u-v)\|^2_2 -\wt %t^2_j(u, v)$ for any $j \in \Lcal_1$ as shown below (\ref{eqn:f*_j_convex}). Hence, we deduce that
%
Moreover, it follows from Lemma~\ref{lem:lower_bd_f*} that for any $u, v \in \Sigma_K \cap \Pcal $ with $\supp(v) \subseteq \Jcal \subset \supp(u)$,
\begin{align*}
 \lefteqn{\min_{j\in [\supp(u-v)]^c\setminus \Jcal} f^*_j(u, v)  \, = \, \min_{j\in [\supp(u)]^c} f^*_j(u, v) }\\
  & \ge  \| A(u-v)\|^2_2 - \max \Big( \max_{j \in [\supp(u)]^c \cap \Lcal_1}  \wt t^2_j(u, v),
  %%% & & \qquad
%
  \max_{j \in [\supp(u)]^c \cap \Lcal_+}  ([\wt t_j(u, v)]_+)^2, \max_{j \in [\supp(u)]^c \cap \Lcal_-}  ([\wt t_j(u, v)]_-)^2 \Big).
\end{align*}
Let $\wt \Lcal(u,v) := \Lcal^-_a(u, v) \cup \Lcal_{uc}(u,v) \cup \Lcal^+_b(u, v) $ for notational simplicity. Define the following quantities:
\begin{eqnarray*}
  \Gamma_1 & := &  \max\Big( \max_{j\in \Lcal^0_a(u,v)}  [\wt t_j(u, v)]_+, \max_{j\in \wt \Lcal(u,v)} |\wt t_j(u,v)|, \max_{j \in \Lcal^0_b(u, v)}  [\wt t_j(u, v)]_- \, \Big)\\
 & & \quad \times
   \min \Big( \min_{j \in \Lcal^-_a(u, v)} \sqrt{a_j(v)/\wt t_j(u, v)}, \, 1,  \, \min_{j \in \Lcal^+_b(u, v)} \sqrt{b_j(v)/\wt t_j(u, v)} \, \Big),  \\  [5pt]
 \Gamma_2 & := &  \max \Big( \max_{j \in  [\supp(u)]^c \cap \Lcal_1 }  |\wt t_j(u, v)|, \max_{j \in [\supp(u)]^c \cap \Lcal_+}  [\wt t_j(u, v)]_+, \max_{j \in [\supp(u)]^c \cap \Lcal_-}  [\wt t_j(u, v)]_-   \Big).
\end{eqnarray*}
Note that if $\Gamma_2 < \Gamma_1$, then $\min_{j \in \supp(u)\setminus \mathcal J} f^*_{j}(u, v) \, < \,  \min_{j \in [\supp(u)]^c} f^*_{j}(u, v)$ such that condition $(\mathbf H)$ given by (\ref{eqn:condition_H'}) holds. Hence, it suffices to show that $\Gamma_2 < \Gamma_1$ as follows.

%
%Therefore, by  letting $\wt \Lcal(u,v) := \Lcal^-_a(u, v) \cup \Lcal_{uc}(u,v) \cup \Lcal^+_b(u, v) $ for %notational simplicity, a sufficient condition for $(\mathbf H')$ in (\ref{eqn:condition_H'}) to hold
%%
%is that
%\begin{eqnarray*}
% \wt \Gamma_1 & := &  \max\Big( \max_{j\in \Lcal^0_a(u,v)}  [\wt t_j(u, v)]_+, \max_{j\in \wt \Lcal(u,v)} %|\wt t_j(u,v)|, \max_{j \in \Lcal^0_b(u, v)}  [\wt t_j(u, v)]_- \, \Big)\\
% & & \quad \times
%
%   \min \Big( \min_{j \in \Lcal^-_a(u, v)} \sqrt{a_j(v)/\wt t_j(u, v)}, \, 1,  \, \min_{j \in \Lcal^+_b(u, %v)} \sqrt{b_j(v)/\wt t_j(u, v)} \, \Big)  \\
% & \, > \, &  \max \Big( \max_{j \in  [\supp(u)]^c \cap \Lcal_1 }  |\wt t_j(u, v)|, \max_{j \in [\supp(u)]^c %\cap \Lcal_+}  [\wt t_j(u, v)]_+, \max_{j \in [\supp(u)]^c \cap \Lcal_-}  [\wt t_j(u, v)]_-   \Big) \, := \, %\wt \Gamma_2. %%%\qquad \qquad \qquad \qquad \qquad \qquad
%\end{eqnarray*}
%

By virtue of the definition of the constant $\theta_{K, \Pcal}$ corresponding to $\Lcal_1, \Lcal_+$ and $\Lcal_-$, we deduce  that $\Gamma_2 \le \theta_{K, \Pcal} \cdot \| u - v \|_2$.
%
%
%  Besides, in view of Proposition~\ref{prop:negative_second_term} and the definition of $\Gamma_1$ in %(\ref{eqn:cone_equivalent_H'_condition}), we have
%
%
%To show that the above sufficient condition holds under (\ref{eqn:sufficient_condition_general}), we see via %the definition of $\theta_{K, \Pcal}$ that
%\begin{equation} \label{eqn:upper_bd_wG2}
%  \wt \Gamma_2 \, \le \, \theta_{K, \Pcal}  \cdot \| u - v\|_2.
%\end{equation}
%
Furthermore, $\supp(u)\setminus \Jcal = \wt \Lcal(u,v) \cup \Lcal^0_a(u,v) \cup \Lcal^0_b(u,v)$ since $\Lcal_0(u, v) = \emptyset$. Hence, in view of Proposition~\ref{prop:negative_second_term}, we deduce that
for any optimal solution $v$ to the minimization problem $\min_{z \in \mathcal P, \ \supp(z)\subseteq \mathcal J} \| A (u - z) \|^2_2$, the following holds:
%
%it follows from Proposition~\ref{prop:negative_second_term} and the similar argument for %Theorem~\ref{thm:sufficency_cone_case} that
%
\begin{eqnarray*}
    \lefteqn{ (1-\delta_{K, \Pcal}) \cdot \| u -v \|^2_2 \, \le \,  \| A (u-v)\|^2_2 \, \le \, \sum_{j \in \supp(u)\setminus\Jcal } \langle A(u-v), A_{\bullet j}  \rangle \cdot (u-v)_j } \\
    & = & \sum_{j \in  \wt \Lcal(u,v) } \langle A(u-v), A_{\bullet j}  \rangle \cdot (u-v)_j + \sum_{j \in \Lcal^0_a(u,v) } \langle A(u-v), A_{\bullet j}  \rangle \cdot (u-v)_j     \\ [5pt]
    & &  \qquad \quad + \sum_{j \in \Lcal^0_b(u,v) } \langle A(u-v), A_{\bullet j}  \rangle \cdot (u-v)_j \\
    & \le & \sum_{j \in \wt \Lcal(u,v) } |\langle A(u-v), A_{\bullet j} \rangle|  \cdot  | (u-v)_j| + \sum_{j \in \Lcal^0_a(u,v) } \langle A(u-v), A_{\bullet j} \rangle_+ \cdot  |(u-v)_j| \\
  & &  \qquad \quad + \sum_{j \in  \Lcal^0_b(u,v) } \langle A(u-v), A_{\bullet j} \rangle_- \cdot  |(u-v)_j| \\
   %
   % & \le & \sum_{j \in \supp(u)\setminus\Jcal } \frac{|\langle A(u-v), A_{\bullet j} \rangle|} {\| %A_{\bullet j}\|_2}  \cdot \|A_{\bullet j} \|_2 \cdot  | (u-v)_j| \\
   %
   %
    &  \le &
     \max\Big(  \max_{j\in \wt \Lcal(u,v)} |\wt t_j(u,v)|, \max_{j\in \Lcal^0_a(u,v)}  [\wt t_j(u, v)]_+, \max_{j \in \Lcal^0_b(u, v)}   [\wt t_j(u, v)]_- \, \Big) \cdot \|(u-v)_{\supp(u)\setminus \Jcal} \|_1 \\
    &  \le &
     \max\Big(  \max_{j\in \wt \Lcal(u,v)} |\wt t_j(u,v)|, \max_{j\in \Lcal^0_a(u,v)}  [\wt t_j(u, v)]_+, \max_{j \in \Lcal^0_b(u, v)}   [\wt t_j(u, v)]_- \, \Big) \cdot \sqrt{|\supp(u)\setminus \Jcal|} \cdot \|u-v \|_2 \\
    & \le & \max\Big(  \max_{j\in \wt \Lcal(u,v)} |\wt t_j(u,v)|, \max_{j\in \Lcal^0_a(u,v)}  [\wt t_j(u, v)]_+, \max_{j \in \Lcal^0_b(u, v)}  [\wt t_j(u, v)]_- \, \Big)  \cdot \sqrt{K} \cdot \|u-v \|_2.
\end{eqnarray*}
where the third inequality holds because $u_j>0=v_j$ for all $j\in \Lcal^0_a(u,v)$ and $u_j<0=v_j$ for all $j\in \Lcal^0_b(u,v)$.
By the definition of $\Gamma_1$, we obtain
\[
     (1-\delta_{K, \Pcal}) \cdot \| u -v \|^2_2 \cdot \min \Big( \min_{j \in \Lcal^-_a(u, v)} \sqrt{a_j(v)/\wt t_j(u, v)}, \ 1,  \, \min_{j \in \Lcal^+_b(u, v)} \sqrt{b_j(v)/\wt t_j(u, v)} \, \Big) \le \Gamma_1 \cdot \sqrt{K} \cdot \|u - v\|_2.
\]
Since $\supp(v) \subset \supp(u)$, we have $\| u-v \|_2>0$. This further implies that
\[
   [(1- \delta_{K,\Pcal})/\sqrt{K}] \cdot \min \Big( \min_{j \in \Lcal^-_a(u, v)} \sqrt{a_j(v)/\wt t_j(u, v)}, \ 1,  \, \min_{j \in \Lcal^+_b(u, v)} \sqrt{b_j(v)/\wt t_j(u, v)} \, \Big) \cdot \| u-v\|_2 \le \Gamma_1.
\]
Using  $\Gamma_2 \le \theta_{K, \Pcal} \| u - v \|_2 $ and the assumption in (\ref{eqn:sufficient_condition_general}), we have $\Gamma_2 < \Gamma_1$ so that condition $(\mathbf H)$ holds.
%
%In view of (\ref{eqn:sufficient_condition_general}) and (\ref{eqn:upper_bd_wG2}), we conclude that the %desired result holds.
%
\end{proof}

\mycut{
Based on the definitions of the index sets $\Lcal^-_a(u, v)$ and $\Lcal^+_b(u, v)$, we see that $0<a_j(v)/\wt t_j(u, v) <1$ for each $j \in \Lcal^-_a(u, v)$ and $0<b_j(v)/\wt t_j(u,v) <1$  for each $j \in \Lcal^+_b(u, v)$. Hence, if $\Lcal^-_a(u, v) \cup \Lcal^+_b(u, v)$ is nonempty, then
\[
 0 <  \min \Big( \min_{j \in \Lcal^-_a(u, v)} \sqrt{a_j(v)/\wt t_j(u, v)}, \ 1, \ \min_{j \in \Lcal^+_b(u, v)} \sqrt{b_j(v)/\wt t_j(u,v)} \, \Big) < 1.
\]
}

%%%\newpage

%---------------------------------------------------------------------------
%
\subsection{Cone Case} \label{subsect:suff_cond_cone}

The uniform recovery conditions developed in Theorem~\ref{thm:sufficency_general_case} can be simplified for specific convex CP admissible sets. To illustrate it, consider an irreducible, closed, convex and CP admissible cone. By Proposition~\ref{prop:CP_admissible_cone},
$\Pcal=  \mathbb R_{\Ical_1} \times (\mathbb R_+)_{\Ical_+} \times (\mathbb R_-)_{\Ical_-}$, where $\Ical_1, \Ical_+$ and $\Ical_-$ form a disjoint union of $\{1, \ldots, N\}$.
The following corollary gives a simpler sufficient condition for condition $(\mathbf H)$ on $\Pcal$ in terms of $\theta_{K, \Pcal}$ and $\theta_{K, \Pcal}$. This result recovers the similar condition given in \cite{MoS_TIT12} for $\Pcal=\mathbb R^N$.

%
%; see Definition~\ref{def:irreducible_CP_set} for the irreducibility. It follows from %Proposition~\ref{prop:CP_admissible_cone}
%
%that $\Pcal$ is a Cartesian product of copies of $\mathbb R, \mathbb R_+$ and $\mathbb R_-$, i.e., $\Pcal=  %\mathbb R_{\Ical_1} \times (\mathbb R_+)_{\Ical_+} \times (\mathbb R_-)_{\Ical_-}$, where $\Ical_1, \Ical_+$ %and $\Ical_-$ form a disjoint union of $\{1, \ldots, N\}$.
%
%Without loss of generality, we assume that $A$ has unit columns i.e., $\| A_{\bullet i}\|_2=1$ for all $i$.
%

%
%To obtain sufficient conditions for exact support recovery, we introduce two positive constants:
%\begin{itemize}
% \item [(1)] The constant $\bar \delta_k>0$ such that $(1-\bar\delta_k) \| u - v \|^2_2 \le \| A(u-v) \|^2_2 %$ for all $u, v\in \sum_k \cap \Pcal$ with $\supp(v) \subset \supp(u)$, and
% \item [(2)]
%      The constant $\bar \theta_k>0$ such that for all $u, v\in \sum_k \cap \Pcal$ with $\supp(v) \subset %\supp(u)$, $\langle A(u-v), A_{\bullet j} \rangle_+ \le \bar\theta_k \cdot \|u-v\|_2$  for $j\in %[\supp(u)]^c\cap \Ical_+$ and $|\langle A(u-v), A_{\bullet j} \rangle| \le \bar\theta_k \cdot \|u-v\|_2$  for %$j\in [\supp(u)]^c \cap \Ical_1$.
%\end{itemize}
%

\begin{corollary} \label{thm:sufficency_cone_case}
  Let $\Pcal=  \mathbb R_{\Ical_1} \times (\mathbb R_+)_{\Ical_+} \times (\mathbb R_-)_{\Ical_-}$, where the index sets $\Ical_1, \Ical_+$ and $\Ical_-$ form a disjoint union of $\{1, \ldots, N\}$, and let $A \in \mathbb R^{m\times N}$ be a matrix with unit columns. Suppose there exist constants $\delta_{K, \Pcal}$ of Property RI on $\Pcal$ and $\theta_{K, \Pcal}$ of Property RO on $\Pcal$ corresponding to $\Ical_1, \Ical_+$ and $\Ical_-$ such that $1- \delta_{K, \Pcal} \, > \, \sqrt{K} \cdot \theta_{K, \Pcal}$. Then condition $(\mathbf H)$ holds on $\Pcal$.
%
%  \begin{equation} \label{eqn:sufficient_condition_cone}
%    1- \delta_{K, \Pcal} \, > \, \sqrt{K} \cdot \theta_{K, \Pcal}.
%  \end{equation}
%Then condition $(\mathbf H)$ holds on $\Pcal$.
%
\end{corollary}

\begin{proof}
\tblue{
 For a given index set $\Jcal$ and any $u, v \in \Pcal$ with $\supp(v) \subseteq \Jcal \subset \supp(u)$, either
 $[a_j(v)=-\infty, b_j(v)=0]$ or $[a_j(v)=0, b_j(v)=+\infty]$ or $[a_j(v)=-\infty, b_j(v)=+\infty]$ for each $j \in \supp(u)\setminus \Jcal$. Hence, $u_j > 0$ for all $j \in \Lcal^0_a(u, v)$,  $u_j < 0$ for all $j \in \Lcal^0_b(u, v)$, and $\Lcal_0(u, v)$ is always the empty set. Therefore, the conditions in C1 hold.
 Moreover, $\Lcal^-_a(u, v)$ is also empty since $a_j(v)=-\infty$ if $a_j(v)<0$. Similarly, $\Lcal^+_b(u, v)$ is empty. Hence, condition~(\ref{eqn:sufficient_condition_general}) reduces to $1- \delta_{K, \Pcal} \, > \, \sqrt{K} \cdot \theta_{K, \Pcal}$.
}
\end{proof}

%%{\bf Special cases}: $\Pcal=\mathbb R^N$ and $\Pcal=\mathbb R^N_+$.
%
%\gap
%

Since $\delta_{K, \Pcal}$ and $\theta_{K, \Pcal}$ may be difficult to obtain numerically due to the conditions such as $\supp(v) \subset \supp(u)$ in their definitions, it is desired that similar constants independent of the above mentioned conditions can be used. This leads to the following quantities.

% for a matrix $A \in \mathbb R^{m\times N}$ with unit columns and the index sets $\Ical_1, \Ical_+$ and %$\Ical_-$ which form a disjoint union of $\{1, \ldots, N\}$:

\begin{definition} \label{def:delta_theta_hat}
 Let a matrix $A \in \mathbb R^{m\times N}$ with unit columns and the index sets $\Ical_1, \Ical_+$ and $\Ical_-$ which form a disjoint union of $\{1, \ldots, N\}$ be given.
\begin{itemize}
 \item [(i)] The constant $\wh \delta_K \in (0, 1)$ is such that $(1-\wh\delta_K) \cdot \| x \|^2_2 \le \| Ax \|^2_2 $ for all $x \in \Sigma_K$;
 \item [(ii)]
      The constant $\wh\theta_K>0$ corresponding to the index set $\Ical_1, \Ical_+$ and $\Ical_-$ is such that for any $x \in \Sigma_K$,
      \[
         \max \Big( \, \max_{ j \in \Ical_1} |\langle Ax, A_{\bullet j} \rangle|, \ \max_{ j \in \Ical_+} \langle Ax, A_{\bullet j} \rangle_+, \ \max_{ j \in \Ical_-} \langle Ax, A_{\bullet j} \rangle_- \, \Big) \, \le \, \wh\theta_K \cdot \|x\|_2.
      \]
\end{itemize}
\end{definition}
To emphasize the dependence of the above constants on $A$ (when $\Ical_1, \Ical_+$ and $\Ical_-$ are fixed), we also write them as $\wh \delta_K(A)$ and $\wh \theta_K(A)$, respectively.
Based on Definition~\ref{def:delta_theta_hat}, it is easy to see that  $\wh \delta_K$ is of Property RI and $\wh \theta_K$ is of Property RO, both on $\Pcal$. Hence, by Corollary~\ref{thm:sufficency_cone_case}, we obtain the following result immediately; its proof is omitted.

\begin{corollary} \label{coro:sufficiency_uniform_constants}
 For a given matrix $A \in \mathbb R^{m\times N}$ with unit columns and a closed, convex, and CP admissible cone $\Pcal$ defined by the index sets $\Ical_1, \Ical_+$ and $\Ical_-$, if there exist positive constants $\wh\delta_K$ and $\wh\theta_K$ given by Definition~\ref{def:delta_theta_hat} such that $ 1-\wh\delta_K> \sqrt{K} \cdot \wh\theta_K$, then  condition  $(\mathbf H)$ given by (\ref{eqn:condition_H'}) holds.
\end{corollary}

In what follows, we discuss the constants $\wh\delta_K$ and $\wh \theta_K$ subject to perturbations of $A$.
\begin{proposition} \label{prop:A_perturbation}
  Let a matrix $A^\diamond \in \mathbb R^{m\times N}$  be such that there exist constants $\wh\delta_K(A^\diamond) \in (0, 1)$ and $\wh\theta_{K}(A^\diamond)>0$ satisfying
  $1- \wh\delta_K(A^\diamond)> \sqrt{K} \cdot \wh\theta_{K}(A^\diamond)$. Then there exists a constant $\eta>0$ such that for any $A $ with $\| A - A^\diamond \|_2 < \eta$, there exist constants $\wh\delta_K(A)>0$ and $\wh\theta_{K}(A)>0$ satisfying the conditions given by Definition~\ref{def:delta_theta_hat} such that $1- \wh\delta_K(A)> \sqrt{K} \cdot \wh\theta_{K}(A)$.
\end{proposition}

\begin{proof}
 For the given matrix $A^\diamond$ and the positive constants $\wh\delta_K(A^\diamond)$ and $\wh\theta_{K}(A^\diamond)$, it suffices to show that for any $\varepsilon>0$, there exist constants $\eta'>0$ and $\eta''>0$ such that (i) for each $A$ with $\| A - A^\diamond\|_2 < \eta'$, there exists a constant $\wh\delta_K(A)>0$ satisfying condition (i) of Definition~\ref{def:delta_theta_hat}  such that $|\wh\delta_K(A)- \wh\delta_K(A^\diamond)| < \varepsilon$; and (ii) for each $A$ with $\| A - A^\diamond\|_2 < \eta''$, there exists a constant $\wh\theta_K(A)>0$ satisfying condition (ii) of Definition~\ref{def:delta_theta_hat} such that $|\wh\theta_K(A)- \wh\theta_K(A^\diamond)| < \varepsilon$.
%
% $|\delta'_k(A)- \delta'_k(A^\diamond)| < \varepsilon$ for all $A \in \{ A \in \mathbb R^{m\times N} \, | \, %\| A - A^\diamond\|_2 < \eta_1\}$ and $|\theta'_k(A)- \theta'_k(A^\diamond)| < \varepsilon$ for all $A \in \{ %A \in \mathbb R^{m\times N} \, | \, \| A - A^\diamond\|_2 < \eta_2\}$.
%

 To show the existence of $\eta'$, we use the inequality $\big| \|Ax\|_2 - \| A^\diamond x\|_2 \big| \le \| A - A^\diamond \|_2 \cdot \|x \|_2$  for any $A$ and $x$ \cite[Proposition 5.3]{ShenMousavi_SIOPT18}. Hence, for all $A$ in the neighborhood $\mathcal U$ of $A^\diamond$ given by $\mathcal U=\{ A \, | \, \|A - A^\diamond\|_2<\alpha\}$ for some $\alpha>0$, we have $\big| \|A x \|^2_2 - \| A^\diamond x \|^2_2 \big | = \big| \|A x \|_2 - \| A^\diamond x \|_2 \big | \cdot ( \|A x \|_2 + \| A^\diamond x \|_2) \le \| A -  A^\diamond\|_2 \cdot \|x\|_2 \cdot (2\| A^\diamond\|_2 + \alpha) \cdot \| x\|_2 \le c' \cdot \| A -  A^\diamond\|_2 \cdot \| x \|^2_2$ for all $x$, where $c' :=2\| A^\diamond\|_2 + \alpha>0$. Hence, $\| A x \|^2_2 \ge \| A^\diamond x \|^2_2 - c' \cdot \| A -  A^\diamond\|_2 \cdot \| x \|^2_2 \ge [1-\wh\delta_K(A^\diamond) - c'\cdot \| A -  A^\diamond\|_2] \cdot \| x \|^2_2$ for all $x$. Letting $\wh \delta_K(A):= \wh \delta_K(A^\diamond) + c' \cdot  \| A - A^\diamond \|_2$, we can obtain a positive constant $\eta'$ with $0<\eta'<\min(\varepsilon/c', \alpha)$ such that for each $A$ with $\| A - A^\diamond\|_2 < \eta'$, $|\wh\delta_K(A)- \wh\delta_K(A^\diamond)| < \varepsilon$.

  To show the existence of $\eta''$, define the function $h_j$ for a fixed index $j$ and a matrix $A$:
 \[
     h_j(A, x) \, := \, \left\{ \begin{array}{lll} |\langle Ax, A_{\bullet j} \rangle|, & \mbox{ if } \ j \in \Ical_1; \\  \langle Ax, A_{\bullet j} \rangle_+,  & \mbox{ if } \ j \in \Ical_+; \\
      \langle Ax, A_{\bullet j} \rangle_-,  & \mbox{ if } \ j \in \Ical_-. \end{array} \right.
 \]
Using the fact that $|x_+- y_+| \le |x -y |$ and $|x_- - y_-| \le |x - y|$ for any $x, y \in \mathbb R$, we have, for each $j$,
\begin{eqnarray*}
  | h_j(A, x) - h_j(A^\diamond, x)| & \le & |\langle Ax, A_{\bullet j} \rangle -\langle A^\diamond x, A^\diamond_{\bullet j} \rangle | \\
  & = & \Big| \langle A^\diamond x, (A-A^\diamond)_{\bullet j} \rangle + \langle (A-A^\diamond) x, A^\diamond_{\bullet j} \rangle + \langle (A-A^\diamond) x, (A-A^\diamond)_{\bullet j} \rangle  \Big| \\
  & \le & | \langle A^\diamond x, (A-A^\diamond) \mathbf e_{j} \rangle| + |\langle (A-A^\diamond) x, A^\diamond \mathbf e_{j} \rangle| + | \langle (A-A^\diamond) x, (A-A^\diamond) \mathbf e_{j} \rangle | \\
  & \le & \| A - A^\diamond \|_2 \cdot [ 2 \|A^\diamond\|_2 + \| A - A^\diamond \|_2 \big] \cdot \| x \|_2,
\end{eqnarray*}
where the last inequality follows from Cauchy-Schwarz inequality and $\| \mathbf e_j \|_2=1$. Therefore, for all $A$ in the neighborhood $\mathcal U$ of $A^\diamond$ given by $\mathcal U=\{ A \, | \, \|A - A^\diamond\|_2<\beta\}$ for some $\beta>0$, we obtain the constant $c:= 2 \|A^\diamond\|_2 +\beta>0$ such that for each $j$, $h_j(A, x) \le h_j(A^\diamond, x) + c \cdot \| A - A^\diamond \|_2 \cdot \| x \|_2$.  In view of
\[
\max_j h_j(A, x) \, = \, \max \Big( \, \max_{ j \in \Ical_1} |\langle Ax, A_{\bullet j} \rangle|, \ \max_{ j \in \Ical_+} \langle Ax, A_{\bullet j} \rangle_+, \ \max_{ j \in \Ical_-} \langle Ax, A_{\bullet j} \rangle_- \, \Big),
\]
 we further have
\begin{eqnarray*}
  \max_j h_j(A, x) & \le & \max_j h_j(A^\diamond, x) +  c \cdot  \| A - A^\diamond \|_2  \cdot \| x \|_2 \, \le \, \wh \theta_K  (A^\diamond) \cdot \| x \|_2 +  c \cdot  \| A - A^\diamond \|_2  \cdot \| x \|_2 \\
  & \le & \big[ \wh \theta_K(A^\diamond) + c \cdot  \| A - A^\diamond \|_2  \big] \cdot \| x \|_2.
\end{eqnarray*}
By letting $\wh \theta_K(A):= \wh \theta_K(A^\diamond) + c \cdot  \| A - A^\diamond \|_2$, it is easy to obtain a positive constant $\eta''$ with $0<\eta''<\min(\varepsilon/c, \beta)$ such that for each $A$ with $\| A - A^\diamond\|_2 < \eta''$, $|\wh\theta_K(A)- \wh\theta_K(A^\diamond)| < \varepsilon$.
%
%
% $h_j(x):=\max\big( |\langle Ax, A_{\bullet j} \rangle|, \langle Ax, A_{\bullet j} \rangle_+, \langle Ax, %A_{\bullet j} \rangle_-)$. Hence,
%
% Hence, for a fixed index $j$, the function $h_j(x):=\max\big( |\langle Ax, A_{\bullet j} \rangle|, \langle %Ax, A_{\bullet j} \rangle_+, \langle Ax, A_{\bullet j} \rangle_-)$. Define the function $h_j: \mathbb R^N %\rightarrow \mathbb R$ for a fixed index $j$ and a matrix $A$:  $h_j(x):=\max( \langle Ax, A_{\bullet j} %\rangle_+, \langle Ax, A_{\bullet j} \rangle_-) \le | \langle Ax, A_{\bullet j} \rangle|$.
%
\end{proof}

\begin{remark} \rm \label{remark:A_perturb}
 The above proposition shows that for fixed index sets $\Ical_1, \Ical_+$ and $\Ical_-$, $\mathcal A := \{ A \in \mathbb R^{m\times N} \, | \, 1- \wh\delta_K(A)> \sqrt{K} \cdot \wh\theta_{K}(A) \}$ is an open set in the matrix space $\mathbb R^{m \times N}$.
%
% an arbitrarily small perturbation of $A$ will maintain the sufficient condition $1- \wh\delta_k(A)> \sqrt{k} %\cdot \wh\theta_{k}(A)$.
%
 Since the set of  matrices of completely full rank, i.e., $A \in \mathbb R^{m\times N}$ is such that every $m\times m$ submatrix of $A$ is invertible \cite{ShenMousavi_SIOPT18}, is open and dense in the matrix space $\mathbb R^{m\times N}$, we conclude that for any $A \in \mathcal A$ and an arbitrarily small $\varepsilon>0$, there exists a matrix $A' \in \mathcal A$ of complete full rank such that $\| A' - A \|< \varepsilon$. An advantage of using the matrix $A'$ is that it leads to a unique $x^k$ in each step (cf. Lemma~\ref{lemma:sol_existence}) and thus gives rise to the exact vector recovery,  provided that the sparsity level $K \le m$.
%
% we can always assume without loss of generality that the matrix $A$ satisfying the sufficient condition %given in Corollary~\ref{coro:sufficiency_uniform_constants} is of completely full rank. A benefit of using %such a matrix is that it leads to a unique optimal solution in each step of the constrained matching pursuit %algorithm provided that the sparsity level $K \le m$.
%
\end{remark}

%
%\gap
%
%\gap
%
%\gap

%----------------------------------------------------------------------
%
%%%\subsection{Non-cone Case} \label{subsect:suff_cond_noncone}

%------------------------------------------------------------------------
%
\section{Conclusions} \label{sect:conclusion}

This paper studies the exact support and vector recovery on a constraint set  via constrained matching pursuit. We show the exact recovery critically relies on a constraint set, and introduce the class of CP admissible sets. Rich properties of these sets are exploited, and various exact recovery conditions are developed for convex CP admissible cones or sets. Future research includes the exact recovery of constrained sparse vectors subject to noise and errors via constrained matching pursuit.
%
% as well as extensions. %%to low-rank matrix recovery via constrained matching pursuit.
%

%---------------------------------------------
%
%\section*{Acknowledgments}
%
\gap
\noindent {\bf Acknowledgements.}
The authors would like to thank Dr. Joel A. Tropp for a helpful discussion on the counterexample given in Section~\ref{subsect:OMP_counterexample}.

%--------------------------------------------------------------------------------------------

%\newpage

\end{document}